\documentclass[12pt]{amsart}
\usepackage[shortlabels]{enumitem}
\usepackage{amssymb}
\usepackage{mathrsfs}
\usepackage{xcolor}
\usepackage{stmaryrd}
\usepackage[all]{xy}
\usepackage[margin=1in]{geometry} 
\usepackage{hyperref}
\usepackage{eucal}
\usepackage{contour}
\usepackage[normalem]{ulem}
\hypersetup{bookmarksdepth=2}
\hypersetup{colorlinks=true}
\hypersetup{linkcolor=blue}
\hypersetup{citecolor=blue}
\hypersetup{urlcolor=blue}

\numberwithin{equation}{section}
\newtheorem{theorem}[equation]{Theorem}

\newtheorem{proposition}[equation]{Proposition}
\newtheorem{lemma}[equation]{Lemma}
\newtheorem{corollary}[equation]{Corollary}
\newtheorem{conjecture}[equation]{Conjecture}
\newtheorem{problem}[equation]{Problem}

\theoremstyle{definition}
\newtheorem{remark}[equation]{Remark}
\newtheorem{example}[equation]{Example}
\newtheorem{definition}[equation]{Definition}

\newcommand{\cC}{\mathcal{C}}
\newcommand{\fC}{\mathfrak{C}}

\newcommand{\sE}{\mathscr{E}}
\newcommand{\bF}{\mathbf{F}}
\newcommand{\cF}{\mathcal{F}}

\newcommand{\sF}{\mathscr{F}}

\newcommand{\bK}{\mathbf{K}}

\newcommand{\rK}{\mathrm{K}}

\newcommand{\bN}{\mathbf{N}}

\newcommand{\rN}{\mathrm{N}}

\newcommand{\bO}{\mathbf{O}}
\newcommand{\cO}{\mathcal{O}}

\newcommand{\sO}{\mathscr{O}}

\newcommand{\cP}{\mathcal{P}}

\newcommand{\bQ}{\mathbf{Q}}

\newcommand{\bS}{\mathbf{S}}

\newcommand{\fS}{\mathfrak{S}}

\newcommand{\cT}{\mathcal{T}}

\newcommand{\bU}{\mathbf{U}}

\newcommand{\bV}{\mathbf{V}}

\newcommand{\fX}{\mathfrak{X}}

\newcommand{\bZ}{\mathbf{Z}}

\newcommand{\fp}{\mathfrak{p}}

\newcommand{\BB}{\mathbb{B}}
\newcommand{\GG}{\mathbb{G}}
\newcommand{\HH}{\mathbb{H}}

\newcommand{\arxiv}[1]{\href{http://arxiv.org/abs/#1}{{\tiny\tt arXiv:#1}}}
\newcommand{\DOI}[1]{\href{http://doi.org/#1}{\color{purple}{\tiny\tt DOI:#1}}}
\newcommand{\stacks}[1]{\cite[\href{http://stacks.math.columbia.edu/tag/#1}{Tag~#1}]{stacks}}
\newcommand{\defn}[1]{\emph{#1}}
\let\ol\overline
\let\ul\underline
\renewcommand{\phi}{\varphi}

\DeclareMathOperator{\avg}{avg} 
\DeclareMathOperator{\im}{im} 
\DeclareMathOperator{\End}{End}

\DeclareMathOperator{\Et}{Et}
\DeclareMathOperator{\Res}{Res}
\DeclareMathOperator{\Ind}{Ind}
\DeclareMathOperator{\Aut}{Aut}

\DeclareMathOperator{\Fun}{Fun}
\DeclareMathOperator{\Hom}{Hom}

\DeclareMathOperator{\Rep}{Rep}
\DeclareMathOperator{\Spec}{Spec}

\newcommand{\id}{\mathrm{id}}
\newcommand{\op}{\mathrm{op}}

\newcommand{\GL}{\mathbf{GL}}

\DeclareMathOperator{\lev}{lev}

\DeclareMathOperator{\Sp}{\mathbf{Sp}}
\DeclareMathOperator{\tr}{tr}
\DeclareMathOperator{\utr}{\ul{tr}}
\DeclareMathOperator{\udim}{\ul{dim}}
\DeclareMathOperator{\dcl}{dcl}

\newcommand{\bone}{\mathbf{1}}
\newcommand{\uotimes}{\mathbin{\ul{\otimes}}}

\let\lbb\llbracket
\let\rbb\rrbracket

\contourlength{1pt}

\contourlength{0.8pt}
\newcommand{\myuline}[1]{%
  \uline{\phantom{#1}}%
  \llap{\contour{white}{#1}}%
}
\DeclareMathOperator{\uRep}{\text{\myuline{\rm Rep}}}
\DeclareMathOperator{\uPerm}{\ul{Perm}}

\title{Classical interpolation categories}

\author{Nate Harman}
\address{Department of Mathematics, University of Georgia, Athens, GA, USA}
\email{\href{mailto:nharman@uga.edu}{nharman@uga.edu}}
\urladdr{\url{https://www.nateharman.com/}}
\thanks{NH was supported by NSF grant DMS-2401515}

\author{Andrew Snowden}
\address{Department of Mathematics, University of Michigan, Ann Arbor, MI, USA}
\email{\href{mailto:asnowden@umich.edu}{asnowden@umich.edu}}
\urladdr{\url{http://www-personal.umich.edu/~asnowden/}}
\thanks{AS was supported by NSF grant DMS-2301871.}

\date{\today}

\begin{document}

\begin{abstract}
We study tensor categories that interpolate the representation categories of finite classical groups. There are (at least) two ways to approach these categories: via ultraproducts and via oligomorphic groups. Both have strengths and weaknesses. The ultraproduct categories are easy to define, but their structure is not clear. On the other hand, the oligomorphic approach requires a certain kind of measure as an input, and the space of measures is not obvious. Furthermore, it is not a priori clear that the two approaches yield the same categories in general. We handle all of these issues: we determine all measures on the oligomorphic groups, and we show that the oligomorphic and ultraproduct categories agree, which gives us basic structural results about the latter. Our results rely upon (and in some sense repackage) enumerative results in finite geometry.
\end{abstract}

\maketitle
\tableofcontents

\section{Introduction}

Pre-Tannakian categories are a natural class of tensor categories generalizing representation categories of algebraic (super)groups. Two general sources of pre-Tannakian categories that are not (super) Tannakian are the method of interpolation introduced by Deligne \cite{Deligne}, and the oligomorphic theory introduced by us \cite{repst}. The purpose of this paper is to connect these two theories and establish fundamental results for the categories in the most important cases, namely, for classical groups over finite fields.

\subsection{Interpolation categories} \label{ss:intro-interp}

To explain the ideas in this paper with minimal technical overhead, we focus on a specific case, namely that of the symplectic groups; some comments on other cases are given below. Fix a finite field $\bF$ of odd cardinality $q$ in what follows.

We begin by recalling Deligne's construction. Let $H_n=\Sp_{2n}(\bF)$ be the symplectic group, with standard representation $V_n=\bF^{2n}$. Let $k_n$ be an algebraically closed field, and let $\cT_n$ be the category of finite dimensional representations of $H_n$ over $k_n$. We view the permutation module $X_n=k_n[V_n]$ as the basic object of $\cT_n$. Fix an ultrafilter $\cF$ on $\bN$, and let $k$ be the ultraproduct of the $k_n$'s, which we assume to be of characteristic~0. Let $\tilde{\cT}$ be the ultraproduct of the categories $\cT_n$. The sequence $X=(X_n)_{n \ge 0}$ defines an object of $\tilde{\cT}$, and we let $\cT$ be the tensor subcategory of $\tilde{\cT}$ generated by $X$. This is the interpolation category.

While the construction of $\cT$ is straightforward enough, the basic structure of $\cT$ is not at all obvious. To illustrate this, we discuss two conjectures about $\cT$. As far as we know, these specific conjectures have not been previously discussed, but they are modeled on known results from the case of symmetric groups. Let $t$ be the categorical dimension of $X$; this is the element of $k$ represented by the sequence $(q^{2n})_{n \ge 1}$.

\begin{conjecture} \label{conj:interp1}
The category $\cT$ (essentially) only depends on $t$, and not the fields $k_n$'s or the ultrafilter $\cF$.
\end{conjecture}

We elaborate on the meaning of this statement. Suppose $\cT'$ is constructed using some other choices $k'_n$ and $\cF'$, and let $t'$ be the dimension of $X'$. If $t$ and $t'$ are both transcendental over $\bQ$, or both algebraic with the same minimal polynomial, then there is a field isomorphism $k \to k'$ mapping $t$ to $t'$, and we conjecture that $k' \otimes_k \cT$ and $\cT'$ are equivalent tensor categories.

We note that Conjecture~\ref{conj:interp1} is far from obvious. Indeed, if each $k_n$ has characteristic~0 then $t$ is transcendental. We can also make choices of positive characteristic fields $k'_n$ such that $t'$ is transcendental. In this case, $\cT$ is constructed from an ultraproduct of semi-simple categories, while $\cT'$ is constructed from an ultraproduct of non-semi-simple categories. Nonetheless, we are predicting that $\cT$ and $\cT'$ are equivalent (after an extension of scalars).

\begin{conjecture} \label{conj:interp2}
Suppose $t \ne 0$. Then $\cT$ is pre-Tannakian and has enough projective objects. Moreover, if $t$ is not of the form $q^{2n}$ with $n \in \bN$ then $\cT$ is semi-simple.
\end{conjecture}

The semi-simplicity statement essentially means that, under the assumption on $t$, small tensor powers of the permutation modules $X_n$ are semi-simple for sufficiently many values of $n$. This is obvious when the $k_n$'s have characteristic~0, but is not at all clear in positive characteristic.

\subsection{Oligomorphic theory}

Let $\bV=\bigcup_{n \ge 1} V_n$, an infinite dimensional symplectic space, and let $G$ be its automorphism group. It is not difficult to see that $G$ has finitely many orbits on $\bV^n$ for each $n \ge 0$, that is, the action is \defn{oligomorphic}. In \cite{repst}, we developed a representation theory of oligomorphic groups that leads to new tensor categories. We now explain some of the ideas that go into this, and the concrete problems that need to be solved to apply the general theory in this specific case.

It will be important to understand the $G$-orbits on powers of $\bV$, so we take a moment to describe them. Let $V_{n,m}$ be an $\bF$-vector space of dimension $2n+m$ equipped with an alternating form whose kernel has dimension $m$; note that $V_{n,0}$ coincides with $V_n$. Let $X_{n,m}$ be the space of all embeddings $V_{n,m} \to \bV$ compatible with the forms. It is not difficult to see that $X_{n,m}$ is a transitive $G$-set, and all orbits on powers of $\bV$ have this form. Moreover, one can show that maps of $G$-sets $X_{n,m} \to X_{a,b}$ correspond to embeddings $V_{a,b} \to V_{n,m}$ compatible with the forms.

A \defn{measure} on $G$ valued in a field $k$ is a rule $\mu$ that assigns to each $G$-map $f \colon X_{n,m} \to X_{a,b}$ a number $\mu(f) \in k$ such that three axioms hold: (a) $\mu(f)=1$ if $f$ is an isomorphism; (b) $\mu(gf)=\mu(g) \cdot \mu(f)$, when this makes sense; and (c) $\mu$ is invariant under pull-backs, in a certain sense; see Definition~\ref{defn:meas2} for details. There is a universal measure valued in a certain ring $\Theta(G)$, and so determing all measures on $G$ amounts to computing this single ring. This is an important problem:

\begin{problem} \label{prob:olig1}
Determine all measures on $G$, that is, compute $\Theta(G)$.
\end{problem}

There is a related ring that will play an important role in our analysis. The (relative) \defn{Burnside ring} of $G$, denoted $\BB$, is the free $\bZ$-module spanned by classes $x_{n,m}$ (which correspond to the $G$-sets $X_{n,m}$), and in which the product $x_{n,m} x_{a,b}$ is the sum of classes corresponding to the orbit decomposition of $X_{n,m} \times X_{a,b}$. For example, computing $x_{1,0}^2$ amounts to determining the orbit decomposition of the set of all 4-tuples $(x, y, x', y') \in \bV^4$ such that $\langle x, y \rangle=\langle x', y' \rangle=1$; there are various possibilities for the orbits depending on how $(x,y)$ and $(x',y')$ interact. Thus $\BB$ encapsulates certain aspects of the enumerative combinatorics of finite symplectic spaces. Understanding its structure is a key problem:

\begin{problem} \label{prob:olig2}
Determine the structure of the Burnside ring $\BB$.
\end{problem}

Suppose $\mu$ is a measure for $G$ valued in a field $k$ of characteristic~0. In \cite{repst}, we construct a (non-abelian) tensor category $\uPerm(G, \mu)$ of ``permutation modules.'' We would like to have a pre-Tannakian abelian envelope of this category. The following is the key problem that needs to be solved to obtain such envelopes:

\begin{problem} \label{prob:olig3}
Show that nilpotent endomorphisms in $\uPerm(G, \mu)$ have trace zero.
\end{problem}

Granted a positive solution to this problem, general results of \cite{repst} provide us with an abelian envelope $\uRep(G, \mu)$. We now have two pre-Tannakian categories associated to the symplectic group: this oligomorphic one, and the ultraproduct category constructed before. There is one final problem to resolve:

\begin{problem} \label{prob:compar}
Show that the ultraproduct approach and oligomorphic approach produce the same class of pre-Tannakian categories.
\end{problem}

\subsection{Results} \label{ss:results}

The main accomplishment of this paper is a resolution of the two conjectures and four problems discussed above, for all infinite rank classical groups. We now give some additional details, again in the case of the symplectic group.

\textit{(a) Enumeration.} We first study the number of embeddings $V_{a,b} \to V_{m,n}$. The main result (Proposition~\ref{prop:Sp-Q}) is that there is a polynomial $Q_{a,b}(t,u)$ such that the number of such embeddings is given by $Q_{a,b}(q^{2m}, q^n)$. This result is completely elementary, but the existence of these polynomials is a fundamental combinatorial feature of the situation that plays an important role in subsequent analysis. The polynomial $P_{a,b}(t)=Q_{a,b}(t,1)$, which counts embeddings of $V_{a,b}$ into the non-degenerate spaces $V_{m,0}$ is particularly nice. We give an explicit formula for it (Proposition~\ref{prop:Sp-P}), which eventually gives an explicit formula for measures.

\textit{(b) The Burnside ring.} We next determine the Burnside ring $\BB$, thus resolving Problem~\ref{prob:olig2}. We show that, after inverting $q$, it is freely generated by the two classes $x_{1,0}$ and $x_{0,1}$ (Proposition~\ref{prop:Sp-burnside}). The fact that the Burnside ring is so small is what ultimately gives us tight control on measures and allows us to prove Conjecture~\ref{conj:interp1}.

A natural approach to proving this is to compute the products of $x_{1,0}$ and $x_{0,1}$ with general classes, and use this to show they generate. While these computations are possible in principle, they quickly become quite complicated; we therefore take a different approach. We observe that $\BB$ carries a filtration whose graded pieces have the same dimension as the number of monomials in $x_{1,0}$ and $x_{0,1}$ of the appropriate degree. Thus, to prove the result, it suffices to show that $x_{1,0}$ and $x_{0,1}$ are algebraically independent. For this, we need a source of homomorphisms out of the Burnside ring. This is provided by a version of the classical marks homomorphisms: for each $n$, $m$, there is a ring homomorphism
\begin{displaymath}
\BB \to \bZ, \qquad x_{a,b} \mapsto \# \Hom_G(X_{n,m}, X_{a,b}).
\end{displaymath}
We note that the number above is equal to the number of embeddings $V_{a,b} \to V_{n,m}$, and this allows us to relate the marks homomorphisms to the $Q$ polynomials discussed above. With this ample supply of homomorphisms in hand, it is not difficult to show that $x_{1,0}$ and $x_{0,1}$ are independent.

\textit{(c) The infinitesimal Burnside ring.} Let $G(n)$ be the subgroup of $G$ fixing the first $2n$ standard basis vectors. The \defn{infinitesimal Burnside ring} $\hat{\BB}$ is the direct limit of the Burnside rings of $G(n)$ over $n$, where the transition maps are given by restriction. In fact, each $G(n)$ is isomorphic to $G$, and so the restriction map on Burnside rings corresponds to a map $\delta \colon \BB \to \BB$ (it is easily seen to be independent of $n$). Thus $\hat{\BB}$ is the direct limit of the system $\BB \to \BB \to \cdots$ where all transition maps are $\delta$.

We compute $\delta$ explicitly, and show that it is an isomorphism after inverting $q$ (Proposition~\ref{prop:Sp-delta}). Thus $\hat{\BB}$ and $\BB$ coincide after inverting $q$ (Corollary~\ref{cor:Sp-inf-burnside}), and so $\hat{\BB}[1/q]$ is freely generated by $x_{1,0}$ and $x_{0,1}$. A curious point here is that $\delta$ is diagonalizable on $\BB[1/q]$, and induces a canonical grading on this ring (Remark~\ref{rmk:Sp-grading}).

\textit{(d) Measures.} By general theory, there is a ring homomorphism of rings $\hat{\BB} \to \Theta(G)$ that is surjective after tensoring to $\bQ$. We thus see that $\Theta(G) \otimes \bQ$ is generated by the images $t$ and $u$ of the classes $x_{1,0}$ and $x_{0,1}$. An explicit computation shows that $u$ is generated by $t$, and so $\Theta(G) \otimes \bQ$ is generated by the single element $t$.

The spaces $V_m=V_{m,0}$ are homogeneous, in the sense that any two embeddings $V_{a,b} \to V_m$ differ by an automorphism of $V_m$. The fact that every $V_{a,b}$ embeds into a finite dimensional homogeneous space is very significant: it means that the group $G$ is \defn{smoothly approximable}. By general theory, the marks homomorphisms defined by the homogeneous spaces give a ring homomorphism out of $\Theta(G)$. Using this, we show that $t$ is transcendental over $\bQ$, and so $\Theta(G) \otimes \bQ=\bQ[t]$ (Proposition~\ref{prop:Sp-Theta}). In other words, in characteristic~0, the measures for $G$ form a 1-parameter family. This resolves Problem~\ref{prob:olig1}.

\textit{(e) Tensor categories.} Suppose now that $k$ is algebraically closed of characteristic~0, and let $t \in k$ be non-zero. Let $\mu_t$ be the $k$-valued measure for $G$ corresponding to $t$. Using a number theoretic result of Laxton \cite{Laxton}, we show that one can choose the fields $k_n$ and the ultrafilter $\cF$ from \S \ref{ss:intro-interp}, such that the element $(q^n)_{n \ge 0}$ of the ultraproduct of the $k_n$'s has the same minimal polynomial over $\bQ$ as $t$. Let $\cT$ be the associated ultraproduct category. We show that (after an appropriate extension of scalars) there is a tensor functor
\begin{equation} \label{eq:functor}
\Phi \colon \uPerm(G, \mu_t) \to \cT.
\end{equation}
Since nilpotent endomorphisms in $\cT$ have trace zero, it follows that the same is true for $\uPerm(G, \mu_t)$. This resolves Problem~\ref{prob:olig3}. With this result in hand, the general theory of \cite{repst} provides an abelian envelope $\uRep(G, \mu_t)$ of the category $\uPerm(G, \mu_t)$, which, by definition, is a pre-Tannakian category; moreover, the abelian envelope is semi-simple if $t$ is a regular parameter, i.e., not of the form $q^{2n}$ with $n \in \bN$.

\textit{(f) Characters.} Let $\GG$ be the subgroup of $G$ consisting of elements that are centralized by some $G(n)$. This group is very large: it contains $\bigcup_{n \ge 1} H_n$, the union of the finite symplectic groups. Let $M$ be an object of $\uRep(G, \mu_t)$. Given $g \in \GG$ there is an endomorphism $g \colon M \to M$ in $\uRep(G(n), \mu_t)$ for $n$ sufficiently large. We define $\chi_M(g)$ to be its trace. We thus have a function
\begin{displaymath}
\chi_M \colon \GG \to k,
\end{displaymath}
which we call the \defn{character} of $M$. We prove that the characters of simple modules are linearly independent (at regular parameters); see \S \ref{ss:GL-char} for details in the general linear group case. Using this, we show that the functor \eqref{eq:functor} is an equivalence. This resolves Problem~\ref{prob:compar}, and, also proves Conjectures~\ref{conj:interp1} and~\ref{conj:interp2}.

In terms of the ultraproduct categories, the independence of characters admits a more concrete interpretation: it essentially says that if $V$ and $W$ are irreducible representations of $\GL_n(\bF)$ that appearing in small tensor powers of the permutation representation, then the characters of $V$ and $W$ can be distinguished at a small element of $\GL_n(\bF)$, meaning one living in the standard copy of $\GL_m(\bF)$ for some small $m$. This even holds if $V$ and $W$ are in positive characteristic.

\subsection{Further comments}

We make a few additional comments on aspects of this paper.

\textit{(a) Interplay.} One interesting aspect of the above story is that the oligomorphic and ultraproduct approaches are played off of one another. For instance, it is easy to see that nilpotents have trace~0 in the ultraproduct category, and the functor \eqref{eq:functor} allows us to transfer this to the oligomorphic side. Conversely, to see that the ultraproduct categories depend only on $t$, we compare them to the oligomorphic categories, instead of attempting directly compare two ultraproduct categories directly (which would be more difficult).

\textit{(b) Other groups.} We have discussed case of the symplectic group above. We establish similar results for the general linear group, the orthogonal group, and the unitary groups. The general linear group plays a special role, as we prove all results in full detail in this case. For the other groups, some results are easily reduced to the $\GL$ case, while other results have nearly identical proofs to the $\GL$ case, which allows us to omit some details. The orthogonal group has a unique feature: it admits two 1-parameter families of measures (first observed by Deligne \cite{DeligneLetter}). This means that the analog of Conjecture~\ref{conj:interp1} in this case is a little more complicated: the ultraproduct category $\cT$ depends on $t$ and one additional piece of parity data.

\textit{(c) The two $\GL$'s.} Let $G=\bigcup_{n \ge 1} \GL_n(\bF)$ be the infinite general linear group over the finite field $\bF$. This acts on $\bV=\bigcup_{n \ge 1} \bF^n$, and also on the restricted dual $\bV_*=\bigcup_{n \ge 1} (\bF^n)^*$. The actions on $\bV$ and on $\bV \times \bV_*$ are both oligomorphic, and induce on $G$ two different topologies, which we call the \defn{parabolic} and \defn{Levi} topologies. This means that the general theory of \cite{repst} can be applied in each case. A priori, there is not an obvious reason why these should agree. The Levi case is what is most analogous to the other classical groups; in a sense, the parabolic topology is a special feature that exists only for $\GL$. This is somewhat analogous to what occurs in the algebraic representation theory of $\GL_n$: the class of rational representations behaves similarly to the representation theory of other algebraic groups, while the class of polynomial representations is a special feature of $\GL_n$.

Previous work on the interpolation categories for the groups $\GL_n(\bF)$ is related to the parabolic version of $G$. This includes Knop's work \cite{Knop}, the diagrammatic description given by Entova-Aizenbud and Heidersdorf \cite{EntovaAizenbudHeidersdorf}, and (most of) our work in \cite[\S 15]{repst}.

In this paper, we focus on the Levi version of $G$. The most important thing to know is that, in the end, the two versions have the same spaces of measures, and yield the same tensor categories. We do not have a conceptual reason for why the two cases have the same measures: this is simply a result of going through both calculations. On the other hand, we do have a conceptual understanding of why the tensor categories agree: the key point is that the modules $\cC(\bV)$ and $\cC(\bV_*)$ are isomorphic by a version of the Fourier transform (see Proposition~\ref{prop:GL-gen}).

Despite yielding the same tensor categories, there are some differences between the two cases worth noting. One particularly nice feature of the Levi case is that the principal open subgroups of $G$ are isomorphic to $G$ itself; this gives restriction functors from $\uRep(\GL_t)$ to $\uRep(\GL_{t/q^n})$. These functors are less transparent in the parabolic picture. Another difference is that the Burnside ring is a much more interesting (and complicated) object in the Levi case. We determine it completely.

\textit{(d) Semi-simplification.} A prevalent feature in the subject of interpolation categories is that there are ``good'' parameters, at which the category is semi-simple, and ``bad'' parameters, at which it is not. Moreover, at bad parameters, the semi-simplification of an appropriate version of the interpolation category recovers the representation category of the associated finite group. For instance, for Deligne's category $\uRep(\fS_t)$, the bad parameters are natural numbers, and if $t=n$ is a natural number then the semi-simplification of (an appropriate subcategory of) $\uRep(\fS_t)$ recovers $\Rep(\fS_n)$. This phenomenon remains true in the cases considered in this paper, though we do not discuss it in detail. We remark that \cite[\S 4.6]{arboreal} can be useful in proving such results in the oligomorphic setting.

\textit{(e) Smooth approximation.} There is an important and well-studied class of oligomorphic groups called the \defn{smoothly approximable} groups\footnote{Typically, ``smoothly approximable'' is applied to the structure the group acts on, but we will use it for the group as well.}. See \S \ref{s:ultra} for our definition, and \cite{CherlinHrushovski} for background. In a rough sense, an oligomorphic group is smoothly approximable if it densely contains an appropriate union of finite groups. The basic examples of such groups are the infinite symmetric group, and the groups studied in this paper. However, there are more examples, such as: (a) infinite classical groups over finite rings like $\bZ/p^n \bZ$; (b) the automorphism group of a free pro-finite $p$-group on a countable generating set\footnote{Technically, this is pro-oligomorphic.}; and (c) one coming from quadratic spaces in characteristic~2 \cite[Definition~2.1.4]{CherlinHrushovski}. For any such group, one has ultraproduct tensor categories and oligomorphic tensor categories, and so one can attempt to prove results like those in this paper. It would be interesting to do this in more cases, or even to obtain general results in this direction.

\subsection{Related work}

We briefly discuss some related literature.
\begin{itemize}
\item Deligne \cite{Deligne} first constructed the interpolation categories discussed in this paper. Knop \cite{Knop, Knop2} later gave a different construction in some cases, and used this to prove some finer results. The oligomorphic construction was first give in \cite{repst}.
\item One can also approach the tensor categories we consider diagrammatically, though this is challenging. Entova--Aizenbud and Heidersdorf give diagrammatics in the $\GL$ case \cite{EntovaAizenbudHeidersdorf}. It would be interesting to extend their work to other cases.
\item Much of what we do here extends known results from the symmetric group case: Harman \cite{Harman1, Harman2} established results on ultraproducts, \cite{ComesOstrik1, ComesOstrik} proved a number of basic results about Deligne's category, and Farahat and Higman interpolated the centers of the group algebras \cite{FarahatHigman}. As far as we are aware, the symmetric group versions of Conjectures~\ref{conj:interp1} and~\ref{conj:interp2} do not appear in published form, but Deligne did propose a version of the first conjecture in a letter to Ostrik (in positive characteristic).
\item Kriz \cite[\S 7.2]{Kriz} discusses the interpolation category for finite symplectic groups through her T-algebra formalism. This category also appears in \cite{azumaya}.
\item As discussed in \S \ref{ss:results}(f), our results have implications about the characters of representations of $\GL_n(\bF)$. Some (deeper) results of a similar flavor are found in \cite{GLT, GH}.
\item In \S \ref{ss:fh}, we define a version of the Farahat--Higman algebra for an arbitrary oligomorphic group. In the case of the infinite symmetric group, we recover the classical algebra of \cite{FarahatHigman}. Some special cases of our construction were studied in \cite{Ryba,KannanRyba}.
\item Our results on Burnside rings and measures rely on some results in enumerative geometry over finite fields. See \cite{Pless, Yoo} for some related results.
\end{itemize}

\subsection{Outline}

Part~\ref{part:th} studies general oligomorphic groups. In \S \ref{s:oligo} and \S \ref{s:tencat}, we review the general theory of \cite{repst}, and prove a few new results along the way. In \S \ref{s:smooth}, we introduce the smooth group algebra of an oligomorphic group. The center of this algebra generalizes the Farahat--Higman algebra, and is closely connected to character theory for oligomorphic groups. In \S \ref{s:ultra}, we study smooth approximation. We give a general method for constructing measures on smoothly approximable groups, and show that the corresponding tensor categories admit a natural functor from an ultraproduct category. We give some general criteria for showing that this functor is an equivalence.

Part~\ref{part:ex} studies specific oligomorphic groups. In \S \ref{s:gl} and \S \ref{s:gl2}, we treat the general linear group. As mentioned above, there are many more details in this case than the others, which is why we split it between two sections. We then handle the symplectic group (\S \ref{s:Sp}), orthogonal group (\S \ref{s:O}), and unitary group (\S \ref{s:U}).

\subsection{Notation}

We list some of the important notation.
\begin{description}[align=right,labelwidth=2.5cm,leftmargin=!]
\item[ $\bN$ ] The set of non-negative integers
\item[ $k$ ] The coefficient field, typically of characteristic~0
\item[ $\bF$ ] A finite field of cardinality $q$
\item[ $\bZ_q$ ] The ring $\bZ[1/q]$
\item[ $G$ ] An oligomorphic group (\S \ref{ss:oligo})
\item[ $\bS(G)$ ] The category of finitary smooth $G$-sets (\S \ref{ss:oligo})
\item[ $\BB(G)$ ] The Burnside ring (\S \ref{ss:burnside})
\item[ $\Theta(G)$ ] The ring carrying the universal measure (\S \ref{ss:meas})
\item[ $\fC$ ] The category of structures for an oligomorphic group (\S \ref{ss:struct})
\item[ $\cC(X)$ ] The Schwartz space on the $G$-set $X$ (\S \ref{ss:int})
\end{description}

\part{Theory} \label{part:th}

\section{Measures on oligomorphic groups} \label{s:oligo}

\subsection{Oligomorphic groups} \label{ss:oligo}

We begin by recalling some fundamental definitions. An action of a group $G$ on a set $\Omega$ is \defn{oligomorphic} if $G$ has finitely many orbits on $\Omega^n$ for all $n \ge 0$. An \defn{oligomorphic group} is a group equipped with a given faithful oligomorphic action. A \defn{pro-oligomorphic group} is a topological group $G$ that is Hausdorff, non-archimedean (open subgroups for a neighborhood basis of the identity), and Roelcke precompact (the set $U \backslash G/V$ is finite for all open subgroups $U$ and $V$).

Suppose $(G, \Omega)$ is a permutation group. For a finite subset $A$ of $\Omega$, let $G(A)$ be the subgroup of $G$ fixing each element of $A$. These subgroups are a neighborhood basis of the identity for a topology on $G$. If $(G, \Omega)$ is oligomorphic then this topology makes $G$ into a pro-oligomorphic group \cite[\S 2.2]{repst}. All pro-oligomorphic groups of interest in this paper are oligomorphic. However, it is convenient to develop the general theory in the pro-oligomorphic setting since most constructions depend only on the topology and not the choice of $\Omega$.

Let $G$ be a pro-oligomorphic group. An action of $G$ on a set $X$ is \defn{smooth} if for every $x \in X$ the stabilizer $G_x$ is an open subgroup of $G$. In what follows, we use the term ``$G$-set'' to mean ``set equipped with a smooth action of $G$.'' The action is \defn{finitary} if $G$ has finitely many orbits on $X$. We let $\bS(G)$ be the category of finitary smooth $G$-sets. This category is closed under finite products. If $U$ is an open subgroup of $G$ then finitary $G$-sets restrict to finitary $U$-sets, so there is a functor $\bS(G) \to \bS(U)$. See \cite[\S 2.3]{repst} for details.

Suppose $X$ is a $G$-set. A \defn{$\hat{G}$-subset} of $X$ is a subset stable under some open subgroup of $G$. Here $\hat{G}$ is just a formal symbol; one should think of it as a kind of infinitesimal neighborhood of the identity in $G$. The collection of $\hat{G}$-subsets of $X$ is an algebra, i.e., it is closed under complements, finite intersections, and finite unions. We can also make sense of a $\hat{G}$-set without an ambient $G$-set: this is a set equipped with a smooth action of some open subgroup, called the group of definition. The notion of finitary is well-defined for $\hat{G}$-sets. See \cite[\S 2.5]{repst} for details.

\begin{example} \label{ex:sym}
Let $\fS$ be the infinite symmetric group, i.e., the group of all permutations of the set $\Omega=\{1, 2, \ldots\}$. One easily sees that $\fS$ acts oligomorphically on $\Omega$, and so $(\fS, \Omega)$ is an oligomorphic group. Let $\fS(n)$ be the subgroup of $\fS$ fixing each of the numbers $1, \ldots, n$. These subgroups form a neighborhood basis for the identity in $\fS$. One can show that every open subgroup of $\fS$ is conjugate to one of the form $H \times \fS(n)$, where $H$ is a subgroup of the finite symmetric group $\fS_n$ \cite[Proposition~14.1]{repst}. The $\hat{G}$-subsets of $\Omega$ are the subsets that are either finite or cofinite. More generally, that $\hat{G}$-subsets of $\Omega^n$ are those that can be defined by a first-order formula using equality and finitely many constants.
\end{example}

\subsection{Burnside rings} \label{ss:burnside}

Let $G$ be a pro-oligomorphic group. The \defn{Burnside ring} of $G$, denoted $\BB(G)$, is the free abelian group with basis indexed by isomorphism classes of transitive $G$-sets. We write $\lbb X \rbb$ for the basis vector corresponding to $X$. More generally, if $X$ is a finitary $G$-set and $X=\coprod_{i=1}^n X_i$ is the orbit decomposition of $X$, we define $\lbb X \rbb = \sum_{i=1}^n \lbb X_i \rbb$. The Burnside ring is indeed a ring, via $\lbb X \rbb \cdot \lbb Y \rbb = \lbb X \times Y \rbb$.

Let $X$ be a transitive $G$-set. If $Y$ is a finitary $G$-set then the set $\Hom_G(X, Y)$ of $G$-maps is finite \cite[Proposition~2.8]{repst}. We define the \defn{marks homomorphism} by
\begin{displaymath}
m_X \colon \BB(G) \to \bZ, \qquad m_X(\lbb Y \rbb) = \# \Hom_G(X, Y).
\end{displaymath}
If we choose an isomorphism $X \cong G/U$ then $\Hom_G(X, Y)=Y^U$, and from this description it is evident that $m_X$ is a ring homomorphism. If $G$ is finite and $X_1, \ldots, X_r$ are representatives for the isomorphism classes of transitive $G$-sets then the corresponding marks homomorphisms give a ring homomorphism $\BB(G) \to \bZ^r$ that becomes an isomorphism after inverting $\# G$; see \cite{Solomon}.

If $U$ is an open subgroup of $G$ then restriction induces a ring homomorphism $\BB(G) \to \BB(U)$. We define the \defn{infinitesimal Burnside ring} of $G$, denoted $\BB(\hat{G})$, to be the direct limit of $\BB(U)$ over open subgroups $U$. If $X$ is a finitary $\hat{G}$-set then the class $\lbb X \rbb$ is well-defined in $\BB(\hat{G})$.

Suppose $(G, \Omega)$ is oligomorphic. A transitive $G$-set $X$ is then a subquotient of some $\Omega^n$; we define the \defn{level} of $X$ to be the minimal such $n$. More generally, the \defn{level} of a $G$-set $X$, denoted $\lev(X)$, is the supremum of the levels of its orbits. If $X$ and $Y$ are $G$-sets then we have
\begin{displaymath}
\lev(X \times Y) \le \lev(X) + \lev(Y).
\end{displaymath}
Define $F^n=F^n \BB(G)$ to be the span of the classes $\lbb X \rbb$ with $\lev(X) \le n$. We have $1 \in F^0$, and the above formula shows that $F^n \cdot F^m$ is contained in $F^{n+m}$, so this gives $\BB(G)$ the structure of a filtered ring. We note that this filtration depends on the choice of $G$-set $\Omega$, i.e., it is not intrinsic to $G$ as a topological group.

\subsection{Measures} \label{ss:meas}

Let $G$ be a pro-oligomorphic group and let $k$ be a commutative ring. We now recall the important notion of measure on $G$. In fact, there are two equivalent definitions.

\begin{definition} \label{defn:meas1}
A \defn{measure} on $G$ valued in $k$ is a rule $\mu$ that assigns to every finitary $\hat{G}$-set $X$ a value $\mu(X) \in k$ such that the following conditions hold:
\begin{enumerate}
\item If $X \cong Y$ then $\mu(X)=\mu(Y)$.
\item If $X$ is a singleton set then $\mu(X)=1$.
\item We have $\mu(X \amalg Y)=\mu(X)+\mu(Y)$.
\item We have $\mu(X^g)=\mu(X)$, where $X^g$ is $X$ with the action of $\hat{G}$ twisted by $g \in G$.
\item Suppose $Y \to X$ is a map of transitive $U$-sets, for some open subgroup $U$, and $F$ is the fiber over some point. Then $\mu(Y)=\mu(X) \cdot \mu(F)$. \qedhere
\end{enumerate}
\end{definition}

\begin{definition} \label{defn:meas2}
A \defn{measure} on $G$ valued in $k$ is a rule $\mu$ that assigns to every map $f \colon Y \to X$ of transitive $G$-sets a quantity $\mu(f) \in k$ such that the following conditions hold:
\begin{enumerate}
\item We have $\mu(f)=1$ if $f$ is an isomorphism.
\item If $g \colon Z \to Y$ is a second map of transitive $G$-sets then $\mu(g \circ f)=\mu(g) \cdot \mu(f)$.
\item Suppose $X' \to X$ is another map of transitive $G$-sets, and $f' \colon Y' \to X'$ is the base change of $f$. Let $Y'=\bigsqcup_{i=1}^n Y'_i$ be the orbit decomposition of $Y'$, and let $f'_i$ be the restriction of $f'$ to $Y'_i$. Then $\mu(f)=\sum_{i=1}^n \mu(f'_i)$. \qedhere
\end{enumerate}
\end{definition}

We briefly explain how the two definitions correspond (see \cite[\S 3.6]{repst} for details). Suppose $\mu$ is given in the first sense. Given a map $f \colon Y \to X$ of transitive $G$-sets, we define $\mu(f)=\mu(F)$, where $F$ is the fiber of $f$ over any point. Now suppose $\mu$ is given in the second sense. Let $X$ be a finitary $\hat{G}$-set, and suppose $U$ is a group of definition of $X$. Letting $f \colon G \times_U X \to G/U$ be the projection on the first coordinate, we define $\mu(X)=\mu(f)$. In what follows, we freely pass between the two points of view.

There is a universal measure valued in a ring $\Theta(G)$. To define $\Theta(G)$, start with the polynomial ring on symbols $[X]$, where $X$ is a finitary $\hat{G}$-set, and then quotient by the ideal generated by the equations corresponding to the axioms of measure. Of course, if $f \colon Y \to X$ is a map of transitive $G$-sets there is also a class $[f]$ in $\Theta(G)$. Universality means that giving a measure valued in $k$ is equivalent to giving a ring homomorphism $\Theta(G) \to k$.

There is a natural surjective ring homomorphism $\BB(\hat{G}) \to \Theta(G)$, defined by mapping the class $\lbb X \rbb$ to the class $[X]$. The kernel of this map is the ideal generated by elements corresponding to axioms (d) and (e) in Definition~\ref{defn:meas1}; see \cite[Proposition~3.10]{repst} for details. In our cases of interest, we will first compute $\BB(\hat{G})$, and then use this surjection to obtain generators for $\Theta(G)$.

There are two ``niceness'' conditions on measures that play an important role in this paper. We say that a measure $\mu$ is \defn{regular} if $\mu(X)$ is a unit of the ring $k$ whenever $X$ is a transitive $G$-set. We say that $\mu$ is \defn{quasi-regular} if there is some open subgroup $U$ such that the restriction of $\mu$ to $U$ is regular\footnote{One easily sees that a measure on $G$ restricts to a measure on any open subgroup.}. These definitions will typically be applied when $k$ is a field. See \cite[\S 3.7]{repst} for additional details.

\begin{remark} \label{rmk:burn-prod}
Let $G$ and $H$ be pro-oligomorphic groups. If $G$ is finite then there is a forgetful functor $\bS(G \times H) \to \bS(H)$, which induces a ring homomorphism
\begin{displaymath}
\BB(G \times H) \to \BB(H).
\end{displaymath}
If $G$ is infinite this clearly does not work. However, one can show that there is a natural ring homomorphism
\begin{displaymath}
\BB(G \times H) \to \Theta(G) \otimes \BB(H), \qquad
\lbb X \rbb \mapsto [F] \otimes \lbb X/G \rbb,
\end{displaymath}
where $X$ is transitive and $F$ is the fiber of $X \to X/H$. This homomorphism reverts to the previous one when $G$ is finite, as then $\Theta(G)=\bZ$; see \cite[\S 3.4(b)]{repst}.
\end{remark}

\subsection{Constant maps} \label{ss:const}

Let $G$ be a pro-oligomorphic group equipped with a measure $\mu$ valued in a ring $k$. We say that a map $f \colon Y \to X$ of non-empty finitary $G$-sets is \defn{$\mu$-constant} if there exists $c \in k$ such that $\mu(f^{-1}(x))=c$ for all $x \in X$; in this case, we define $\mu(f)=c$. Essentially by the definition of measure, if $X$ is transitive then $f$ is $\mu$-constant. These maps have the following properties:
\begin{enumerate}
\item If $f \colon Y \to X$ is an isomorphism of $G$-sets then $f$ is $\mu$-constant and $\mu(f)=1$.
\item Suppose $g \colon Z \to Y$ and $f \colon Y \to X$ are $\mu$-constant. Then $g \circ f$ is $\mu$-constant, and $\mu(gf)=\mu(g) \cdot \mu(f)$.
\item Suppose $f \colon Y \to X$ is $\mu$-constant and $X' \to X$ is any map of $G$-sets, with $X'$ non-empty. Let $f' \colon Y' \to X'$ be the base change of $f$. Then $f'$ is $\mu$-constant, and $\mu(f')=\mu(f)$.
\end{enumerate}
We leave the verifications to the reader.

\subsection{Pulling back measures} \label{ss:pullback}

Let $G$ and $H$ be pro-oligomorphic groups, and let $H \to G$ be a group homomorphism. Consider the following condition:
\begin{itemize}
\item[$(\dag)$] Finitary $G$-sets restrict to finitary (and smooth!) $H$-sets.
\end{itemize}
This condition implies that the map $H \to G$ is continuous: indeed, if $U$ is an open subgroup of $G$ then $G/U$ is a smooth $H$-set, and so $H \cap U$ must be an open subgroup of $H$, as it is the stabilizer in $H$ of $1 \in G/U$. If $(G, \Omega)$ is oligomorphic then $(\dag)$ holds if and only if $\Omega$ is a finitary $H$-set. Suppose that $(\dag)$ holds. We then have a restriction functor
\begin{displaymath}
\bS(G) \to \bS(H)
\end{displaymath}
We now investigate how this functor interacts with measures.

\begin{proposition} \label{prop:pullback}
Let $\nu$ be a measure for $H$. Suppose that the following condition holds:
\begin{itemize}
\item[$(\ast)$] Whenever $f \colon Y \to X$ is a map transitive $G$-sets, $f$ is $\nu$-constant as a map of $H$-sets.
\end{itemize}
We then obtain a measure $\mu$ for $G$ by defining $\mu(f)=\nu(f)$, for $f$ as above.
\end{proposition}

\begin{proof}
The axioms for measure follow from the properties of constant maps in \S \ref{ss:const}.
\end{proof}

We now give a group-theoretic criterion for verifying the condition appearing in the above proposition. Let $G'$ be the set of elements of $G$ whose centralizer in $H$ is open. This is easily seen to be a subgroup of $G$.

\begin{proposition} \label{prop:pullback-crit}
Suppose $G'$ and $H$ generate a dense subgroup of $G$. Then condition $(\ast)$ from Proposition~\ref{prop:pullback} holds for any measure $\nu$.
\end{proposition}

\begin{proof}
Let $f \colon Y \to X$ be a map of transitive $G$-sets. We must show that $f$ is $\nu$-constant. Let $x$ be an element of $X$. If $h \in H$ then the $\hat{H}$-sets $f^{-1}(x)$ and $f^{-1}(hx)$ are isomorphic up to twisting the action by $h$, and thus have the same measure. If $g \in G'$ then $g \colon f^{-1}(x) \to f^{-1}(gx)$ is an isomorphism of $\hat{H}$-sets (since the centralizer of $g$ in $H$ is open), and so the two sets have the same measure. We thus see that if $g$ belongs to the subgroup generated by $G'$ and $H$ then $f^{-1}(x)$ and $f^{-1}(gx)$ have the same measure. Since this group is dense in $G$, it acts transitively on $X$, and so the result follows.
\end{proof}

\subsection{Stabilizer classes} \label{ss:stab}

Let $G$ be a pro-oligomorphic group. Sometimes one does not want to work with all of $\bS(G)$, but only a subcategory of it. Stabilizer classes give a convenient way of parametrizing particularly nice subcategories.

Let $\sO(G)$ be the collection of all open subgroups of $G$. A \defn{stabilizer class} is a subset $\sE$ of $\sO(G)$ satisfying the following conditions: (a) $\sE$ contains $G$; (b) $\sE$ is closed under finite intersections; (c) $\sE$ is closed under conjugation; and (d) $\sE$ is a neighborhood basis of the identity. Of course, $\sO(G)$ is itself an example of a stabilizer class.

Let $\sE$ be a stabilizer class. We say that an action of $G$ on a set $X$ is \defn{$\sE$-smooth} if for each $x \in X$ the stabilizer $G_x$ belongs to $\sE$. We let $\bS(G; \sE)$ be the category of finitary $\sE$-smooth $G$-sets. It is again closed under products. One has a Burnside ring $\BB(G, \sE)$ relative to $\sE$, as well as the infinitesimal Burnside ring. One also has a notion of measure relative to $\sE$ (see \cite[Remark~3.2(d)]{repst}), and a corresponding universal ring $\Theta(G, \sE)$.

We say that a stabilizer class $\sE$ is \defn{large} if for every open subgroup $U$ in $G$ there is some $V \in \sE$ such that $V$ is contained in $U$ with finite index. In this case, the natural map
\begin{displaymath}
\Theta(G; \sE) \otimes \bQ \to \Theta(G) \otimes \bQ
\end{displaymath}
is an isomorphism. See \cite[\S 2.6]{distal} for a proof, and a more precise result.

Suppose $(G, \Omega)$ is oligomorphic. We obtain a stabilizer class $\sE(\Omega)$ by taking all subgroups of the form $G(A)$, with $A$ a finite subset of $\Omega$. We typically use the term \defn{$\Omega$-smooth} in place of $\sE(\Omega)$-smooth. We note that a transitive $G$-set is $\Omega$-smooth if and only if it occurs as an orbit on $\Omega^n$ for some $n$. We typically write $\bS(G, \Omega)$ in place of $\bS(G, \sE(\Omega))$, and similarly for $\BB$ and $\Theta$.

\subsection{Structures} \label{ss:struct}

Fix an oligomorphic group $(G, \Omega)$. Let $A$ be a finite subset of $\Omega$. The \defn{definable closure} of $A$, denoted $\dcl(A)$, is the fixed point set of $G(A)$. We say that $A$ is \defn{definably closed} if $A=\dcl(A)$. The definable closure of $A$ is finite, definably closed, and satisfies $G(\dcl{A})=G(A)$.

We define a category $\fC$ as follows. The objects are the definably closed subsets of $\Omega$. A morphism $i \colon A \to B$ is a function for which there exists $g \in G$ such that $i(x)=gx$ for all $x \in A$. We call objects of $\fC$ \defn{structures} and morphisms \defn{embeddings}. The importance of this perspective comes from the fact that, in practice, $\fC$ will often have a concrete description (e.g., see Example~\ref{ex:sym-struc}).

Let $A$ be a structure. We say that a function $i \colon A \to \Omega$ is an \defn{embedding} if the induced function $A \to i(A)$ is one, i.e., there exists $g \in G$ such that $i(x)=gx$ for all $x \in A$. We let $X(A)$ be the set of all such embeddings. This is naturally a $G$-set, for if $i$ is an embedding and $g \in G$ then $x \mapsto gi(x)$ is also an embedding. If $i \colon A \to B$ is an embedding in $\fC$ then an embedding $B \to \Omega$ restricts to one on $A$; this defines a $G$-equivariant map $i^* \colon X(B) \to X(A)$.

\begin{proposition} \label{prop:struct-equiv}
The functor $A \mapsto X(A)$ is an anti-equivalence between the category $\fC$ of structures and the category of transitive $\Omega$-smooth $G$-sets.
\end{proposition}

\begin{proof}
Write $\Phi$ for the functor in question. We first show that $\Phi$ is faithful. Suppose that $i$ and $j$ are embeddings $A \to B$ such that $i^*=j^*$. Let $k \in X(B)$ be the canonical embedding, i.e., the identity map $B \to \Omega$. Then $i^*(k)=i$ and $j^*(k)=k$, and so we see $i=j$, as required.

We next show that $\Phi$ is full. Suppose $f \colon X(B) \to X(A)$ is a $G$-map. Let $i \in X(B)$ be the canonical embedding (as above), and let $i'=f(i)$, so that $i' \colon A \to \Omega$ is an embedding. The stabilizer of $i$ in $G$ is the group $G(B)$ fixing each element of $B$. Since $f$ is a $G$-map, it follows that $G(B)$ fixes $i'$, which exactly means that image of $i'$ is contained in $\Omega^{G(B)}$. This is the definable closure of $B$, which coincides with $B$ by our definition of structure. Thus $i'$ factors as $i'=i \circ j$, where $j \colon A \to B$ is an embedding. We have $gi'=gij$ for all $g \in G$. Since every element of $X(B)$ is of the form $gi$, we see that $f$ coincides with $j^*$, and so $\Phi$ is full.

Finally, we show that $\Phi$ is essentially surjective. Let $X$ be a $\Omega$-smooth $G$-set. By definition, $X$ is isomorphic to an orbit on some $\Omega^n$, say the orbit of $(x_1, \ldots, x_n)$. Let $A$ be the definable closure of $A_0=\{x_1, \ldots, x_n\}$. We have a natural $G$-map $f \colon X(A) \to X$ by $f(i)=(i(x_1), \ldots, i(x_n))$. This map is surjective since $X$ is transitive. We now show it is injective. Thus suppose we have $i,j \in X(A)$ such that $f(i)=f(j)$, meaning $i(x)=j(x)$ for $x \in A_0$. By definition, there exist $g,h \in G$ such that $i(x)=gx$ and $j(x)=hx$ for all $x \in A$. We thus see that $gh^{-1}$ belongs to $G(A_0)$. Since $G(A_0)=G(A)$, we see that $i=j$, as required.
\end{proof}

Suppose $i \colon A \to B$ and $j \colon A \to C$ are embeddings of structures. An \defn{amalgamation} of $B$ and $C$ over $A$ is a structure $D$ equipped with maps $j' \colon B \to D$ and $i' \colon C \to D$ such that $j' \circ i = i' \circ j$ and $D$ is the definable closure of $\im(i') \cup \im(j')$. One easily sees that the fiber product $X(B) \times_{X(A)} X(C)$ is naturally isomorphic to the disjoint union of the sets $X(D)$ as $(D, i', j')$ varies over isomorphism classes of amalgamations. (Key point: if $(j', i')$ is an element of the fiber product then the definable closure $D$ of $\im(i') \cup \im(j')$ is an amalgamation.)

The above discussion allows one to translate concepts from $G$-sets to structures. For example, $\BB(G, \Omega)$ has a basis indexed by isomorphism classes of structures. The product of two structures in $\BB(G, \Omega)$ is the sum of their amalgamations (over the empty structure). Similarly, a measure can be seen as a rule that assigns a number to each embedding $A \to B$, such that some conditions hold.

\begin{example} \label{ex:sym-struc}
Consider the infinite symmetric group $(\fS, \Omega)$ from Example~\ref{ex:sym}. All finite subsets of $\Omega$ are definably closed. The category of structures $\fC$ is equivalent to the category of finite sets, with morphisms being injections. If $A$ is a finite set then $X(A)$ is simply the set of all injections $A \to \Omega$.
\end{example}

\section{Tensor categories from oligomorphic groups} \label{s:tencat}

\subsection{Integration} \label{ss:int}

Let $G$ be a pro-oligomorphic group and let $\mu$ be a measure for $G$ valued in a field $k$. Suppose $X$ is a $G$-set and $\phi \colon X \to k$ is a function. We say that $\phi$ is \defn{smooth} if it is $U$-invariant for some open subgroup $U$, and \defn{Schwartz} if additionally its support consists of finitely many $U$-orbits. We let $\cC(X)$ be the collection of all Schwartz functions on $X$; we call this the \defn{Schwartz space} of $X$.

Given a Schwartz function $\phi$ on $X$, we can form the integral
\begin{displaymath}
\int_X \phi(x) dx.
\end{displaymath}
Integration is a linear functional on $\cC(X)$, and uniquely characterized by the condition that the integral of the indicator function $1_A$ is $\mu(A)$, where $A$ is a finitary $\hat{G}$-subset of $X$. See \cite[\S 4.3]{repst} for details. More generally, if $f \colon X \to Y$ is a morphism in $\bS(G)$ then there is a push-forward map $f_* \colon \cC(X) \to \cC(Y)$. These constructions have the basic properties one would expect; in particular, push-forwards are transitive, which implies a version of Fubini's theorem holds in this setting. See \cite[\S 4.4]{repst}.

\subsection{Permutation modules}

Let $X$ and $Y$ be finitary $G$-sets. A \defn{$Y \times X$ matrix} is a Schwarz function $A \in \cC(Y \times X)$. If $Z$ is a third such set and $B$ is a $Z \times Y$ matrix then we define $BA$ to be the $Z \times X$ matrix given by
\begin{displaymath}
(BA)(z, x) = \int_Y B(z,y) A(y,z) dy.
\end{displaymath}
This version of matrix multiplication has the expected properties. Note that an element of $\cC(X)$ can be viewed as a $X \times \bone$ matrix, i.e., a column vector indexed by $X$, and a $Y \times X$ matrix $A$ induces a linear map $A \colon \cC(X) \to \cC(Y)$ by matrix multiplication. The matrix $A$ is determined by its associated linear map: the $x$th column of $A$ is the Schwartz function $A \delta_x$. See \cite[\S 7]{repst} for additional details.

We now define a category $\uPerm(G, \mu)$ of formal permutation modules. The objects are the Schwartz spaces $\cC(X)$, where $X$ is a finitary $G$-set. A morphism $\cC(X) \to \cC(Y)$ is a $G$-invariant $Y \times X$ matrix (or the linear map it induces); we note that the space of morphisms is a finite dimensional $k$-vector space. Composition is matrix multiplication (or composition of linear transformations). This category has direct sums and a tensor product $\uotimes$. On objects, these operations are given by
\begin{displaymath}
\cC(X) \oplus \cC(Y) = \cC(X \amalg Y), \qquad
\cC(X) \uotimes \cC(Y) = \cC(X \times Y).
\end{displaymath}
On morphisms, the direct sum is the usual block matrix construction, while tensor product is the usual Kronecker product of matrices. Every object of $\uPerm(G, \mu)$ is rigid and self-dual. See \cite[\S 8]{repst} for more details.

\subsection{Nilpotents}

Let $\mu$ be a measure for $G$ valued in a field $k$. If $A$ is an $X \times X$ matrix with entries in $k$, its \defn{trace} is defined by
\begin{displaymath}
\tr(A)=\int_X A(x,x) dx.
\end{displaymath}
If $A$ is $G$-invariant then it defines an endomorphism $A$ of $\cC(X)$ in $\uPerm(G, \mu)$, and the trace of this endomorphism (in the sense of tensor categories) coincides with the above trace \cite[Proposition~8.10]{repst}. The following is a very important condition on $\mu$:
\begin{itemize}
\item[(Nil)] If $A$ is a $G$-invariant nilpotent matrix then $A$ has trace~0.
\end{itemize}
In \cite{repst}, we showed that (Nil) automatically holds for measures valued in a field of positive characteristic. We observe one simple result about (Nil) here.

\begin{proposition} \label{prop:nil}
Let $\mu$ be a measure for $G$ valued in a field $k$, let $k'/k$ be an extension field, and let $\mu'$ be the corresponding $k'$-valued measure. Then $\mu$ satisfies (Nil) if and only if $\mu'$ does.
\end{proposition}

\begin{proof}
If $k$ has positive characteristic then both $\mu$ and $\mu'$ satisfy (Nil). Suppose then that $k$ has characteristic~0. If $\mu'$ satisfies (Nil) then clearly $\mu$ does as well. Conversely, suppose that $\mu$ satisfies (Nil); it suffices to prove that $\mu'$ satisfies (Nil) when $k'$ is algebraically closed, so we assume this. Let $A$ be a $G$-invariant nilpotent $X \times X$ matrix with values in $k'$. Let $R$ be the ring of $G$-invariant $X \times X$  matrices with entries in $k$, and let $R'$ be the corresponding ring over $k'$, so that $R' = k' \otimes_k R$. Since $(R/J(R)) \otimes_k k'$ is semi-simple, it follows that $J(R')=k' \otimes_k J(R)$, where $J(-)$ denotes the Jacobson radical \cite{Steinberg}. The image of $A$ in $R'/J(R')$ is a nilpotent element, and therefore a sum of commutators; here we use that $R'/J(R')$ is semi-simple, and therefore a product of matrix algebras (since $k'$ is algebraically closed), and so the claim is elementary. We can therefore write
\begin{displaymath}
A = BC-CB + \sum_{i=1}^n \alpha_i D_i,
\end{displaymath}
where $B,C \in R'$, $\alpha_i \in k'$, and $D_i \in J(R)$. Since $\mu$ satisfies (Nil), the $D_i$ have trace~0. Of course, commutators have trace~0 as well, so $A$ has trace~0, as required.
\end{proof}

The above proposition implies that the (Nil) property for $\mu$ only depends on the kernel of the ring homomorphism $\mu \colon \Theta(G) \to k$. In other words, we can consider the locus in $\Spec(\Theta(G))$ where (Nil) holds.

\subsection{Abelian envelopes} \label{ss:abenv}

The category $\uPerm(G, \mu)$ is essentially never abelian. We now discuss two ways to make an abelian version. The simplest abelian result is the following:

\begin{theorem}
Suppose that $\mu$ is regular and satisfies (Nil). Then the Karoubi envelope of $\uPerm(G, \mu)$ is a semi-simple pre-Tannakian category.
\end{theorem}

\begin{proof}
By \cite[Proposition~7.19]{repst}, the endomorphism algebra of any object in $\uPerm(G, \mu)$ is semi-simple, and so the result follows.
\end{proof}

The next abelian result we discuss is more general, but also significantly more complicated. Since the details will not matter so much for this paper, we just indicate the main ideas. We define the \defn{completed group algebra} of $G$, denoted $A$, to be the inverse limit of the Schwartz spaces $\cC(G/U)$ as $U$ runs over open subgroups. This has an algebra structure via convolution of functions. We say that an $A$-module $M$ is \defn{smooth} if for each $x \in M$ there is some $U$ such that the action of $A$ on $x$ factors through $\cC(G/U)$. We let $\uRep(G, \mu)$ be the category of finite length\footnote{In \cite{repst}, we do not impose the finite length condition, but it will be more convenient to do so here.} smooth $A$-modules. If $X$ is a finitary $G$-set then the Schwartz space $\cC(X)$ is naturally a smooth $A$-module, and this allows one to define a functor
\begin{displaymath}
\uPerm(G, \mu) \to \uRep(G, \mu),
\end{displaymath}
at least if the spaces $\cC(X)$ are known to have finite length. See \cite[Part~III]{repst} for details. The following is the main result about the $\uRep$ category:

\begin{theorem}
Suppose that $\mu$ is quasi-regular and satisfies (Nil). Then $\uRep(G, \mu)$ is naturally a pre-Tannakian category. Moreover, the above functor is a tensor functor, and makes $\uRep$ the abelian envelope of $\uPerm$ in the sense of \cite[Definition~3.1.2]{CEAH}.
\end{theorem}

\begin{proof}
See \cite[Theorem~13.2]{repst} and \cite[Theorem~13.13]{repst}.
\end{proof}

We note that if $\mu$ is regular and satisfies (Nil) then $\uRep(G, \mu)$ is equivalent to the Karoubi envelope of $\uPerm(G, \mu)$.

\subsection{Restriction functors}

Let $G$ and $H$ be pro-oligomorphic groups and let $H \to G$ be a group homomorphism that induces a functor
\begin{displaymath}
\bS(G) \to \bS(H)
\end{displaymath}
as in \S \ref{ss:pullback}. Let $\nu$ be a measure on $H$ valued in a field $k$, and suppose condition $(\ast)$ from Proposition~\ref{prop:pullback} holds, so that we have an induced measure $\mu$ on $G$.

\begin{proposition} \label{prop:res-ab}
We have the following:
\begin{enumerate}
\item There is a natural faithful tensor functor
\begin{displaymath}
\Phi_0 \colon \uPerm(G, \mu) \to \uPerm(H, \nu), \qquad \cC(X) \mapsto \cC(X).
\end{displaymath}
\item If $\nu$ satisfies (Nil) then so does $\mu$.
\item Suppose $\mu$ and $\nu$ are quasi-regular and $\nu$ satisfies (Nil). Then $\Phi_0$ induces an exact faithful tensor functor
\begin{displaymath}
\Phi \colon \uRep(G, \mu) \to \uRep(H, \nu).
\end{displaymath}
\end{enumerate}
\end{proposition}

\begin{proof}
We have a functor $\bS(G) \to \uPerm(H, \nu)$ by $X \mapsto \cC(X)$. If $f \colon Y \to X$ is a map of $G$-sets then there are corresponding maps $f_*$ and $f^*$ in $\uPerm(H, \nu)$. If $X$ and $Y$ are transitive then the composition $f_* f^*$ is multiplication by $\nu(f)=\mu(f)$ since $f$ is $\nu$-constant. The mapping property for $\uPerm(G, \mu)$ established in \cite[\S 9.4]{repst} now provides the functor $\Phi_0$, and so (a) holds. Statements (b) follows directly from (a), while (c) follows from the mapping property of abelian envelopes.
\end{proof}

\subsection{Vector space objects} \label{ss:vec}

Let $G$ be a pro-oligomorphic group equipped with a measure $\mu$ valued in an algebraically closed field $k$. Let $\bF$ be a finite field of cardinality $q$, with $q$ non-zero in $k$. Fix a non-trivial additive character $\psi$ of $\bF$ valued in $k$.

By a \defn{vector space} in $\bS(G)$, we mean an $\bF$-vector space object in this category. Concretely, this is an $\bF$-vector space $V$ equipped with an action of $G$ that is finitary, smooth, and $\bF$-linear. For the oligomorphic groups we consider in Part~\ref{part:ex}, the basic $G$-set $\Omega$ will be a vector space object in this sense. Another example of a vector space in $\bS(G)$ occurs in Kriz's quantum Delannoy category \cite{Kriz2}.

Suppose $V$ and $W$ are vector spaces in $\bS(G)$, and we have a $G$-invariant bilinear pairing
\begin{displaymath}
\langle, \rangle \colon V \times W \to \bF
\end{displaymath}
that is \defn{non-degenerate}, in the sense that for every non-zero $x \in V$ there exists some $y \in W$ such that $\langle x, y \rangle \ne 0$, and similarly with the roles of $V$ and $W$ switched. Moreover, assume that $\mu(V)$ and $\mu(W)$ are non-zero. We define the \defn{Fourier transform} and the \defn{inverse Fourier transform} to be the linear maps
\begin{align*}
\sF &\colon \cC(V) \to \cC(W), & (\sF \phi)(y) &= \int_W \psi(\langle x, y \rangle) \phi(x) dy \\
\sF' &\colon \cC(W) \to \cC(V), & (\sF' \phi)(x) &= \mu(V)^{-1} \int_V \psi(-\langle x, y \rangle) \phi(x) dx
\end{align*}
We have the following version of Fourier inversion:

\begin{proposition} \label{prop:fourier}
$\sF$ and $\sF'$ are inverse isomorphisms; in particular, $\mu(V)=\mu(W)$.
\end{proposition}

\begin{proof}
We have
\begin{displaymath}
(\cF'\cF \phi)(x) = \mu(V)^{-1} \int_V \int_W \psi(\langle x'-x,y \rangle) \phi(x') dx' dy.
\end{displaymath}
Switch the order of integration. If $x \ne x'$ then $\psi(\langle x-x', - \rangle)$ is a non-trivial character of $W$ (since the pairing is non-degenerate), and so the inner integral vanishes (Lemma~\ref{lem:fourier} below). If $x=x'$ then the inner integral is just $\mu(W)$. We thus find $\cF' \circ \cF$ is the identity. A similar computation shows that $\cF \circ \cF'$ is equal to $c \cdot \id$, where $c=\mu(W)/\mu(V)$ is non-zero. We thus see that $\cF$ is an isomorphism, and so $\cC(V)$ and $\cC(W)$ have equal categorical dimension, i.e., $\mu(V)=\mu(W)$. It follows that $c=1$, and so $\cF$ and $\cF'$ are mutually inverse.
\end{proof}

\begin{lemma} \label{lem:fourier}
Let $\alpha \colon W \to k^{\times}$ be a Schwartz function that is also a non-trivial group homomorphism. Then $\int_W \alpha(y) dy = 0$.
\end{lemma}

\begin{proof}
Call the integral $c$. Since $\alpha$ is non-trivial, there is some $y_0 \in W$ such that $\alpha(y_0) \ne 1$. Multiplying the integral by $\alpha(y_0)$ and making a change of variables, we find $\alpha(y_0)c=c$, and so $c=0$.
\end{proof}

The maps $\sF$ and $\sF'$ are given by $G$-invariant matrices: e.g., $\sF$ is given by the matrix $(x,y) \mapsto \psi(\langle x, y \rangle)$. It follows that they define morphisms in the category $\uPerm(G, \mu)$. We thus obtain the following corollary:

\begin{corollary}
The Fourier transform $\sF \colon \cC(V) \to \cC(W)$ is an isomorphism in the category $\uPerm(G, \mu)$.
\end{corollary}

This is significant since it shows that one does not obtain new representations of $G$ by using the dual $W$ of $V$; that is, in the abelian category $\uRep(G, \mu)$ the modules $\cC(V)$ and $\cC(V) \oplus \cC(W)$ generate the same tensor subcategory. This is the key fact needed in \S \ref{s:gl} when we compare the two versions of the general linear group (parabolic and Levi).

The above corollary is all that we will require in this paper. However, we make a few additional comments here. Let $V$ be a vector space in $\bS(G)$. Then $\cC(V)$ is a commutative co-commutative Hopf algebra in $\uPerm(G, \mu)$. It also has operations induced by scalar multiplication by elements of $\bF$. We call this structure an \defn{$\bF$-Hopf algebra}. If $A$ is an $\bF$-Hopf algebra then the dual object $A^{\vee}$ is naturally an $\bF$-Hopf algebra as well; we call it the \defn{Cartier dual} of $A$. We write $\cC^*(V)$ for the Cartier dual of $\cC(V)$. The underlying object is $\cC(V)$, since it is self-dual, but the $\bF$-Hopf structure is changed; for instance, the algebra structure on $\cC^*(V)$ is given by convolution. One can show that $\cC^*(V)$ is an \'etale algebra if and only if $\mu(V) \ne 0$. Moreover, if $V$ and $W$ have a non-degenerate pairing then the Fourier transform $\cC^*(V) \to \cC(W)$ is an isomorphism of $\bF$-Hopf algebras.

Here is one interesting application of the above discussion. Suppose $\uPerm(G, \mu)$ has an abelian envelope $\uRep(G, \mu)$. Let $\Et$ be the category of \'etale algebras in $\uRep(G, \mu)$. In \cite{discrete}, we show that $\Et^{\op}$ is a pre-Galois category, meaning it is equivalent to $\bS(H)$ for some pro-oligomorphic group $H$. There is a natural inclusion $\bS(G) \to \Et^{\op}$ given by $X \mapsto \cC(X)$. Now suppose that $V$ is a vector space in $\bS(G)$ with $\mu(V) \ne 0$. Then $\cC^*(V)$ is a vector space in $\Et^{\op}$, and one can show that is admits a non-degenerate pairing with $V$. Thus, while $V$ may not have a dual in the pre-Galois category $\bS(G)$, it does have a canonical dual in a (possibly) larger pre-Galois category (under the assumptions we have made). It would be interesting to see this without going through tensor categories.

\section{The smooth group algebra} \label{s:smooth}

\subsection{c-smooth elements}

Fix a pro-oligomorphic group $G$. We say that an element $g \in G$ is \defn{c-smooth} if it is smooth for the conjugation action of $G$ on itself. In other words, $g$ is c-smooth if its centralizer
\begin{displaymath}
Z_G(g) = \{ h \in G \mid gh=hg \}
\end{displaymath}
is an open subgroup of $G$. We let $\GG$ be the set of c-smooth elements in $G$, which is easily seen to be a normal subgroup of $G$. In many cases, the group $\GG$ will be trivial; when this happens, the constructions in \S \ref{s:smooth} are uninteresting. However, for the groups of interest in this paper, $\GG$ will be a dense subgroup of $G$.

The object $\GG$ is interesting since $G$ acts smoothly on it. We can therefore regard it as an ind-object of $\bS(G)$, i.e., we have $\GG=\varinjlim_{\beta \in J} \GG^{\le \beta}$ where $J$ is a directed set and each $\GG^{\le \beta}$ is an object of $\bS(G)$. In fact, $\GG$ is a group object in the category of ind-objects. Moreover, if $X$ is any object of $\bS(G)$ then the action map
\begin{displaymath}
\GG \times X \to X, \qquad (g, x) \mapsto gx
\end{displaymath}
is a morphism in the ind-category of $\bS(G)$.

\begin{remark}
The definition of $\GG$ given above is not ``correct'' in all cases; see the discussion in \S \ref{ss:better-GG}. However for all examples in this paper, the above definition will suffice.
\end{remark}

\begin{example}
Let $\fS$ be the infinite symmetric group. Recall that we have defined this as the group of all permutations of the set $\Omega=\{1,2,\ldots\}$. An element of $\fS$ is c-smooth if and only if it is finitary, i.e., it fixes all but finitely many elements of $\Omega$. The group $\GG$ of c-smooth elements is thus the finitary version of the infinite symmetric group, which is the ascending union $\bigcup_{n \ge 1} \fS_n$ of the finite symmetric groups. The $G$-orbits on $\GG$ are indexed by partitions with infinite first row, corresponding to cycle decompositions in the usual manner.
\end{example}

\subsection{The smooth group algebra}

Fix a measure $\mu$ on $G$ valued in a field $k$. Let $B=\cC(\GG)$ be the Schwartz space on the set $\GG$. Recall that Schwartz functions are required to have finitary support, and so $B=\varinjlim \cC(\GG^{\le \beta})$. The multiplication map on $\GG$ induces a convolution product on $\GG$. Precisely, for $\phi_1, \phi_2 \in B$, their convolution product is given by
\begin{displaymath}
(\phi_1 \ast \phi_2)(g) = \int_{\GG} \phi_1(h) \phi_2(h^{-1}g) dh.
\end{displaymath}
One easily sees that $\phi_1 \ast \phi_2$ is indeed a Schwartz function on $\GG$, and that this gives $B$ the structure of an associative algebra; the identity element is the point mass $\delta_1$. We call $B=B_k(G, \mu)$ the \defn{smooth group algebra} of $G$.

The expression $B=\varinjlim \cC(\GG^{\le \beta})$ shows that we can regard $B$ as an ind-object in $\uPerm(G, \mu)$. One easily sees that it is an algebra object in the category of ind-objects. In fact, it is actually a Hopf algebra: the antipode is induced by inversion, the co-unit is given by integration over $\GG$, and the co-multiplication is push-forward along the diagonal $\GG \to \GG \times \GG$. Note that
\begin{displaymath}
\cC(\GG \times \GG) = \cC(G) \uotimes \cC(G) = B \uotimes B,
\end{displaymath}
which is why co-multiplication has the correct target. Generally, the above identifications are invalid if we use the ordinary tensor product of vector spaces, and so $B$ is not a Hopf algebra in the category of $k$-vector spaces.

Suppose $X$ is a finitary $G$-set. Then the Schwartz space is naturally a $B$-module. Indeed, given $\phi_1 \in B$ and $\phi_2 \in \cC(X)$, the product the Schwartz function $\phi_1 \ast \phi_2$ on $X$ defined by
\begin{displaymath}
(\phi_1 \ast \phi_2)(x) = \int_{\GG} \phi_1(g) \phi_2(g^{-1} x) dg.
\end{displaymath}
One eaily verifies that this gives a well-defined module structure. In fact, this construction defines a map
\begin{displaymath}
B \uotimes \cC(X) \to \cC(X)
\end{displaymath}
in (the ind-category of) $\uPerm(G, \mu)$, and so $\cC(X)$ is a module over $B$ in this category.

The previous paragraph holds in the abelian setting too. Precisely, suppose $\mu$ is quasi-regular and satisfies (Nil). Then for any $M$ in $\uRep(G, \mu)$ we have a natural map
\begin{displaymath}
B \uotimes M \to M
\end{displaymath}
in the (ind-category of) $\uRep(G, \mu)$, which gives $M$ the structure of a $B$-module. In the regular case, this follows from the previous paragraph since $\uRep(G, \mu)$ is then equivalent to the Karoubi envelope of $\uPerm(G, \mu)$. In general, the action of $\phi \in B$ on $x \in M$ is given by
\begin{displaymath}
\phi \cdot x = \int_{\GG} \phi(g) gx dg
\end{displaymath}
The integral here is taken in the sense of \cite[\S 11.9]{repst}, and it is routine to verify that this defines a $B$-module structure on $M$.

\subsection{The Farahat--Higman algebra} \label{ss:fh}

We define the \defn{Farahat--Higman algebra} of $G$, denoted $Z=Z_k(G, \mu)$, to be the $G$-invariant subalgebra of $B$. Thus, essentially by definition, we have
\begin{displaymath}
Z = \Hom_{\uPerm}(\bone, B),
\end{displaymath}
where here we are really taking $\Hom$ spaces in the ind-category of $\uPerm(G, \mu)$. As a vector space, $Z$ admits a fairly concrete description. Suppose that $\GG=\coprod_{\alpha \in I} C_{\alpha}$ is the decomposition of $\GG$ into $G$-orbits, i.e., the $C_{\alpha}$ are the conjugacy classes of c-smooth elements in $G$. Then $Z$ has a basis given by the indicator functions of the $C_{\alpha}$'s. We also have the following:

\begin{proposition}
The Farahat--Higman algebra $Z$ is contained in the center of $B$. If $\GG$ is dense in $G$ then $Z$ is equal to the center of $B$.
\end{proposition}

\begin{proof}
Suppose $\phi_1 \in Z$ and $\phi_2 \in B$. Then
\begin{align*}
(\phi_1 \ast \phi_2)(g)
&= \int_{\GG} \phi_1(h) \phi_2(h^{-1} g) dh
= \int_{\GG} \phi_1(gh^{-1}) \phi_2(h) dh \\
&= \int_{\GG} \phi_1(h^{-1}g) \phi_2(h) dh
= (\phi_2 \ast \phi_1)(g)
\end{align*}
In the second step, we made the change of variables $h \to gh^{-1}$, and in the third step we used that $\phi_1$ is $G$-invariant (where $G$ acts by conjugation on $\GG$). This shows that $\phi_1$ is central.

Now suppose that $\GG$ is dense in $G$ and $\phi$ belongs to the center of $\BB$. For any $g \in \GG$, the point mass $\delta_g$ belongs to $B$, and so we have $\delta_g \ast \phi = \phi \ast \delta_g$. This shows that $\phi$ is invariant under $\GG$. Since $\phi$ is also invariant under some open subgroup and $\GG$ is dense, it follows that $\phi$ is invariant under all of $G$.
\end{proof}

We have seen that a Schwartz space $\cC(X)$ is naturally a module over $B$. It is therefore a module over the subalgebra $Z$ as well. In fact:

\begin{proposition}
The action of $Z$ on Schwartz spaces commutes with morphisms in the category $\uPerm(G, \mu)$. In other words, the category $\uPerm(G, \mu)$ is naturally $Z$-linear.
\end{proposition}

\begin{proof}
Let $a \in Z$, thought of as a map $\bone \to B$ in (the ind-category) of $\uPerm(G, \mu)$. Given a morphism $A \colon \cC(X) \to \cC(Y)$ in $\uPerm(G, \mu)$, consider the following diagram:
\begin{displaymath}
\xymatrix{
\cC(X) \ar[r]^-{a \uotimes \id} \ar[d]^A &
B \otimes \cC(X) \ar[r] \ar[d]^{\id \uotimes A} &
\cC(X) \ar[d]^A \\
\cC(Y) \ar[r]^-{a \uotimes \id} &
B \otimes \cC(Y) \ar[r] &
\cC(Y) }
\end{displaymath}
The left square obviously commutes, and one easily sees that the right one does too: this follows from the fact that the action map $\GG \times X \to X$ in $\bS(G)$ is functorial. Thus the outer square commutes, which proves the proposition.
\end{proof}

The above proposition extends to abelian envelopes: if $\mu$ is quasi-regular and satisfies (Nil) then $\uRep(G, \mu)$ is naturally $Z$-linear. This structure can be very helpful when analyzing this category.

\begin{example}
For the infinite symmetric group, the Farahat--Higman algebra defined here coincides with the classical one from \cite{FarahatHigman}. For other choices of $G$, we obtain some of the generalizations of the classical Farahat--Higman algebra that have been recently studied \cite{KannanRyba, Ryba, WanWang}.
\end{example}

\subsection{Characters}

We now suppose that $\mu$ is quasi-regular and satisfies (Nil), so that we have the pre-Tannakian category $\uRep(G, \mu)$. We also assume that $k$ is algebraically closed, for simplicity. For finite groups one can study representations using characters, we'd like to do something similar for oligomorphic groups.  We now define and study a few kinds of characters in this setting.

First, suppose that $M$ is an arbitrary object of $\uRep(G, \mu)$. As we have seen above, $M$ is naturally a $B$-module. It follows that there is a canonical map $B \to M^{\vee} \uotimes M$ in (the ind-category of) $\uRep(G, \mu)$. Composing with the evaluation map, we obtain a map
\begin{displaymath}
\chi^*_M \colon B \to k
\end{displaymath}
in (the ind-category of) $\uRep(G, \mu)$. Concretely, this map can be described as follows. Given $\phi \in B$, let $U$ be an open subgroup of $G$ such that $\phi$ is $U$-invariant. Then the map $M \to M$ induced by $\phi$ is actually a map in $\uRep(U, \mu)$, and $\chi^*_M(\phi)$ is the trace of this map. For $\phi_1, \phi_2 \in B$, we have
\begin{displaymath}
\chi^*_M(\phi_1 \ast \phi_2) = \chi^*_M(\phi_2 \ast \phi_1)
\end{displaymath}
by basic properties of trace. Moreover, $\chi^*_M$ is $G$-invariant since it is a map in $\uRep(G, \mu)$.

If $g \in \GG$ then $\delta_g \in B$. We define the \defn{character} of $M$ to be the function
\begin{displaymath}
\chi_M \colon \GG \to k, \qquad \chi_M(g)=\chi_M^*(\delta_g) = \utr(g \vert M).
\end{displaymath}
It is invariant under conjugation by $G$ since $\chi^*_M$ is. We have $\chi_M(1)=\udim{M}$, as usual. Moreover:

\begin{proposition}
Suppose $M$ and $N$ are objects of $\uRep(G, \mu)$. Then
\begin{displaymath}
\chi_{M \oplus N} = \chi_M + \chi_N, \qquad
\chi_{M \uotimes N} = \chi_M \cdot \chi_N, \qquad
\chi_{M^{\vee}}(g) = \chi_M(g^{-1}),
\end{displaymath}
where addition and multiplication are done pointwise.
\end{proposition}

\begin{proof}
These identities are straightforward.
\end{proof}

The character $\chi_M$ is the restriction of $\chi^*_M$ to special elements of $B$. However, we can in fact recover $\chi^*_M$ from $\chi_M$, as the following proposition shows.

\begin{proposition} \label{prop:chi-star}
Let $M$ be an object of $\uRep(G, \mu)$ and let $\phi \in B$. Then
\begin{displaymath}
\chi^*_M(\phi) = \int_{\GG} \phi(g) \chi_M(g) dg.
\end{displaymath}
\end{proposition}

\begin{proof}
First note that since $\phi$ is a Schwartz function on $\GG$ and $\chi_M$ is $G$-invariant, the integrand above is a Schwartz function on $\GG$; thus its integral is defined. Let $h \in \GG$, let $U$ be an open subgroup, and let $X=\{uhu^{-1} \mid u \in U\}$. We verify the identity for $1_X$; this implies the general case since these functions span $B$. Let $V=Z_G(h) \cap U$. If $M$ is any object of $\uRep(G, \mu)$ then we have an averaging operator
\begin{displaymath}
\avg_{U/V} \colon M^V \to M^U, \qquad
\avg_{U/V}(x) = \int_{U/V} gx dg.
\end{displaymath}
Moreover, any map in $\uRep(G, \mu)$ commutes with the averaging operators. See \cite[\S 11.8, \S 11.9]{repst} for details. Now, $\delta_g$ belongs to $B^V$, and $\avg_{U/V}(\delta_g)=1_X$. On the trivial module, $\avg_{U/V}$ acts by the scalar $\mu(U/V)=\mu(X)$. We thus find
\begin{displaymath}
\chi^*_M(1_X) = \chi^*_M(\avg_{U/V}(\delta_h)) = \mu(U/V) \chi_M(\delta_g) = \int_{\GG} 1_X(g) \chi_M(g) dg
\end{displaymath}
and so the result follows. Note that $\chi_M(g)=\chi_M(h)$ for all $g \in X$ since $\chi_M$ is conjugation invariant, and $X \cong U/V$.
\end{proof}

Now let $L$ be a simple object of $\uRep(G, \mu)$. By Schur's lemma, the only endomorphisms of $L$ (in this category) are scalar multiplications. Since $\uRep(G, \mu)$ is $Z$-linear, it follows that elements of $Z$ act on $L$ by scalars. We thus have a $k$-algebra homomorphism
\begin{displaymath}
\psi_L \colon Z \to k
\end{displaymath}
which gives the scalar by which an element of $Z$ acts. We call $\psi_L$ the \defn{central character} of $L$. More generally, if $M$ is an arbitrary element of $\uRep(G, \mu)$, we say that $M$ admits a central character if there is an algebra homomorphism $\psi_M \colon Z \to k$ such that $Z$ acts on $M$ through $\psi_M$. Ordinary characters and central characters are closely linked to one another:

\begin{proposition} \label{prop:chi-psi}
Let $L$ be a simple object of $\uRep(G, \mu)$, let $g \in \GG$, let $C \subset \GG$ be its conjugacy class, and let $1_C \in Z$ be the indicator function of $C$. Then
\begin{displaymath}
\mu(C) \cdot \chi_L(g) = \udim(L) \cdot \psi_L(1_C).
\end{displaymath}
\end{proposition}

\begin{proof}
View $1_C$ as an operator on $L$. On the one hand, it is scalar multiplication by $\psi_L(1_C)$, and so its trace is $\udim(L) \cdot \psi_L(1_C)$. On the other hand, its trace is
\begin{displaymath}
\chi^*_L(1_C) = \int_C \chi_M(h) dh = \mu(C) \cdot \chi_M(g),
\end{displaymath}
where in the first step we used Proposition~\ref{prop:chi-star}, and in the second the fact that $\chi_M$ is equal to $\chi_M(g)$ identically on $C$. The result follows.
\end{proof}

Let $\epsilon \colon Z \to k$ be the map
\begin{displaymath}
\epsilon(\phi) = \int_{\GG} \phi(g) dg.
\end{displaymath}
This is an algebra homomorphism; in fact, it is the restriction of the co-unit of $B$ to $Z$. One easily sees that $\epsilon$ is the central character of the unit object. We call $\epsilon$ the \defn{trivial central character}.

\subsection{Independence of characters}

In the representation theory of finite groups, a central result is that irreducible characters are linearly independent. We now investigate this in the oligomorphic setting. We obtain only conditional results here. We begin with a general observation:

\begin{lemma} \label{lem:char-ind}
Let $R$ be a $k$-algebra and let $\psi_i \colon R \to k$, for $1 \le i \le n$, be distinct $k$-algebra homomorphisms. Then the $\psi_i$ are linearly independent.
\end{lemma}

\begin{proof}
Since $\psi_1$ and $\psi_2$ are not equal, there is some $x \in R$ such that $\psi_1(x) \ne \psi_2(x)$. Letting $y=(x-\psi_2(x))/(\psi_1(x)-\psi_2(x))$, we have $\psi_1(y)=1$ and $\psi_2(y)=0$. In the same way, for each $2 \le i \le n$, we can find $y_i \in R$ such that $\psi_1(y_i)=1$ and $\psi_i(y_i)=0$. Letting $z=y_2 \cdots y_n$, we have $\psi_1(z)=1$ and $\psi_i(z)=0$ for all $2 \le i \le n$. In the same way, for each $1 \le i \le n$, we can find $z_i \in R$ such that $\psi_i(z_j)=\delta_{i,j}$. If $\sum_{i=1}^n a_i \psi_i=0$ then, evaluating at $z_i$, we find $a_i=0$. Thus any linear dependence is trivial.
\end{proof}

The following is the main result we prove here:

\begin{proposition} \label{prop:cent-char}
Suppose $\mu$ is regular and the following condition holds:
\begin{itemize}
\item[$(\ast)$] If $L$ is a non-trivial simple then $\psi_L$ is non-trivial.
\end{itemize}
Then, for simples $L$ and $M$, we have $L \cong M$ if and only if $\psi_L=\psi_M$.
\end{proposition}

\begin{proof}
If $L \cong M$ then $\psi_L=\psi_M$. Conversely, suppose $\psi_L=\psi_M$. By Proposition~\ref{prop:chi-psi}, we have $\chi_L = c \cdot \chi_M$ where $c=\udim(L)/\udim(M)$ is a non-zero scalar. Write
\begin{displaymath}
M \otimes L^{\vee} = \bigoplus_{i=1}^n X_i, \qquad
L \otimes L^{\vee} = \bigoplus_{j=1}^m Y_j,
\end{displaymath}
where the $X_i$ and $Y_j$ are simple; moreover, index the $Y_j$'s so that $Y_1$ is trivial and the other $Y_j$ are non-trivial. We have
\begin{displaymath}
\chi_{L \otimes L^{\vee}} = \chi_L \chi_{L^{\vee}} = c \chi_M \chi_{L^{\vee}} = c \chi_{M \otimes L^{\vee}},
\end{displaymath}
and so
\begin{displaymath}
\sum_{j=1}^m \chi_{Y_j} = c \sum_{i=1}^n \chi_{X_i}.
\end{displaymath}
Appealing to Proposition~\ref{prop:chi-psi} again, we find
\begin{displaymath}
\sum_{j=1}^m \udim(Y_j) \psi_{Y_j} = c \sum_{i=1}^n \udim(X_i) \psi_{X_i}.
\end{displaymath}
Since characters are linearly independent (Lemma~\ref{lem:char-ind}), we can equate the coefficient of the trivial characters appearing above. On the left side, $(\ast)$ implies that only $\psi_{Y_1}$ is the trivial character, and so the trivial character appears with non-zero coefficient. It follows that some $\psi_{X_i}$ is trivial, and so $X_i$ is trivial by $(\ast)$. Thus $M \otimes L^{\vee}$ contains a trivial summand, and so $M \cong L$.
\end{proof}

\begin{corollary} \label{cor:cent-char}
Maintain the assumptions from the proposition. If $L_1, \ldots, L_n$ are mutually non-isomorphic simples then $\chi_{L_1}, \ldots, \chi_{L_n}$ are linearly independent.
\end{corollary}

\begin{proof}
Consider a linear dependence $\sum_{i=1}^n \alpha_i \chi_{L_i}=0$. By Proposition~\ref{prop:chi-psi}, we thus find
\begin{displaymath}
\sum_{i=1}^n \alpha_i \udim(L_i) \psi_{L_i} = 0.
\end{displaymath}
Since the $\psi_{L_i}$ are distinct, they are linearly independent (Lemma~\ref{lem:char-ind}), and so $\alpha_i=0$ for all $i$. Note that $\udim(L_i) \ne 0$ since $\uRep(G, \mu)$ is semi-simple.
\end{proof}

\subsection{A better definition} \label{ss:better-GG}

In our work on oligomorphic groups, a guiding philosophy is that everything relevant to tensor categories should depend only on the pre-Galois category $\bS(G)$ and not $G$ itself. The ind-object $\GG$ of $\bS(G)$, as we have defined it, actually depends on $G$ and is not intrinsic to $\bS(G)$. It is therefore an ``incorrect'' construction.

Here is an example. Let $\fX$ be the Cantor set, let $G$ be the group of all permutations of $\fX$ (ignoring the topology), and let $H$ be the group of all self-homeomorphisms of $\fX$. These groups both act oligomorphically on $\fX$. If $A$ and $B$ are finite subsets of $\fX$ then any bijection $A \to B$ is induced by a self-homeomorphism of $\fX$, and so it follows that $H$ is dense in $G$. Thus $\bS(G)=\bS(H)$. Let $\GG$ and $\HH$ be the c-smooth elements in $G$ and $H$. One can show that the c-smooth elements in either group are just the finitary permutations. A self-homeomorphism of $\fX$ that fixes all but finitely many points is the identity. Thus $\HH$ is trivial, while $\GG$ is dense in $G$. Therefore, $\GG$ and $\HH$ do not correspond under the equivalence $\bS(G)=\bS(H)$.

The ``correct'' definition of $\GG$ is as follows: $\GG$ is the ind-object of $\bS(G)$ that represents the functor $S \mapsto \Aut_S(S \times (-))$. This clearly only depends on the category $\bS(G)$. One can show that if $G$ is complete (in a certain sense) then $\GG$ coincides with the c-smooth elements in $G$. In the previous paragraph, $G$ is complete while $H$ is not. The groups we consider in Part~\ref{part:ex} will all be complete, so the issues discussed here are irrelevant to them.

\section{Smooth approximation and ultraproducts} \label{s:ultra}

\subsection{Homogeneous objects} \label{ss:homo}

Let $G$ be a pro-oligomorphic group. Let $\Delta$ be a transitive $G$-set, and let $H=\Aut_G(\Delta)$; note that $H$ is a finite group. Consider the functor
\begin{displaymath}
\Pi_{\Delta} \colon \bS(G) \to \bS(H), \qquad \Pi_{\Delta}(X) = \Hom_G(\Delta, X).
\end{displaymath}
This functor commutes with finite coproducts (since $\Delta$ is transitive) and with finite limits (since it is co-representable). We say that $\Delta$ is \defn{(absolutely) homogeneous} if whenever $X$ is a transitive $G$-set the $H$-set $\Pi_{\Delta}(X)$ is either transitive or empty.

There is also a version of this definition relative to a stabilizer class $\sE$. We say that a transitive $\sE$-smooth $G$-set $\Delta$ is \defn{homogeneous relative to $\sE$} if the above condition on $\Pi_{\Delta}(X)$ holds whenever $X$ is transitive and $\sE$-smooth. We now investigate how homogeneity depends on the choice of stabilizer class. We say that a transitive $G$-set $X$ is \defn{simply connected} if any finite cover is trivial, that is, if $Y$ is a $G$-set and $Y \to X$ is a $G$-map with finite fibers then $Y$ (as a set over $X$) is isomorphic to a finite disjoint union of copies of $X$.

\begin{proposition} \label{prop:sc-Pi}
Let $\Delta$ be simply connected and let $Y \to X$ be a finite cover of transitive $G$-sets. Then $\Pi_\Delta(Y) \to \Pi_\Delta(X)$ is surjective.
\end{proposition}

\begin{proof}
Suppose we have a $G$-map $f \colon \Delta \to X$. Then $Y \times_X \Delta$ is a finite cover of $\Delta$. Since $\Delta$ is simply connected, this cover has a $G$-equivariant section, which means we can lift $f$ to a map $\Delta \to Y$. The result follows.
\end{proof}


\begin{proposition} \label{prop:rel-homo}
Suppose $\Delta$ is simply connected and $\sE$ is a large stabilizer class. If $\Delta \in \bS(G, \sE)$ is homogeneous relative to $\sE$, then $\Delta$ is an absolutely homogeneous object of $\bS(G)$.
\end{proposition}

\begin{proof}
Let $X$ be a transitive $G$-set. We must show that $\Pi_{\Delta}(X)$ is either empty or $H$-transitive, where $H=\Aut(\Delta)$. Since $\sE$ is large, there is a finite cover $Y \to X$ where $X$ is transitive and $\sE$-smooth. By Proposition~\ref{prop:sc-Pi}, the map $\Pi_{\Delta}(Y) \to \Pi_{\Delta}(X)$ is surjective. Thus if $\Pi_{\Delta}(X)$ is non-empty then $\Pi_{\Delta}(Y)$ is non-empty, and $H$-transitive (since $\Delta$ is homogeneous relative to $\sE$), and so $\Pi_{\Delta}(X)$ is $H$-transitive (being a quotient of such a set).
\end{proof}

Suppose now that $(G, \Omega)$ is oligomorphic. We say that a structure $A$ is \defn{homogeneous} if the corresponding transitive $G$-set $X(A)$ is homogeneous relative to $\sE(\Omega)$. Explicitly, this means that whenever $i$ and $j$ are embeddings $B \to A$, for some structure $B$, there is an automorphism $\sigma$ of $A$ such that $i=\sigma \circ j$. In practice, we will show that structures are homogeneous, deduce from this that certain transitive sets are relatively homogeneous, and then use the above proposition to obtain absolute homogeneity. The following observation will help in the final step.

\begin{proposition} \label{prop:fincover}
Let $(G, \Omega)$ be oligomorphic and let $A \to B$ be an embedding of structures. Then $X(B) \to X(A)$ is an infinite cover if and only if the following condition holds:
\begin{itemize}
\item[$(\ast)$] for every $h \ge 0$ there exists an embedding $A \to C$ in $\fC$ that admits at least $h$ extensions to $B$.
\end{itemize}
\end{proposition}

\begin{proof}
Suppose $X(B) \to X(A)$ is an infinite cover. Let $h \ge 0$ be given. Pick $i \in X(A)$, and let $j_1, \ldots, j_h$ be distinct points of $X(B)$ mapping to $X(A)$. Thus $i \colon A \to \Omega$ is an embedding, and $j_1, \ldots, j_h$ are distinct extensions of $i$ to $B$. Let $C$ be the definable closure of the set $j_1(B) \cup \cdots \cup j_h(B)$. Then $i$ induces an embedding $A \to C$ with at least $h$ extensions to $B$, and so $(\ast)$ holds.

Now suppose that $(\ast)$ holds. Let $h \ge 0$ be given, and let $i \colon A \to C$ be an embedding with $h$ distinct extensions $j_1, \ldots, j_h$ to $B$. Fixing an embedding $C \to \Omega$, we see that $j_1, \ldots, j_h$ are $h$ distinct points of $X(B)$ mapping to $i \in X(A)$. Thus $X(B) \to X(A)$ is not a finite cover.
\end{proof}

\begin{example}
Consider the infinite symmetric group $(\fS, \Omega)$. As we have seen, structures are just finite sets. It is clear that any finite set is a homogeneous structure. The transitive $G$-set corresponding to the finite set $[n]$ is the subset $\Omega^{[n]}$ of $\Omega^n$ consisting of tuples with distinct entries. Thus $\Omega^{[n]}$ is homogeneous relative to $\sE(\Omega)$. From the classification of open subgroups of $\fS$ (see Example~\ref{ex:sym}), we see that $\sE(\Omega)$ is large and $\Omega^{[n]}$ is simply connected. Thus $\Omega^{[n]}$ is homogeneous in $\bS(G)$.
\end{example}

\subsection{Smooth approximation} \label{ss:smooth}

Let $\{\Delta_{\alpha}\}_{\alpha \in I}$ be a collection of homogeneous objects of $\bS(G)$. Let $H_{\alpha} = \Aut(\Delta_{\alpha})$, and put $\Pi_{\alpha} = \Pi_{\Delta_{\alpha}}$. For a transitive $G$-set $X$, let $\Sigma(X) \subset I$ be the collection of indices $\alpha$ for which $\Pi_{\alpha}(X)$ is non-empty. We say that $\Delta_{\bullet}$ is a \defn{smooth approximation} of $\bS(G)$ if $\Sigma(X)$ is cofinite for every transitive $G$-set $X$. We say that $G$ is \defn{smoothly approximable} if such a collection $\Delta_{\bullet}$ exists.

Suppose $(G, \Omega)$ is oligomorphic and that each $\Delta_{\alpha}$ is $\Omega$-smooth, and thus corresponds to a homogeneous structure $A_{\alpha}$. For $\{\Delta_{\alpha}\}_{\alpha \in I}$ to be a smooth approximation of $\bS(G)$, it is necessary and sufficient that each structure embeds into all but fintely many $A_{\alpha}$. Indeed, this condition exactly means that for each transitive $\Omega$-smooth $G$-set $X$, the set $\Sigma(X)$ is cofinite. If $Y$ is an arbitrary transitive $G$-set, then there is a surjection $X \to Y$ for some $\Omega$-smooth $G$-set, and so $\Sigma(Y)$ contains $\Sigma(X)$, and is thus cofinite if $\Sigma(X)$ is.

\begin{example}
The infinite symmetric group $\fS$ is smoothly approximable: the family $\{\Omega^{[n]}\}_{n \ge 1}$ is a smoothy approximation. Indeed, each finite set embeds into all but finitely many of the sets $[n]$.
\end{example}

\begin{remark}
In the definition of smoothly approximable, we demanded that $\Sigma(X)$ is cofinite for all $I$, in other words, $\Sigma(X)$ should belong to the Fr\'echet filter on $I$. One can generalize the definition by fixing a filter on $I$ refining the Fr\'echet filter, and asking that $\Sigma(X)$ belong to this filter for each $X$. We will not require this extra generality, and so work with the simpler definition.
\end{remark}

\begin{remark}
What is the relationship between the conditions ``$G$ is smoothly approximable'' and ``the c-smooth elements are dense?'' The first does not imply the second, as one sees by considering the infinite general linear group with the parabolic topology. Does the second imply the first?
\end{remark}

\subsection{Counting measures} \label{ss:count}

Let $\{\Delta_{\alpha}\}_{\alpha \in I}$ be a smooth approximation of $\bS(G)$, and maintain the notation from \S \ref{ss:smooth}. Let $\Fun(I, \bZ)$ denote the set of all functions $I \to \bZ$, and let $\Fun^{\circ}(I, \bZ)$ denote the quotient where two functions are identified if they agree on a cofinite subset of $I$. Note that $\Fun^{\circ}(I, \bZ)$ is a ring. Let $f \colon Y \to X$ be a map of transitive $G$-sets. If $\Pi_{\alpha}(X)$ is non-empty then $\Pi_{\alpha}(f) \colon \Pi_{\alpha}(Y) \to \Pi_{\alpha}(X)$ is a map of transitive $H_{\alpha}$-sets, and so all fibers have the same cardinality; let $\tilde{\nu}_{\alpha}(f)$ be this common size. If $\Pi_{\alpha}(X)$ is empty, put $\tilde{\nu}_{\alpha}(i)=0$. We thus have a function
\begin{displaymath}
\tilde{\nu}(f) \colon I \to \bZ, \qquad \alpha \mapsto \tilde{\nu}_{\alpha}(i).
\end{displaymath}
We let $\nu(f)$ be the element of $\Fun^{\circ}(I, \bZ)$ represented by $\tilde{\nu}(f)$.

\begin{proposition}
The rule $f \mapsto \nu(f)$ is a measure for $\bS(G)$ valued in the ring $\Fun^{\circ}(I, \bZ)$.
\end{proposition}

\begin{proof}
We verify the three axioms in Definition~\ref{defn:meas2}.

(a) If $f$ is an isomorphism then $\tilde{\nu}_{\alpha}(f)=1$ for all but finitely many $\alpha$, and so $\nu(f)=1$.

(b) Let $g \colon Z \to Y$ be a second map of transitive $G$-sets. If $\alpha$ is an index such that $\Pi_{\alpha}(X)$ is non-empty then we have maps of transitive $H_{\alpha}$-sets $\Pi_{\alpha}(Z) \to \Pi_{\alpha}(Y) \to \Pi_{\alpha}(X)$, and so $\tilde{\nu}_{\alpha}(fg)=\tilde{\nu}_{\alpha}(f) \tilde{\nu}_{\alpha}(g)$. Since this identity holds for all but finitely many $\alpha$, we have $\nu(fg)=\nu(f) \nu(g)$, as required.

(c) Let $f \colon Y \to X$ be a map of transitive $G$-sets, let $X' \to X$ be another map of transitive $G$-sets, and let $f' \colon Y' \to X'$ be the base change of $f$. Let $Y'=\bigsqcup_{i=1}^n Y'_i$ be the orbit decomposition of $Y'$, and let $f'_i$ be the restriction of $f'$ to $Y'_i$. Let $\alpha$ be an index such that each of $X$, $X'$, $Y$, and $Y'_i$ map to a transitive $H_{\alpha}$-set under $\Pi_{\alpha}$. The diagram
\begin{displaymath}
\xymatrix{
\Pi_{\alpha}(Y') \ar[r] \ar[d] & \Pi_{\alpha}(Y) \ar[d] \\
\Pi_{\alpha}(X') \ar[r] & \Pi_{\alpha}(X) }
\end{displaymath}
is a cartesian square of $H_{\alpha}$-sets. Counting fiber sizes, we obtain the identity
\begin{displaymath}
\tilde{\nu}_{\alpha}(f) = \sum_{i=1}^n \tilde{\nu}_{\alpha}(f'_i).
\end{displaymath}
Since this holds for all but finitely many $\alpha$, we find
\begin{displaymath}
\nu(f) = \sum_{i=1}^n \nu(f'_i),
\end{displaymath}
which completes the proof.
\end{proof}

\begin{corollary}
We have a ring homomorphism $\Theta(G) \to \Fun^{\circ}(I, \bZ)$, via $[f] \mapsto \nu(f)$.
\end{corollary}

\subsection{Ultraproduct categories} \label{ss:ultra}

Let $\{\Delta_{\alpha}\}_{\alpha \in I}$ be a smooth approximation of $\bS(G)$, and maintain the notation from \S \ref{ss:smooth}. Fix a non-principal ultrafilter $\cF$ on $I$. For each $\alpha \in I$, let $k_{\alpha}$ be a field, and let $k$ be their ultraproduct. We have a natural ring homomorphism
\begin{displaymath}
\Fun^{\circ}(I, \bZ) \to k, \qquad n \mapsto (n(\alpha))_{\alpha \in I}.
\end{displaymath}
Let $\mu$ be the measure obtained $\nu$ by composing with the above homomorphism. We call $\mu$ an \defn{ultraproduct measure}.

Let $\tilde{\cT}$ be the ultraproduct of the categories of finite dimensional $k_{\alpha}[H_{\alpha}]$-modules. This naturally a $k$-linear rigid tensor category. Suppose $X \in \bS(G)$. Then $X_{\alpha}=\Pi_{\alpha}(X)$ is a $H_{\alpha}$-set of finite cardinality. We define $\cP(X)$ to be the ultraproduct of the permutation modules $k_{\alpha}[X_{\alpha}]$, regarded as an object of $\tilde{\cT}$. We note that if $Y \in \bS(G)$ then we have natural isomorphisms
\begin{displaymath}
\cP(X \amalg Y) = \cP(X) \oplus \cP(Y), \qquad
\cP(X \times Y) = \cP(X) \otimes \cP(Y).
\end{displaymath}
We define $\cT$ to be the full subcategory of $\tilde{\cT}$ spanned by objects that are isomorphic to a subquotient of some $\cP(X)$. This is an abelian subcategory of $\tilde{\cT}$ that is closed under the tensor product.

The following is one of the central problems we study in this paper:

\begin{problem} \label{keyprob}
Compare the tensor categories $\uPerm(G, \mu)$ and $\cT$.
\end{problem}

We now establish one basic connection between these two categories. Suppose $f \colon Y \to X$ is a map in $\bS(G)$. There is then an induced function $f_{\alpha} \colon Y_{\alpha} \to X_{\alpha}$ for each $\alpha$. This induces pushforward and pullbacks on the linearizations of these sets. Taking ultraproducts, we thus have maps
\begin{displaymath}
f_* \colon \cP(Y) \to \cP(X), \qquad
f^* \colon \cP(X) \to \cP(Y).
\end{displaymath}
These constructions are both functorial, that is, if $g \colon Z \to Y$ is another map then $(gf)_*=g_*f_*$ and $(gf)^*=f^*g^*$.

\begin{proposition} \label{prop:ultra-func}
There is a natural fully faithful tensor functor
\begin{displaymath}
\Phi_0 \colon \uPerm_k(G, \mu) \to \cT, \qquad
\cC(X) \mapsto \cP(X).
\end{displaymath}
If $f \colon Y \to X$ is a morphism in $\bS(G)$ then $\Phi_0$ maps $f_*$ and $f^*$ to $f_*$ and $f^*$.
\end{proposition}

\begin{proof}
Let $\cT_0$ be the full subcategory of $\cT$ spanned by the objects $\cP(X)$. The rule $X \mapsto \cP(X)$ is a balanced functor $\Phi_{00} \colon \bS(G) \to \cT_0$ in the sense of \cite[\S 9]{repst}. One easily sees that $\Phi_{00}$ is additive, symmetric monoidal, and plenary. It follows that $\Phi_{00}$ is a linearization of $\bS(G)$ in the sense of \cite[Definition~9.8]{repst}. Thus, by \cite[Theorem~9.9]{repst}, it extends to an equivalence of tensor categories
\begin{displaymath}
\Phi_0 \colon \uPerm_k(G, \mu') \to \cT_0
\end{displaymath}
for some measure $\mu'$. To show that $\mu'=\mu$, it suffices to show that $\Phi_{00}$ is $\mu$-adapted.

Suppose $f \colon Y \to X$ is a map of transitive objects in $\bS(G)$. Let $n_{\alpha}$ be the common fiber size of the map $f_{\alpha} \colon Y_{\alpha} \to X_{\alpha}$. Then the map
\begin{displaymath}
(f_{\alpha})_* f_{\alpha}^* \colon k_{\alpha}[Y_{\alpha}] \to k_{\alpha}[X_{\alpha}]
\end{displaymath}
is multiplication by $n_{\alpha}$. It follows that $f_* f^*$ is multiplication by the element $c=(n_{\alpha})$ of $k$. We have $\mu(f)=c$ by definition of $\mu$. This shows that $\Phi_{00}$ is $\mu$-adapted, and completes the proof.
\end{proof}

Since nilpotent endomorphisms in $\cT$ have trace zero, we obtain the following corollary:

\begin{corollary} \label{cor:ultra-nil}
The ultraproduct measure $\mu$ satisfies condition (Nil), that is, nilpotent endomorphisms in the category $\uPerm_k(G, \mu)$ have trace zero.
\end{corollary}

\subsection{Abelian envelopes} \label{ss:ultra-ab}

Maintain the above situation, and suppose that $\mu$ is quasi-regular. Since (Nil) holds (Corollary~\ref{cor:ultra-nil}), it follows that $\uRep(G, \mu)$ is the abelian envelope of $\uPerm(G, \mu)$. In particular, $\Phi_0$ induces an exact tensor functor
\begin{displaymath}
\Phi \colon \uRep(G, \mu) \to \cT.
\end{displaymath}
We would like this functor to be an equivalence; this is, in a sense, the ideal solution to Problem~\ref{keyprob}. The rest of \S \ref{s:ultra}, we give some criteria for this. We begin with two simple observations.

\begin{proposition} \label{prop:Phi-full-faithful}
The functor $\Phi$ is fully faithful.
\end{proposition}

\begin{proof}
This follows since $\Phi_0$ is fully faithful, $\Phi$ is exact, and every object of $\uRep(G, \mu)$ admits a presentation and co-presentation by Schwartz spaces. We provide some details. Let $M$ and $N$ be objects of $\uRep(G, \mu)$. We must show that the map
\begin{equation} \label{eq:Phi}
\Hom(M, N) \to \Hom(\Phi(M), \Phi(N))
\end{equation}
is an isomorphism. First suppose that $N=\cC(Y)$ for some finitary $G$-set $Y$. Choose a presentation
\begin{displaymath}
\cC(X_1) \to \cC(X_0) \to M \to 0.
\end{displaymath}
Then we have the following commutative diagram
\begin{displaymath}
\xymatrix{
0 \ar[r] & \Hom(M, \cC(Y)) \ar[r] \ar[d] & \Hom(\cC(X_0), \cC(Y)) \ar[r] \ar[d] & \Hom(\cC(X_1), \cC(Y)) \ar[d] \\
0 \ar[r] & \Hom(\Phi(M), \cP(Y)) \ar[r] & \Hom(\cP(X_0), \cP(Y)) \ar[r] & \Hom(\cP(X_1), \cP(Y)). }
\end{displaymath}
The right two vertical arrows are isomorphisms since $\Phi_0$ is fully faithful. Thus the left vertical arrow is an isomorphism. We have thus shown that \eqref{eq:Phi} is an isomorphism whenever $N$ is a Schwartz space. We can now deduce the result for general $N$ by choosing a co-presentation of $N$ by Schwartz spaces and employing a similar argument.
\end{proof}

\begin{proposition} \label{prop:Phi-equiv}
The following are equivalent:
\begin{enumerate}
\item $\Phi$ is an equivalence.
\item Some $\cP(X)$ is projective in $\cT$.
\item Every object of $\cT$ is a quotient of some $\cP(X)$.
\end{enumerate}
\end{proposition}

\begin{proof}
(a) $\Rightarrow$ (b). Since $\mu$ is quasi-regular, there is an open subgroup $U$ such that $\mu$ is regular on $U$, and so $\uRep(U, \mu)$ is semi-simple. It follows that if $M$ is any object of $\uRep(U, \mu)$ then its induction to $G$ is a projective object of $\uRep(G, \mu)$. Applying this with $M=\bone$, we see that $\cC(G/U)$ is projective. Since $\Phi$ is an equivalence, we find that $\cP(G/U)$ is projective.

(b) $\Rightarrow$ (c). Let $M$ be an object of $\cT$, and write $M=M_0/M_1$, where $M_1 \subset M_0 \subset \cP(Y)$, and $Y$ is some finitary $G$-set. Since $\cT$ is a rigid tensor category and $\cP(X)$ is projective, it follows that $M_0 \otimes \cP(X)$ is projective, and also injective. Since this is an injective subobject of $\cP(X) \otimes \cP(Y)=\cP(X \times Y)$, it is a summand. Thus $\cP(X) \otimes M_0$ is a quotient of $\cP(X \times Y)$. Since we have surjections
\begin{displaymath}
\cP(X) \otimes M_0 \to \cP(X) \otimes M \to M,
\end{displaymath}
we see that $M$ is a quotient of $\cP(X \times Y)$, as required. (The second map above comes from tensoring the surjection $\cP(X) \to \bone$ with $M$.)

(c) Let $M$ be an object of $\cT$. We then have a presentation
\begin{displaymath}
\cP(X_1) \stackrel{f}{\to} \cP(X_0) \to M \to 0.
\end{displaymath}
Indeed, we can find a surjection $g \colon \cP(X_0) \to M$, and we can then find a surjection from some $\cP(X_1)$ onto $\ker(g)$. Write $f=\Phi_0(f')$, where $f' \colon \cC(X_1) \to \cC(X_0)$. Then $M=\Phi(M')$, where $M'$ is the cokernel of $f'$. Thus $\Phi$ is essentially surjective. It is also fully faithful (Proposition~\ref{prop:Phi-full-faithful}), and therefore an equivalence.
\end{proof}

\subsection{The regular case} \label{ss:ultra-ab-reg}

Maintain the set-up of \S \ref{ss:ultra-ab}. We now give some criterion for $\Phi$ to be an equivalence when the measure $\mu$ is regular. We begin with the following:

\begin{proposition} \label{prop:ultra-equiv1}
Suppose that $\mu$ is regular and the following condition holds:
\begin{itemize}
\item[$(\ast)$] Whenever $L$ is a non-trivial simple object of $\uRep(G, \mu)$, the object $\Phi(L)$ of $\cT$ has no trivial subquotient.
\end{itemize}
Then $\Phi$ is an equivalence.
\end{proposition}

\begin{proof}
We first claim that the trivial object of $\cT$ is projective and injective. Indeed, suppose we have a surjective map $\epsilon \colon M \to \bone$ in $\cT$. By definition, there is some $X \in \bS(G)$ and subobjects $M_2 \subset M_1 \subset \cP(X)$ such that $M=M_1/M_2$. Write $\cC(X) = \bigoplus_{i=1}^n L_i$, where the $L_i$'s are simples of $\uRep(G, \mu)$, and order the $L_i$'s so that $L_1, \ldots, L_r$ are trivial and $L_{r+1}, \ldots, L_n$ are non-trivial. We thus have $\cP(X) = P' \oplus P''$, where $P'$ is a sum of trivials and $P''=\bigoplus_{i=r+1}^n \Phi(L_i)$. Let $M'$ be the image of $M_1 \cap P'$ in $M$. Then $M/M'$ is a subquotient of $P''$, and thus has no trivial subquotient by $(\ast)$. Now, if $\epsilon \vert_{M'}$ were~0 then $\epsilon$ would factor through $M/M'$, and thus be~0, which is not possible since $\epsilon$ is surjective. Thus $\epsilon_{M'}$ is non-zero, and this gives a splitting of $\epsilon$. This shows that $\bone$ is projective $\cT$, and so $\Phi$ is an equivalence (Proposition~\ref{prop:Phi-equiv}).
\end{proof}

We now give a way of verifying the condition in Proposition~\ref{prop:ultra-equiv1} using the Farahat--Higman algebra $Z$ defined in \S \ref{ss:fh}. We begin with the following observation.

\begin{proposition}
The category $\cT$ and the functor $\Phi$ are naturally $Z$-linear.
\end{proposition}

\begin{proof}
Let $\GG_{\alpha} = \Hom_G(\Delta_{\alpha}, \GG)$. Choosing a base point in $\Delta_{\alpha}$, we have $\Delta_{\alpha}=G/U_{\alpha}$ for some open subgroup $U_{\alpha}$, and we find
\begin{displaymath}
H_{\alpha} = N_G(U_{\alpha})/U_{\alpha}, \qquad
\GG_{\alpha} = Z_G(U_{\alpha}).
\end{displaymath}
It follows that $H_{\alpha}$ acts on $\GG_{\alpha}$ (by conjugation), and there is a $H_{\alpha}$-equivariant group homomorphism $\GG_{\alpha} \to H_{\alpha}$, using the conjugation action on the target. Moreover, one can verify that this structure is independent of the choice of base point of $\Delta_{\alpha}$, and is thus natural.

Write $\GG=\varinjlim_{\beta \in J} \GG^{\le \beta}$, where $J$ is a directed set and each $\GG^{\le \beta}$ is a finitary $G$-set, and let $\GG^{\le \beta}_{\alpha}=\Pi_{\alpha}(\GG^{\le \beta})$. Then $\GG_{\alpha}$ is the directed union of the $\GG^{\le \beta}_{\alpha}$, each of which is $H_{\alpha}$-stable. Let $C^{\le \beta}=\cP(\GG^{\le \beta})$, which is the ultraproduct of the spaces $k_{\alpha}[\GG^{\le \beta}_{\alpha}]$ as $\alpha$ varies, and let $C$ be the direct limit (over $\beta$) of these spaces, thought of as an ind-object of $\cT$; one can regard this as the ``bounded ultraproduct'' of the group algebras $k_{\alpha}[\GG_{\alpha}]$. The multiplication on $\GG$ endows $C$ with an algebra structure.

Let $D$ be the ultraproduct of the group algebras $k_{\alpha}[H_{\alpha}]$, and let $D^{\circ}$ be the ultraproduct of the centers of these group algebras, which is a central subalgebra of $D$. The category $\cT$ is $D^{\circ}$-linear. Now, the homomorphism $\GG_{\alpha} \to H_{\alpha}$ induces a homomorphism $C \to D$. The space $\Hom_{\cT}(\bone, C^{\le \beta})$ is the ultraproduct of the $H_{\alpha}$-invariants in $k_{\alpha}[\GG^{\le \beta}_{\alpha}]$, and maps into $D^{\circ}$. Thus $\cT$ is linear over the algebra $\Hom_{\cT}(\bone, C)$.

It follows from the definition of $\Phi$ that $\Phi(B)=C$ (as algebras), where here we are applying $\Phi$ to ind-objects in the natural manner. Since $\Phi$ is fully faithful, we have $Z=\Hom(\bone, B)=\Hom_{\cT}(\bone, C)$. Thus, composing with the above homomorphism, we obtain a natural algebra homomorphism $Z \to D^{\circ}$. This shows that $\cT$ is $Z$-linear.

We now verify that $\Phi$ is $Z$-linear. Let $X$ be an object of $\bS(G)$. Under the identification $\Delta_{\alpha}=G/U_{\alpha}$, we have $X_{\alpha}=X^{U_{\alpha}}$. Thus $H_{\alpha}$ and $\GG_{\alpha}$ naturally act on this set, with the action of $\GG_{\alpha}$ factoring through the homomorphism $\GG_{\alpha} \to H_{\alpha}$; moreover, all this structure is independent of the choice of basepoint of $\Delta_{\alpha}$. We thus see that the action of $\Hom_{\cT}(\bone, C)$ on $\cP(X)$ factors through the homomorphism $\Hom_{\cT}(\bone, C) \to D^{\circ}$. The result now follows.
\end{proof}

The following is our main conditional result on Problem~\ref{keyprob} in the regular case.

\begin{theorem} \label{thm:ultra-equiv2}
Suppose that $\mu$ is regular and condition $(\ast)$ from Proposition~\ref{prop:cent-char} holds. Then $\Phi$ is an equivalence.
\end{theorem}

\begin{proof}
Suppose $L$ is a non-trivial simple object of $\uRep(G, \mu)$. Then $Z$ acts through a non-trivial character on $\Phi(L)$ by Proposition~\ref{prop:cent-char}. It follows that $\Phi(L)$ has no subquotient isomorphic to the trivial representation. Thus condition~$(\ast)$ of Proposition~\ref{prop:ultra-equiv1} holds, and so $\Phi$ is an equivalence.
\end{proof}

\subsection{The quasi-regular case}

Maintain the set-up of \S \ref{ss:ultra-ab}. We now give a method for proving that $\Phi$ is an equivalence when the measure $\mu$ is quasi-regular.

Let $G'$ be an open subgroup of $G$. Recall that $\{\Delta_{\alpha}\}_{\alpha \in I}$ is our smooth approximation of $G$. Suppose that we have a smooth approximation $\{\Delta'_{\alpha}\}_{\alpha \in I}$ of $G'$ indexed by the same set $I$. We say that the smooth approximations of $G$ and $G'$ are \defn{compatible} if for each $\alpha \in I$ we have an isomorphism $\Ind_{G'}^G(\Delta'_{\alpha}) \to \Delta_{\alpha}$ of $G$-sets, where $\Ind_{G'}^G(X) = G \times_{G'} X$. Concretely, this means that we have an open subgroup $U_{\alpha}$ of $G'$ such that $\Delta_{\alpha} \cong G/U_{\alpha}$ and $\Delta'_{\alpha} \cong G'/U_{\alpha}$. We let $H'_{\alpha}=\Aut_{G'}(\Delta'_{\alpha})$ and define $\Pi'_{\alpha}$ analogously to the $G$-case. We assume this compatibility condition in what follows.

\begin{proposition} \label{prop:comp-smooth}
We have the following:
\begin{enumerate}
\item Let $X$ be a finitary $G$-set. Then the natural map
\begin{displaymath}
\Res^{H_{\alpha}}_{H'_{\alpha}}(\Pi_{\alpha}(X)) \to \Pi'_{\alpha}(\Res^G_{G'}(X))
\end{displaymath}
is an isomorphism for all $\alpha \in I$.
\item Let $Y$ be a finitary $G'$-set. Then the natural map
\begin{displaymath}
\Ind_{H'_{\alpha}}^{H_{\alpha}}(\Pi'_{\alpha}(Y)) \to \Pi_{\alpha}(\Ind_{G'}^G(Y))
\end{displaymath}
is an isomorphism for all but finitely many $\alpha \in I$.
\end{enumerate}
\end{proposition}

\begin{proof}
(a) We have natural identifications
\begin{displaymath}
\Pi_{\alpha}(X) = \Hom_G(\Delta_{\alpha}, X) = \Hom_{G'}(\Delta'_{\alpha}, X) = \Pi'_{\alpha}(X),
\end{displaymath}
where in the second step we used the adjunction between restriction and induction. This identification is $H'_{\alpha}$-equivariant, and so the first statement follows.

(b) It suffices to treat the case where $Y$ is transitive. Suppose $\alpha$ is such that $\Pi'_{\alpha}(Y)$ is non-empty. We then show that the map in question is an isomorphism.

Write $\Delta'_{\alpha}=G'/U$ and $\Delta_{\alpha}=G/U$. Since $\Pi'_{\alpha}(Y)=Y^U$ is non-empty, we can choose an isomorphism $Y=G'/V$ with $U \subset V$; we thus have $\Ind_{G'}^G(Y)=G/V$. Since $H'_{\alpha}=N_{G'}(U)/U$ acts transitively on $\Pi'_{\alpha}(Y)$, we have $\Pi'_{\alpha}=N_{G'}(U)/V$; similarly, $\Pi_{\alpha}(Y)=N_G(U)/V$. We also have
\begin{displaymath}
\Ind_{H_{\alpha}'}^{H_{\alpha}}(\Pi'_{\alpha}(Y)) = H_{\alpha} \times_{H'_{\alpha}} N_{G'}(U)/V = N_G(U)/V = \Pi_{\alpha}(Y).
\end{displaymath}
It is now clear that the map in question is an isomorphism. Indeed, the domain and target are abstractly isomorphic transitive $H_{\alpha}$-sets. Since the map is $H_{\alpha}$-equivariant, it is surjective, and thus injective since the domain and target have the same cardinality.
\end{proof}

Since we have a a smooth approximation of $G'$, we get ultraproduct categories $\tilde{\cT}'$ and $\cT'$ by the general constructions of \S \ref{ss:ultra-ab}. For a finitary $G'$-set $Y$, we write $\cP'(Y)$ for the corresponding object of $\cT'$. Since the induction functor $\Ind_{G'}^G$ is faithful, it follows that $H'_{\alpha}$ is a subgroup of $H_{\alpha}$. We therefore have restriction and induction functors between their representation categories. These induce exact functors
\begin{displaymath}
\Res \colon \tilde{\cT} \to \tilde{\cT}', \qquad
\Ind \colon \tilde{\cT}' \to \tilde{\cT}.
\end{displaymath}
We now investigate them.

\begin{proposition}
We have the following:
\begin{enumerate}
\item If $X$ is a finitary $G$-set then we have a natural isomorphism $\Res(\cP(X)) = \cP'(\Res^G_{G'}(X))$.
\item The restriction functor maps $\cT$ into $\cT'$.
\item If $Y$ is a finitary $G'$-set then we have a natural isomorphism $\Ind(\cP'(Y)) = \cP(\Ind_{G'}^G(Y))$.
\item The induction functor maps $\cT'$ into $\cT$.
\end{enumerate}
\end{proposition}

\begin{proof}
Statements (a) and (c) follow from Proposition~\ref{prop:comp-smooth} and the way induction and restriction work for permutation representations of finite groups. Since restriction is exact, it takes subquotients of $\cP(X)$ to subquotients of $\cP'(\Res_{G'}^G(X))$, and thus maps $\cT$ into $\cT'$. The same argument works for induction.
\end{proof}

\begin{corollary} \label{cor:ultra-quot}
Suppose that every object of $\cT'$ is a quotient of some $\cP'(Y)$. Then every object of $\cT$ is a quotient of some $\cP(X)$.
\end{corollary}

\begin{proof}
Let $M$ be an object of $\cT$. By assumption, we have a surjection $\cP'(Y) \to \Res(M)$ for some finitary $G'$-set $Y$. Since induction is exact, we have a surjection
\begin{displaymath}
\cP(X) = \Ind(\cP'(Y)) \to \Ind(\Res(M)),
\end{displaymath}
where $X=\Ind_{G'}^G(Y)$. Since the co-unit $\Ind(\Res(M)) \to M$ is surjective (since this is true for finite groups), the result follows.
\end{proof}

We now connect the above discussion to the $\uRep$ categories. First we need a result on measures. By the general theory of smooth approximations, we have measures $\nu$ and $\nu'$ for $G$ and $G'$ valued in the ring $\Fun^{\circ}(I, \bZ)$.

\begin{proposition}
The measure $\nu'$ is the restriction of the measure $\nu$ to $G'$.
\end{proposition}

\begin{proof}
Define $\nu^*$ to be the restriction of $\nu$ to $G'$. Let $f' \colon Y' \to X'$ be a map of transitive $G'$-sets, and let $f \colon Y \to X$ be the induction of $f'$ from $G'$ to $G$, which is a map of transitive $G$-sets. Then $\nu^*(f')=\nu(f)$; this is simply the definition of restriction of measures from the perspective of Definition~\ref{defn:meas2}. By Proposition~\ref{prop:comp-smooth}, $\Pi_{\alpha}(f)$ is the induction from $H'_{\alpha}$ to $H_{\alpha}$ of $\Pi'_{\alpha}(f')$, for all but finitely many $\alpha$. It follows that the fiber size of $\Pi_{\alpha}(f)$ agrees with the fiber size of $\Pi'_{\alpha}(f')$, for almost all $\alpha$. Since the former computes $\nu^*(f')$ and the latter $\nu'(f')$, we find $\nu^*=\nu'$, as required.
\end{proof}

Let $\mu'$ be the measure obtained from $\nu'$ by applying the ring homomorphism $\Fun^{\circ}(I, \bZ) \to k$ corresponding to our chosen ultrafilter $\cF$. By the above proposition, this coincides with the restriction of $\mu$ to $G'$. The same construction as in \S \ref{ss:ultra-ab} gives a tensor functor
\begin{displaymath}
\Phi' \colon \uRep(G', \mu') \to \cT'.
\end{displaymath}
We now have the following diagram:
\begin{displaymath}
\xymatrix@C=3em{
\uRep(G, \mu) \ar[r]^-{\Phi} \ar[d] & \cT \ar[d] \\
\uRep(G', \mu') \ar[r]^-{\Phi'} & \cT' }
\end{displaymath}
It is not difficult to show that the square commutes (up to isomorphism), though we will not need this. Our main result is the following:

\begin{proposition} \label{prop:ultra-equiv3}
If $\Phi'$ is an equivalence then so is $\Phi$.
\end{proposition}

\begin{proof}
If $\Phi'$ is an equivalence then every object of $\cT'$ is a quotient of some $\cP'(X)$ (Proposition~\ref{prop:Phi-equiv}). Thus the same holds in $\cT$ (Corollary~\ref{cor:ultra-quot}), and so $\Phi$ is an equivalence (Proposition~\ref{prop:Phi-equiv} again).
\end{proof}

The above result is useful for showing that $\Phi$ is an equivalence when $\mu$ is quasi-regular: one chooses $G'$ so that $\mu'$ is regular, then shows that $\Phi'$ is an equivalence using the methods of \S \ref{ss:ultra-ab-reg}, and finally deduces from the above proposition that $\Phi$ is an equivalence. For this approach to work, one must be able to choose $G'$ that has a compatible smooth approximation. We do not know if this is possible in general, but it will be possible in all cases of interest in this paper.

\part{Examples} \label{part:ex}

\section{The general linear group: combinatorics} \label{s:gl}

\subsection{Set-up} \label{ss:GL-setup}

Fix a finite field $\bF$ of cardinality $q$. Define $\hat{\fC}$ to be the following category. Objects are $\bZ/2$-graded $\bF$-vector spaces $V=V_0 \oplus V_1$ equipped with a bilinear form $\langle, \rangle \colon V_0 \times V_1 \to \bF$. A morphism $f \colon V \to W$ is an injective linear map that respects the gradings and forms, i.e., we have
\begin{displaymath}
f(V_i) \subset W_i, \qquad \langle f(x), f(y) \rangle=\langle x, y \rangle.
\end{displaymath}
Let $\fC$ be the subcategory spanned by finite dimensional objects.

Let $\bV$ be a vector space over $\bF$ with basis $\{e_i, f_i\}_{i \ge 1}$. Let $\bV_0$ be the subspace spanned by the $e_i$'s, and let $\bV_1$ be the subspace spanned by the $f_i$'s. Define a pairing between $\bV_0$ and $\bV_1$ by
\begin{displaymath}
\langle e_i, f_j \rangle = \delta_{i,j}.
\end{displaymath}
In this way, $\bV$ is an object of $\hat{\fC}$. We let $G$ be its automorphism group. Concretely, $G$ is the subgroup of $\GL(\bV_0)$ consisting of matrices that are both row and column finite. The group $G$ acts oligomorphically on $\bV$. The category of structures for $(G, \bV)$, as defined in \S \ref{ss:struct}, is the category $\fC$ above. The key point is that the definable closure of a finite subset of $\bV$ is the linear span of the homogeneous pieces of its elements.

We let $V_{\ell,m,n}$ be a $\bZ/2$-graded vector space over $\bF$ of dimension $(\ell+m|\ell+n)$ equipped with a form that is a perfect pairing on the $(\ell|\ell)$ piece, and vanishes on the $(m|n)$ piece. We let $X_{\ell,m,n}$ be the space of embeddings $V_{\ell,m,n} \to \bV$. One easily sees that every object of $\fC$ is isomorphic to a unique $V_{\ell,m,n}$, and so the $X_{\ell,m,n}$ are exactly the $\bV$-smooth transitive $G$-sets by Proposition~\ref{prop:struct-equiv}.

\subsection{Enumeration} \label{ss:GL_enum}

Let $N_{a,b,c}^{\ell,m,n}$ be the number of embeddings of $V_{a,b,c}$ into $V_{\ell,m,n}$.

\begin{proposition} \label{prop:GL-recur}
We have the following recurrences for $N$:
\begin{align*}
N_{a,b,c}^{\ell,m,n} &= (q^{\ell}-1) q^{\ell+m+n-1} \cdot N_{a-1,b,c}^{\ell-1,m,n} \\
N_{0,b,c}^{\ell,m,n} &= (q^{\ell+m}-q^m) q^{b-1} N_{0,b-1,c}^{\ell-1,m,n} + (q^m-1) q^{b-1} N_{0,b-1,c}^{\ell,m-1,n} \\
N_{0,b,c}^{\ell,m,n} &= (q^{\ell+n}-q^n) q^{c-1} N_{0,b,c-1}^{\ell-1,m,n} + (q^n-1) q^{c-1} N_{0,b,c-1}^{\ell,m,n-1}.
\end{align*}
We require $a \ge 1$, $b \ge 1$, and $c \ge 1$ respectively.
\end{proposition}

\begin{proof}
(a) Giving an embedding of $V_{a,b,c}$ into $V_{\ell,m,n}$ is equivalent to giving an embedding $i$ of $V_{1,0,0}$ into $V_{\ell,m,n}$ together with an embedding of $V_{a-1,b,c}$ into the orthogonal complement of the image of $i$. The latter is isomorphic to $V_{\ell-1,m,n}$. We thus find
\begin{displaymath}
N_{a,b,c}^{\ell,m,n} = N_{1,0,0}^{\ell,m,n} \cdot N_{a-1,b,c}^{\ell-1,m,n},
\end{displaymath}
since there are $N_{1,0,0}^{\ell,m,n}$ choices for $i$. To complete the proof of this identity, we must show
\begin{displaymath}
N_{1,0,0}^{\ell,m,n} = (q^{\ell}-1) q^{\ell+m+n-1}.
\end{displaymath}
This is clear if $\ell=0$ since both sides vanish. Thus suppose $\ell \ge 1$. Let $x$ and $y$ be the even and odd basis vectors of $V_{1,0,0}$. We can map $x$ to any even vector in $V_{\ell,m,n}$ not in the nullspace, which gives $q^{\ell+m}-q^m$ choices. We can then map $y$ to any odd vector that pairs to~1 with $x$. The set of such vectors is a torsor for the orthogonal complement of $x$, which has dimension $\ell+n-1$, and so the result follows.

(b) Write $V_{0,b,c}=V_{0,1,0} \oplus V_{0,b-1,c}$, and let $x$ be a basis vector for $V_{0,1,0}$. Given an embedding $i$ of $V_{0,b,c}$ into $V_{\ell,m,n}$, we let $y$ be the image of $x$, and let $j$ be the restriction of $i$ to $V_{0,b-1,c}$. The embedding $i$ is determined by the pair $(y,j)$. To count the choices of $i$, we count the choices of $(y,j)$ by first picking $y$.

First suppose that $y$ is a null vector; there are $q^m-1$ choices for such $y$. The pair $(y,j)$ then comes from an $i$ if and only if the induced map $\ol{j} \colon V_{0,b-1,c} \to V_{\ell,m,n}/\bF y$ is an embedding. This quotient space is isomorphic to $V_{\ell,m-1,n}$, and so there are $N_{0,b-1,c}^{\ell,m-1,n}$ choices for $\ol{j}$. Each choice of $\ol{j}$ admits $q^{b-1}$ lifts to a choice of $j$. We thus find that there are $(q^m-1) q^{b-1} N_{0,b-1,c}^{\ell,m-1,n}$ choices of $i$ in this case.

Now suppose that $y$ is not a null vector; there are $q^{\ell+m}-q^m$ choices for such $y$. Then $j$ must map into the orthogonal complement $V'_{\ell,m,n}$ of $y$. The pair $(y,j)$ comes from an $i$ if and only if the induced map $\ol{j} \colon V_{0,b-1,c} \to V'_{\ell,m,n}/\bF y$ is an embedding. This quotient space is isomorphic to $V_{\ell-1,m,n}$. Continuing as above, we find $(q^{\ell+m}-q^m) q^{b-1} N_{0,b-1,c}^{\ell-1,m,n}$ choices in this case. The formula thus follows.

(c) This follows from the previous formula by symmetry.
\end{proof}

We now define polynomials $Q_{a,b,c} \in \bZ_q[t,u,v]$ by putting $Q_{0,0,0}=1$ and imposing the following recurrences (valid for $a \ge 1$, $b \ge 1$, and $c \ge 1$ respectively):
\begin{align*}
Q_{a,b,c}(t,u,v) &= q^{-1} (t-1) tuv \cdot Q_{a-1,b,c}(q^{-1} t,u,v) \\
Q_{0,b,c}(t,u,v) &= q^{b-1} (t-1) u \cdot Q_{0,b-1,c}(q^{-1}t, u, v) + q^{b-1} (u-1) \cdot Q_{0,b-1,c}(t,q^{-1}u,v) \\
Q_{0,b,c}(t,u,v) &= q^{c-1} (t-1) v \cdot Q_{0,b,c-1}(q^{-1}t, u, v) + q^{c-1} (v-1) \cdot Q_{0,b,c-1}(t,u,q^{-1}v)
\end{align*}
For instance, we have
\begin{displaymath}
Q_{1,0,0} = q^{-1}(t-1)tuv, \qquad
Q_{0,1,0} = tu-1, \qquad
Q_{0,0,1} = tv-1.
\end{displaymath}
The next result follows easily from Proposition~\ref{prop:GL-recur}.

\begin{proposition} \label{prop:GL-Qpoly}
The polynomials $Q_{a,b,c}$ are well-defined. Moreover, we have
\begin{displaymath}
N_{a,b,c}^{\ell,m,n} = Q_{a,b,c}(q^{\ell}, q^m, q^n)
\end{displaymath}
for all $\ell,m,n \ge 0$.
\end{proposition}

We define polynomials $P_{a,b,c} \in \bZ_q[t]$ by
\begin{displaymath}
P_{a,b,c}(t) = Q_{a,b,c}(t,1,1).
\end{displaymath}
Since these polynomials will be especially important, we compute them explicitly:

\begin{proposition} \label{prop:GL-Ppoly}
We have
\begin{displaymath}
P_{a,b,c}(t) = q^{-(a+b)(a+c)} t^a \prod_{i=0}^{a+b+c-1} (t-q^i).
\end{displaymath}
\end{proposition}

\begin{proof}
Plugging $u=v=1$ into the recurrences for $Q$, we find
\begin{align*}
P_{a,b,c}(t) &= q^{-1}(t-1)t \cdot P_{a-1,b,c}(q^{-1} t) \\
P_{0,b,c}(t) &= q^{b-1} (t-1) \cdot P_{0,b-1,c}(q^{-1} t) \\
P_{0,b,c}(t) &= q^{c-1} (t-1) \cdot P_{0,b,c-1}(q^{-1} t)
\end{align*}
The result now follows by straightforward calculation.
\end{proof}

\subsection{The Burnside ring}

Let $\BB=\BB(G, \bV)$ be the relative Burnside ring for $G$. Define $x_{a,b,c}$ to be the class $\lbb X_{a,b,c} \rbb$ in $\BB$. These elements form a $\bZ$-basis of $\BB$. The following proposition describes $\BB$ as ring.

\begin{proposition} \label{prop:GL-burnside}
We have a ring isomorphism
\begin{displaymath}
\phi \colon \bZ_q[X, Y, Z] \to \BB[1/q], \qquad X \mapsto x_{1,0,0}, \quad Y \mapsto x_{0,1,0}, \quad Z \mapsto x_{0,0,1}
\end{displaymath}
\end{proposition}

\begin{proof}
Recall from \S \ref{ss:burnside} the notion of level, and the associated filtration on $\BB$. If $V$ is an object of $\fC$ of dimension $n$, then $X(V)$ appears as an orbit on $\bV^n$, and so $X(V)$ has level $\le n$. It follows that $F^n \BB$ is the $\bZ$-span of the classes $x_{a,b,c}$ with $2a+b+c \le n$. Define a filtration on $\bZ[X,Y,Z]$ by taking $F^n \bZ[X,Y,Z]$ to be the $\bZ$-span of those monomials $X^aY^bZ^c$ with $2a+b+c \le n$. Then $\phi$ is a homomorphism of filtered rings. Since the filtered pieces of the two rings are free $\bZ_q$-modules of the same finite rank, it suffices to show that $\phi$ is injective modulo $p'$ for all primes $p' \nmid q$. Fix such a $p'$.

Let $N_{a,b,c}^*$ be the function on $\bN^3$ defined by $(\ell,m,n) \mapsto N_{a,b,c}^{\ell,m,n}$. The marks homomorphisms for the objects $V_{\ell,m,n}$ furnish a ring homomorphism
\begin{displaymath}
\BB \to \Fun(\bN^3, \bZ), \qquad x_{a,b,c} \mapsto N_{a,b,c}^*.
\end{displaymath}
By our results relating $N$ and $Q$, it follows that the additive map
\begin{displaymath}
\psi \colon \BB \to \bZ_q[t,u,v], \qquad x_{a,b,c} \mapsto Q_{a,b,c}(t,u,v)
\end{displaymath}
is a ring homomorphism. Since the polynomials $Q_{1,0,0}$, $Q_{0,1,0}$, and $Q_{0,0,1}$ are algebraically independent modulo $p'$, it follows that  $\psi \circ \phi$ is injective modulo $p'$, and so $\phi$ is as well.
\end{proof}

\subsection{The infinitesimal Burnside ring}

Let $\hat{\BB}$ be the infinitesimal Burnside ring for $G$ relative to $\sE(\bV)$. Let $G(n) \subset G$ be the open subgroup fixing each $e_i$ and $f_i$ for $1 \le i \le n$. Let $\bV(n) \subset \bV$ be the span of the vectors $e_i$ and $f_i$ with $i>n$, on which $G(n)$ naturally acts. The pair $(G(n), \bV(n))$ is isomorphic to $(G, \bV)$ via shifting basis vectors. We define
\begin{displaymath}
\delta \colon \BB \to \BB
\end{displaymath}
to be the ring homomorphism induced by restricting to $G(1)$, and identifying $G(1)$ with $G$.

\begin{proposition}
We have
\begin{align*}
\delta(x_{1,0,0}) &= qx_{1,0,0} + (q-1)(x_{0,1,0}+1)(x_{0,0,1}+1) \\
\delta(x_{0,1,0}) &= qx_{0,1,0} + (q-1) \\
\delta(x_{0,0,1}) &= qx_{0,0,1} + (q-1)
\end{align*}
In particular, $\delta$ is an automorphism after inverting $q$.
\end{proposition}

\begin{proof}
As a $G(1)$-set, we have $\bV_0 = \bF \times \bV_0(1)$, which shows that $\delta([\bV_0])=q [\bV_0]$. As $[\bV_0]=1+x_{0,1,0}$, the second formula follows. The third is similar.

For $a \in \bF$, let $Y_a$ be the subset of $\bV_0 \times \bV_1$ consisting of pair $(u,v)$ such that $\langle u, v \rangle=a$. We note that the sets $Y_a$ with $a \ne 0$ are all isomorphic to one another. Also, the complement of $Y_0$ in $\bV_0 \times \bV_1$ is the disjoint union of the $Y_a$ with $a \ne 0$, and thus isomorphic to $q-1$ copies of $Y_1$. Write $Y_a(1)$ for the analogous set defined with respect to $\bV(1)$. 

The set $X_{1,0,0}$ is identified with $Y_1$. This set is further identified with the subset of $\bF^2 \times \bV_0(1) \times \bV_1(1)$ consisting of tuples $(a, b, u, v)$ such that $ab+\langle u, v \rangle = 1$. We thus have an isomorphism
\begin{displaymath}
Y_1 \cong \coprod_{a,b \in \bF} Y_{1-ab}(1)
\end{displaymath}
of $G(1)$-sets. There are $q-1$ choices of $(a,b)$ where $1-ab=0$. In the Burnside ring of $G(1)$, we thus have
\begin{displaymath}
[Y_1] = (q-1) [Y_0(1)] + (q^2-q+1) [Y_1(1)].
\end{displaymath}
Since $[Y_0]=[\bV_0][\bV_1]-(q-1)[Y_1]$, and analogously for the $(1)$ versions, we find
\begin{displaymath}
[Y_1] = (q-1) [\bV_0(1)][\bV_1(1)] + q[Y_1(1)].
\end{displaymath}
The result now follows.
\end{proof}

The restriction map
\begin{displaymath}
\BB(G(n), \bV(n)) \to \BB(G(n+1), \bV(n+1))
\end{displaymath}
is isomorphic to $\delta$, after identifying the source and target with $\BB$. In particular, it is an isomorphism after inverting $q$. Since $\hat{\BB}$ is the direct limit of the above maps, we find:

\begin{corollary}
The natural map $\BB[1/q] \to \hat{\BB}[1/q]$ is an isomorphism.
\end{corollary}

\begin{remark} \label{rmk:GL-grading}
Define elements of $\BB[1/q]$ as follows:
\begin{displaymath}
w_1=x_{1,0,0}-q^{-1}w_2 w_3, \quad
w_2 = x_{0,1,0}+1, \quad w_3 = x_{0,0,1}+1.
\end{displaymath}
Then $w_1$, $w_2$, and $w_3$ are eigenvectors of $\delta$ of eigenvalue $q$. Since $\BB[1/q]$ is isomorphic to the polynomial ring $\bZ_q[w_1,w_2,w_3]$, it follows that $\delta$ induces a canonical grading on $\BB[1/q]$: the degree $n$ piece is the $q^n$-eigenspace of $\delta$.
\end{remark}

\subsection{Classification of transitive sets}

In \S \ref{ss:GL-setup}, we say that the $X_{\ell,m,n}$ account for all $\bV$-smooth transitive $G$-sets. We now determine all (smooth) transitive $G$-sets.

\begin{proposition} \label{prop:GL-trans}
Every transitive $G$-set is isomorphic to one of the form $X_{\ell,m,n}/\Gamma$, where $\Gamma$ is a subgroup of $\Aut(X_{\ell,m,n})$.
\end{proposition}

Write $X_{\ell,m,n}=G/U_{\ell,m,n}$. The proposition implies that if $U$ is an open subgroup of $G$, meaning it contains some $U_{a,b,c}$, then it contains a conjugate of some $U_{\ell,m,n}$ with finite index; in fact, the proposition is nearly equivalent to this statement. The following corollary of the proposition will be important when we analyze measures.

\begin{corollary} \label{cor:GL-large}
The stabilizer class defined by $\bV$ is large.
\end{corollary}

Essentially by definition, every transitive $G$-set is isomorphic to one of the form $X(V)/R$, where $V$ is an object of $\fC$ and $R$ is a $G$-stable equivalence relation on $X(V)$. There are two obvious sources of equivalence relations on $X(V)$. First, if $\Gamma$ is a subgroup of $\Aut(V)$ then $\Gamma$ acts on $X(V)$, and this action induces an equivalence relation $R_{\Gamma}$. Second, if $W$ is a substructure of $V$ then we obtain an equivalence relation $R_W$ on $X(V)$ by declaring two embeddings $i,j, \colon V \to \bV$ to be equivalent if they agree on $W$. The proposition is implied by the following lemma:

\begin{lemma} \label{lem:GL-trans-1}
If $R$ is a $G$-stable equivalence relation on $X(V)$ then either $R=R_{\Gamma}$ for some $\Gamma \subset \Aut(V)$, or else $R$ contains $R_W$ for some proper substructure $W$ of $V$.
\end{lemma}

Indeed, suppose the lemma is true. We prove by induction on $\dim(V)$ that every quotient of $X(V)$ has the desired form. Thus let $R$ be an equivalence relation on $X(V)$. If $R=R_{\Gamma}$ then $X(V)/R$ has the desired form. Otherwise, $X(V)/R$ is a quotient of $X(V)/R_W=X(W)$ for a proper substructure $W$ of $V$, and so has the desired form by the inductive hypothesis. (This method to prove that a stabilizer class is large was used in \cite[\S 3.5]{cantor} and \cite[\S 4.3]{regcat}.)

Let $i,j \colon V \to \bV$ be two embeddings. We define the \defn{defect} of $(i,j)$ to be the codimension of $i(V) \cap j(V)$ in $i(V)$. Note that defect~0 means $i(V)=j(V)$. We define the \defn{defect} of an equivalence relation $R$ to be the maximum defect of a pair $(i,j) \in R$.

\begin{lemma} \label{lem:GL-trans-2}
Let $R$ be an equivalence relation on $X(V)$. Then $R=R_{\Gamma}$ for some $\Gamma$ if and only if $R$ has defect~0.
\end{lemma}

\begin{proof}
The equivalence relation $R_{\Gamma}$ consists of pairs $(i, i \gamma)$ with $i \in X(V)$ and $\gamma \in \Gamma$, and thus has defect~0. Suppose now that $R$ has defect~0. Let $\Gamma$ be the set of elements $\gamma \in \Aut(V)$ such that $R$ contains $(i, i \gamma)$ for some embedding $i \colon V \to \bV$. Since $G$ acts transitive on $X(V)$, if $\gamma \in \Gamma$ then $(i, i \gamma)$ belongs to $R$ for all $i \in X(V)$. We show (a) $\Gamma$ is a subgroup of $\Aut(V)$, and (b) $R=R_{\Gamma}$.

(a) Clearly, $\Gamma$ contains the identity. Suppose $\gamma, \gamma' \in \Gamma$. Then $(i, i\gamma)$ and $(i, i\gamma')$ belong to $R$, for any $i \in X(V)$. Since $\bV$ is homogeneous, there is $g \in G$ such that $gi=i\gamma$. Applying this to the second pair above, we see that $(i\gamma, i\gamma\gamma')$ belongs to $R$. Since $R$ is an equivalence relation, we see that $(i, i\gamma\gamma')$ belongs to $R$, and so $\gamma \gamma'$ belongs to $\Gamma$. We have thus shown that $\Gamma$ is a submonoid of $\Aut(V)$. It is a subgroup since the latter is finite.

(b) We have $R_{\Gamma} \subset R$ by definition. Conversely, suppose $(i,j) \in R$. Since $i(V)=j(V)$, we have $j=i\gamma$ for some $\gamma \in \Aut(V)$. Thus $\gamma \in \Gamma$, and so $(i,j) \in R_{\Gamma}$, as required.
\end{proof}

It remains to analyze the positive defect case. We do this in the next two lemmas. In what follows, we fix a $G$-stable equivalence relation $R$ on $X(V)$.

\begin{lemma} \label{lem:GL-trans-3}
Suppose $(i,j) \in R$ has defect $d$, and let $0 \le d' \le d$. Then there exists $(i',j') \in R$ of defect $d'$ and a codimension $d'$ substructure $W$ of $V$ such that $i'(x)=j'(x)$ for all $x \in W$.
\end{lemma}

\begin{proof}
Let $\bV^{\le n} \subset \bV$ be the span of the $e_i$ and $f_i$ with $i \le n$. Choose $n$ such that $i(V)$ and $j(V)$ are contained in $\bV^{\le n}$. Write $j=gi$, where $g \in G$ fixes all basis vectors $e_i$ and $f_i$ with $i>n$. Let $h \in G$ stabilize $i$. Since $(i,gi)$ belongs to $R$ and $R$ is $G$-stable and symmetric, we see that $(i,g^{-1}i)$ and $(i,hg^{-1}i)$ and $(gi,ghg^{-1}i)$ each belong to $R$. By transitivity, we see that $(i, ghg^{-1} i)$ also belongs to $R$. We take $i'=i$ and $j'=ghg^{-1} i$, where $h$ is to be determined.

Let $W$ be a subspace of $V$ of codimension $d'$ such that $i(W)$ contains $i(V) \cap j(V)$, and let $U=i(V)+g^{-1}i(W)$. Pick $h \in G$ fixing each element of $U$ such that the kernel of the map
\begin{displaymath}
\xymatrix{
\bV^{\le n} \ar[r]^-h & \bV \ar[r] & \bV/\bV^{\le n} }
\end{displaymath}
is exactly $U$. Here is how to construct such an $h$. Let $U'$ be a homogeneous complementary space to $U$ in $\bV^{\le n}$. Let $\bV^{>n}$ be the span of the $e_i$ and $f_i$ with $i>n$. Choose an injective linear map $a \colon U' \to \bV^{>n}$ such that $a(U')$ is isotropic. Extend $a$ to~0 on $U$. Then $x \mapsto x+a(x)$ is an embedding $\bV^{\le n} \to \bV$ in $\hat{\cC}$. By homogeneity, there is $h \in G$ such that $h(x)=x+a(x)$ for all $x \in \bV^{\le n}$, and this has the requisite property.

We have $i'(x)=j'(x)$ for $x \in W$. Indeed, this amounts to the identity $i(x)=ghg^{-1} i(x)$, which follows since $h$ stabilizes $g^{-1} i(W)$. We thus see that $i'(W) \subset i'(V) \cap j'(V)$. To complete the proof, we must show that this is an equality, as this will show that $(i', j')$ has defect $d'$.

Let $W'=(j')^{-1}(\bV^{\le n})$. We claim $W' \subset W$. Thus let $x \in W'$. This means that $ghg^{-1}i(x) \in \bV^{\le n}$, and so $hg^{-1}i(x) \in \bV^{\le n}$ since $\bV^{\le n}$ is $g$-stable. We thus see that $g^{-1}i(x)$ belongs to $\bV^{\le n} \cap h^{-1} \bV^{\le n}$, which is equal to $U$ by our choice of $h$. Thus $i(x) \in gU=i(W)+j(V)$, and so $i(x) \in i(V) \cap (i(W)+j(V))$. Since $i(W)$ contains $i(V) \cap j(V)$, this implies $x \in W$, as required. Now, since $i'$ maps $V$ into $\bV^{\le n}$, we find that $i'(V) \cap j'(V) \subset j'(W)=i'(W)$, which completes the proof.
\end{proof}

\begin{lemma} \label{lem:GL-trans-4}
If $R$ has positive defect then there is a codimension~1 substructure $W$ of $V$ such that $R$ contains $R_W$.
\end{lemma}

\begin{proof}
By the previous lemma, there is a codimension~1 substructure $W$ of $V$ and a pair $(i,j) \in R$ of defect~1 such that $i(x)=j(x)$ for all $x \in W$. We show $R$ contains this $R_W$.

Suppose $j' \colon V \to \bV$ is an embedding that agrees with $i$ on $W$, and such that $(i,j')$ has defect~1. Let $x_1, \ldots, x_n$ be a homogeneous basis of $V$ such that $x_1, \ldots, x_{n-1}$ is a basis of $W$. Then $i(V)+j(V)$ has basis
\begin{displaymath}
i(x_1), \ldots, i(x_n), j(x_n),
\end{displaymath}
and similarly for $i(V)+j'(V)$. Let
\begin{displaymath}
a \colon i(V)+j(V) \to i(V)+j'(V)
\end{displaymath}
be the linear map fixing each of $i(x_1), \ldots, i(x_n)$ and mapping $j(x_n)$ to $j'(x_n)$. This is clearly an isomorphism of homogeneous vector spaces. We claim that $a$ is a map of structures, i.e., it preserves pairings. For $1 \le \alpha \le n-1$, we have
\begin{displaymath}
\langle i(x_{\alpha}), j(x_n) \rangle = \langle j(x_{\alpha}), j(x_n) \rangle = \langle x_{\alpha}, x_n \rangle,
\end{displaymath}
and similarly for $j'$; we thus see that
\begin{displaymath}
\langle i(x_{\alpha}), j(x_n) \rangle = \langle i(x_{\alpha}), j'(x_n) \rangle
\end{displaymath}
for such $\alpha$. This also holds for $\alpha=n$, since then the pairings vanish, as all elements have the same parity. Similarly, $\langle j(x_n), j(x_n) \rangle=\langle j'(x_n), j'(x_n) \rangle$, as both vanish. This proves the claim. Since $\bV$ is homogeneous, there is some $g \in G$ that induces $a$. This $g$ satisfies $gi=i$ and $gj=j'$. Since $(i,j)$ belongs to $R$ and $R$ is $G$-stable, we see that $(i,j')$ belongs to $R$ as well.

Now suppose that $j' \colon V \to \bV$ agrees with $i$ on $W$, but that $(i,j')$ has defect~0. This means that $i(V)=j'(V)$, and so $j'=i\gamma$ for some $\gamma \in \Aut(V)$ that fixes each element of $W$. Now, $j\gamma$ agrees with $i$ on $W$, and $(i,j\gamma)$ has defect~1, and so $(i,j\gamma)$ belongs to $R$ by the previous paragraph. Since $(i,j)$ also belongs to $R$, we see that $(j,j\gamma)$ belongs to $R$. Since $R$ is $G$-stable, it follows that $(i,i\gamma)=(i,j')$ belongs to $R$.

Finally, suppose that $(i',j')$ is any pair of embeddings $V \to \bV$ that agrees on $W$. Note that this pair has defect~0 or~1. We can find $g \in G$ such that $gi'=i$. We have seen above that $g(i',j')$ belongs to $R$. Since $R$ is $G$-stable, it follows that $(i',j')$ belongs to $R$. Thus $R$ contains $R_W$, as required.
\end{proof}

\subsection{Smooth approximation}

We now show that $G$ is smoothly approximable (\S \ref{ss:smooth}).

\begin{proposition} \label{prop:GL-simply-conn}
The $G$-set $X_{a,b,c}$ is simply connected, for all $a$, $b$, $c$.
\end{proposition}

\begin{proof}
First suppose that we have an embedding $i \colon V_{a,b,c} \to V_{a',b',c'}$ such that the induced map $i^* \colon X_{a',b',c'} \to X_{a,b,c}$ is a finite cover. By Proposition~\ref{prop:fincover}, there is some $h$ such that any embedding $V_{a,b,c} \to V_{\ell,m,n}$ admits at most $h$ extensions to $V_{a',b',c'}$. We thus find
\begin{displaymath}
N_{a',b',c'}^{\ell,m,n} \le h \cdot N_{a,b,c}^{\ell,m,n}
\end{displaymath}
for all $\ell$, $m$, and $n$. In particular, taking $m=n=0$, we find
\begin{displaymath}
P_{a',b',c'}(q^{\ell}) \le h \cdot P_{a,b,c}(q^{\ell}).
\end{displaymath}
Since $P_{a,b,c}(q^{\ell})$ is a polynomial of degree $2a+b+c-1$ in $q^{\ell}$, and similarly for $P_{a',b',c'}(q^{\ell})$, we find that $2a'+b'+c' \le 2a+b+c$. Thus $\dim(V_{a',b',c'}) \le \dim(V_{a,b,c})$, and so $i$ is an isomorphism, and so $i^*$ is as well, that is, the finite cover $i^*$ is trivial.

Now suppose that we have a finite cover $f \colon X \to X_{a,b,c}$ for some transitive $G$-set $X$. Since the stabilizer class $\sE(\bV)$ is large, there is a finite cover $g \colon X_{a',b',c'} \to X$ for some $a'$, $b'$, and $c'$. By the previous paragraph, $fg$ is an isomorphism, and so $f$ is an isomorphism. Thus all finite covers of $X_{a,b,c}$ are trivial, as required.
\end{proof}

\begin{proposition} \label{prop:GL-smooth}
The family $\{X_{\ell,0,0}\}_{\ell \ge 0}$ is a smooth approximation of $G$.
\end{proposition}

\begin{proof}
The structure $V_{\ell,0,0}$ is homogeneous by a variant of Witt's theorem on embeddings into non-degenerate quadratic spaces. It follows that $X_{\ell,0,0}$ is homogeneous relative to $\sE(\bV)$. Proposition~\ref{prop:rel-homo} now implies that $X_{\ell,0,0}$ is absolutely homogeneous. Finally, any structure $V_{a,b,c}$ embeds into some $V_{\ell,0,0}$, since $N^{\ell,0,0}_{a,b,c}=P_{a,b,c}(q^{\ell})$ is positive for $\ell$ sufficiently large (see Proposition~\ref{prop:GL-Ppoly}). This completes the proof.
\end{proof}

\subsection{Measures}

Let $N_{a,b,c}^{*,0,0}$ be the function defined by $\ell \mapsto N^{\ell,0,0}_{a,b,c}$. Since $\{X_{\ell,0,0}\}_{\ell \ge 0}$ is a smooth approximation of $G$ (Proposition~\ref{prop:GL-smooth}), by \S \ref{ss:count} we obtain a counting measure
\begin{displaymath}
\Theta(G; \bV) \to \Fun^{\circ}(\bN, \bZ), \qquad x_{a,b,c} \mapsto N^{*,0,0}_{a,b,c}.
\end{displaymath}
Since $N^{\ell,0,0}_{a,b,c}=P_{a,b,c}(q^{\ell})$, it follows that we have a ring homomorphism
\begin{displaymath}
\phi \colon \Theta(G; \bV) \to \bZ_q[t], \qquad
\phi(X_{a,b,c})=P_{a,b,c}(t).
\end{displaymath}
The following result shows that this measure is universal.

\begin{theorem} \label{thm:GL-measure}
The map $\phi$ is a ring isomorphism after inverting $q$.
\end{theorem}

\begin{proof}
Let $a=[X_{1,0,0}]$, $b=[X_{0,1,0}]$, and $c=[X_{0,0,1}]$. Since $\Theta(G; \bV)$ is a quotient of $\hat{\BB}$, it follows that $a$, $b$, and $c$ generate $\Theta(G; \bV)[1/q]$ as a $\bZ_q$-algebra.

Now, $X_{1,0,0}$ is the subset of $\bV_0 \times \bV_1$ where the pairing is equal to~1. Projecting onto the first coordinate gives a $G$-equivariant map $X_{1,0,0} \to X_{0,1,0}$. Let $F$ be the fiber. Then, by the axioms of measure, we have
\begin{displaymath}
[X_{1,0,0}]=[F] \cdot [X_{0,1,0}].
\end{displaymath}
Now, $F$ is the set of vectors $v \in \bV_1$ such that $\langle e_1, v \rangle=1$. This equation simply means the coefficient of $f_1$ in $v$ is~1. Thus the map $\bV_1(1) \to F$ given by $w \mapsto f_1+w$ is an isomorphism. As we saw above, we have $q[\bV_1(1)]=[\bV_1]$, and $[\bV_1]=c+1$. Putting this all together, we find
\begin{displaymath}
a=q^{-1} b(c+1)=q^{-1} bc + q^{-1} b.
\end{displaymath}
Of course, we can run the whole argument by projecting onto $X_{0,0,1}$ instead, and this leads to a similar equation but with $b$ and $c$ switched. We conclude that $b=c$.

The above discussion shows that $b$ generates $\Theta_q(G;\bV)[1/q]$ as a $\bZ_q$-algebra. Since $\phi(b)=P_{0,1,0}=t-1$, the result thus follows.
\end{proof}

\begin{corollary}
The map $\phi$ induces a ring isomorphism $\Theta(G) \otimes \bQ \to \bQ[t]$.
\end{corollary}

\begin{proof}
This follows since the stabilizer class $\sE(\bV)$ is large; see \S \ref{ss:stab} and Corollary~\ref{cor:GL-large}.
\end{proof}

Given a field $k$ of characteristic~0 and an element $t \in k$, we let
\begin{displaymath}
\mu_t \colon \Theta(G) \to k
\end{displaymath}
be the measure obtained by composing the isomorphism $\phi$ with the homomorphism $\bQ[t] \to k$ mapping $t$ to $t$. The proof of Theorem~\ref{thm:GL-measure} shows that $\mu_t$ is the unique $k$-valued measure satisfying $\mu_t(X_{0,1,0})=t-1$, or, equivalently, satisfying $\mu_t(\bV_0)=t$. There is a map $\Theta(G(n)) \to \Theta(G)$, which allows us to restrict measures from $G$ to $G(n)$. If we identify $G(n)$ with $G$, then the restriction of $\mu_t$ is $\mu_{t/q^n}$.

\begin{proposition} \label{prop:GL-reg}
Let $k$ be a field of characteristic~0 and let $t \in k$.
\begin{enumerate}
\item The measure $\mu_t$ is regular if and only if $t$ is not~0, and not of the form $q^i$ with $i \ge 0$.
\item The measure $\mu_t$ is quasi-regular if and only if $t \ne 0$.
\end{enumerate}
\end{proposition}

\begin{proof}
(a) If $t$ is not zero and not a positive power of $q$ then $P_{a,b,c}(t)$ is non-zero for all $a$, $b$, and $c$. This follows from the explicit computation of these polynomials (Proposition~\ref{prop:GL-Ppoly}). Thus $\mu_t$ is regular. Conversely, if $t$ is zero or a positive power of $q$ then $P_{a,b,c}(t)=0$ for some choice of $a$, $b$, and $c$, and so $\mu_t$ is not regular.

(b) If $t$ is non-zero then, for $n$ sufficiently large, $\mu_{t/q^n}$ is regular by (a), which means that the restriction of $\mu_t$ to $G(n)$ is regular, and so $\mu_t$ is quasi-regular. It $t=0$ then the restriction of $\mu_t$ to $G(n)$ is clearly not regular. Since the $G(n)$'s are cofinal in $\sO(G)$, it follows that $\mu_t$ is not regular on any open subgroup.
\end{proof}

\begin{remark}
We have an embedding $G \times G \to G$ by partitioning the basis set into two infinite subsets. This induces a map $\BB(G) \to \BB(G \times G)$. We also have a map $\BB(G \times G) \to \Theta(G) \otimes \BB(G)$ by Remark~\ref{rmk:burn-prod}. A similar remark holds for the versions of $\BB$ and $\Theta$ relative to $\sE(\bV)$. Combining these, we obtain a map
\begin{displaymath}
\delta_t \colon \BB \to \bZ_q[t] \otimes \BB.
\end{displaymath}
One can show that $\delta$ is obtained from $\delta_t$ by putting $t=q$. Moreover, the elements $w_1$, $w_2$, and $w_3$ defined in Remark~\ref{rmk:GL-grading} are eigenvectors for $\delta_t$ of eigenvalue $t$.
\end{remark}

\subsection{Ultraproduct measures} \label{ss:GL-ultra-meas}

We have just classified all measures for $G$ (in characteristic~0). We now examine how the ultraproduct measures constructed in \S \ref{ss:ultra} fit into this classification. Fix a field $\kappa_n$ for each $n \in \bN$, let $\cF$ be a non-principal ultrafilter on $\bN$, and let $\kappa$ be the ultraproduct of the $\kappa_n$'s, which we assume to have characteristic~0. Let $\nu$ be the ultraproduct measure defined in this situation, with respect to the smooth approximation $\{X_{n,0,0}\}_{n \ge 0}$.

\begin{proposition} \label{prop:GL-ultra-meas-param}
We have $\nu=\mu_{\tau}$, where $\tau \in \kappa$ is defined by the sequence $(q^n)_{n \ge 0}$.
\end{proposition}

\begin{proof}
Let $H_n=\GL_n(\bF)$ be the automorphism group of $X_{n,0,0}$, and let $\Pi_n \colon \bS(G) \to \bS(H_n)$ be the functor $\Hom_G(X_{n,0,0}, -)$. By definition, $\nu(X_{0,1,0})$ is the element of $\kappa$ represented by the sequence $(a_n)_{n \ge 0}$, where $a_n=\# \Pi_n(X_{0,1,0})$. The set $\Pi_n(X_{0,1,0})$ is in bijection with the set of embeddings $V_{0,1,0} \to V_{n,0,0}$. Such an embedding is simply picking a non-zero element of the degree~0 part of $V_{n,0,0}$, and so there are $q^n-1$ embeddings. We thus have $a_n=q^n-1$. By the classification of measures, $\nu=\mu_{\tau}$ where $\tau-1=\nu(X_{0,1,0})$. We thus find $\tau=(q^n)_{n \ge 0}$, as claimed.
\end{proof}

We say that two measures $\mu_t$ and $\mu_{\tau}$ for $G$, valued in (perhaps different) fields of characteristic~0 are \defn{equivalent} if they define the same point of $\Spec(\Theta(G))$. Concretely, this means that either $t$ and $\tau$ are both transcendental over $\bQ$, or that $t$ and $\tau$ are both algebraic over $\bQ$ and have the same minimal polynomial. We next show that nearly all measures are equivalent to ultraproduct measures.

\begin{proposition} \label{prop:ultra-equiv}
Let $k$ be a field of characteristic~0, and let $t \in k$ be non-zero. Then $\mu_t$ is equivalent to some ultraproduct measure.
\end{proposition}

\begin{corollary} \label{cor:GL-nil}
With notation as in the proposition, $\mu_t$ satisfies (Nil).
\end{corollary}

\begin{proof}
By Proposition~\ref{cor:ultra-nil}, any ultraproduct measures satisfy (Nil). By Proposition~\ref{prop:nil}, if a measure satisfies (Nil) then any equivalent measure does. The result thus follows.
\end{proof}

We require two lemmas before proving the proposition. Given a set $S$ of integers, the \defn{prime divisors} of $S$ are the prime numbers dividing some element of $S$.

\begin{lemma} \label{lem:ultra-equiv-1}
Let $h(x)$ be a polynomial with integer coefficients that is not a scalar multiple of a power of $x$, let $b \ge 2$ be an integer, and let $r \ge 0$ be an integer. Then the set $\{h(b^n)\}_{n \ge r}$ has infinitely many prime divisors.
\end{lemma}

\begin{proof}
Replacing $h(x)$ by $h(b^rx)$, we assume $r=0$. Let $d$ be the degree of $h(x)$. Let
\begin{displaymath}
f(x) = \prod_{i=0}^d (x-b^i) = \sum_{i=0}^d c_i x^i.
\end{displaymath}
Then letting $a_n=h(b^n)$, we have
\begin{displaymath}
\sum_{i=0}^d c_i a_{n+i} = 0
\end{displaymath}
for all $n \ge 0$. The result now follows from the main theorem of \cite{Laxton}.
\end{proof}

\begin{lemma} \label{lem:ultra-equiv-2}
Let $K$ be a number field, let $a \in K$ be non-zero, and let $b \ge 2$ be an integer. Then there is a strictly increasing sequence of positive integers $n_1, n_2, \ldots$ and a sequence of prime ideals $\fp_1, \fp_2, \ldots$ of increasing residue characteristic such that $a=b^{n_i} \pmod{\fp_i}$ for all $i \ge 1$.
\end{lemma}

\begin{proof}
There is no harm in enlarging $K$, so we assume $K$ is Galois over $\bQ$. Let $h(x) \in \bQ[x]$ be the monic minimal polynomial of $a$ over $\bQ$, and write $h(x)=\prod_{i=1}^d(x-a_i)$, where $a_1=a$. Let $M$ be a positive integer such that $M h(x)$ has integer coefficients; note that $a$ belongs to $\cO_K[1/M]$.

Suppose we have constructed $n_1, \ldots, n_r$ and $\fp_1, \ldots, \fp_r$. By Lemma~\ref{lem:ultra-equiv-1}, we can find $n>n_r$ and a prime $\ell$ greater than the residue characteristic of $\fp_r$ (and not dividing $M$) such that $\ell$ divides $h(b^n)$. Since $h(b^n) = \prod_{i=1}^d (b^n-a_i)$ and the $b^n-a_i$ are Galois conjugate (for $1 \le i \le d$), each of these numbers must be divisible by a prime ideal over $\ell$. In particular, there is some prime $\fp$ above $\ell$ dividing $b^n-a$, meaning $a=b^n \pmod{\fp}$. We take $n_{r+1}=n$ and $\fp_{r+1}=\fp$.
\end{proof}

\begin{proof}[Proof of Proposition~\label{prop:ultra-equiv}]
First suppose that $t$ is transcendental. We then need to pick the $\kappa_n$'s and $\cF$ so that $\tau$ is transcendental. For this, we can simply take all the $\kappa_n$'s to have characteristic~0.

Next suppose that $t$ is algebraic. Let $K$ be the number field generated by $t$, and let $h(x)$ be the minimal polynomial of $t$. Let $n_1, n_2, \ldots$ and $\fp_1, \fp_2, \ldots$ be the sequences produced by Lemma~\ref{lem:ultra-equiv-2} with $a=t$ and $b=q$, and put $S=\{n_1,n_2,\ldots\}$. Let $\kappa_{n_i}=\cO_K/\fp_{n_i}$, and let $\kappa_n$ be any field if $n \not\in S$. Let $\cF$ be any non-principal ultrafilter on $\bN$ containing the set $S$. Since $h(q^n)=0$ holds in $\kappa_n$ for all $n \in S$, we see that $h(\tau)=0$. Thus the ultraproduct measure $\nu=\mu_{\tau}$ is equivalent to $\mu_t$, as required.
\end{proof}

\begin{remark}
When $a$ is a rational number, Lemma~\ref{lem:ultra-equiv-2} is proved in \cite{Polya}.
\end{remark}

\section{The general linear group: representations} \label{s:gl2}

\subsection{Tensor categories}

We maintain the general set-up from \S \ref{s:gl}. Fix, for the duration of \S \ref{s:gl2}, an algebraically closed field $k$ of characteristic~0 and a non-zero element $t \in k$. We say that $t$ is a \defn{regular parameter} if $\mu_t$ is regular; this simply means that $t$ is not of the form $q^n$ with $n \in \bN$. Let $\uPerm_k(\GL_t)$ be the category $\uPerm_k(G, \mu_t)$. We let $\uRep_k(\GL_t)$ be the category $\uRep_k(G, \mu_t)$ of finite length modules over the completed group algebra. We omit the $k$ subscript when possible.

\begin{theorem} \label{thm:GL-abelian}
The category $\uRep(\GL_t)$ is pre-Tannakian, and the abelian envelope of the category $\uPerm(\GL_t)$. If $t$ is a regular parameter then $\uRep(\GL_t)$ is semi-simple.
\end{theorem}

\begin{proof}
This follows from the general theory in \S \ref{ss:abenv}, combined with Proposition~\ref{prop:GL-reg} and Corollary~\ref{cor:GL-nil}.
\end{proof}

We make one simple observation about this category here. Recall that $\bV=\bV_0 \times \bV_1$ holds as $G$-sets, where $\bV_0$ and $\bV_1$ are the graded pieces of $\bV$.

\begin{proposition} \label{prop:GL-gen}
The Schwartz space $\cC(\bV_0)$ generates $\uRep(\GL_t)$. In fact, every object of $\uRep(\GL_t)$ is a quotient of a direct sum of tensor powers of $\cC(\bV_0)$.
\end{proposition}

\begin{proof}
Every object of $\uRep(\GL_t)$ is a quotient of some Schwartz space $\cC(X)$. Since every transitive $G$-set is a subquotient of $\bV^n$, it follows that every object of $\uRep(\GL_t)$ is a quotient of a sum of tensor powers of $\cC(\bV)$. We have $\cC(\bV)=\cC(\bV_0) \uotimes \cC(\bV_1)$. By Proposition~\ref{prop:fourier}, we have an isomorphism $\cC(\bV_0) \cong \cC(\bV_1)$, and so the result follows.
\end{proof}

\begin{remark} \label{rmk:Levi-Parabolic}
As noted in \cite{repst}, there are two versions of the infinite general linear group. The one we are discussing, namely $G$, consists of the automorphisms of $\bV_0 \oplus \bV_1$ preserving the grading and pairing. The other one, call it $H$, consists of all automorphisms of the vector space $\bV_0$. We call $G$ the \defn{Levi version} and $H$ the \defn{parabolic version}, due to the shape of the basic open subgroups. These give two potentially two different versions of $\uRep(\GL_t)$. A priori the parabolic version is a full subcategory of the Levi version, and this proposition shows that in fact every object in the Levi category lies in the parabolic category, so the two versions coincide. We note that the discussion here shows that the oligomorphic fundamental group (defined in \cite{discrete}) for $\uRep(H, \mu_t)$ does not coincide with $H$, since \'etale algebras in $\uRep(G, \mu_t)$ also live in this category; see \cite{delalg} for details.
\end{remark}

\subsection{Simples} \label{ss:GL-simple}

For \S \ref{ss:GL-simple}, we suppose $t$ is a regular parameter. Recall that $X_{0,n,0}$ is the space of injective linear maps $\bF^n \to \bV_0$. Let $A_n$ be the Schwartz space $\cC(X_{0,n,0})$. The finite group $H_n=\GL_n(\bF)$ acts on $A_n$ through its action on $\bF^n$. For each injective linear map $i \colon \bF^{n-1} \to \bF^n$, there is a pull-back map $X_{0,n,0} \to X_{0,n-1,0}$, and an associated push-forward map $A_n \to A_{n-1}$. Let $B_n$ be the intersection of the kernels of these maps, as $i$ varies. This space is stable under the action of $H_n$. Let $\{S_{\lambda}\}_{\lambda \in \Lambda_n}$ be the set of irreducible representations of $H_n$ (up to isomorphism) over $k$, where $\Lambda_n$ is an index set. Decompose $B_n$ under the action of $H_n$:
\begin{displaymath}
B_n = \bigoplus_{\lambda \in \Lambda_n} L_{\lambda}(t) \otimes S_{\lambda},
\end{displaymath}
where $L_{\lambda}(t)$ is an object of $\uRep(\GL_t)$. Let $\Lambda=\coprod_{n \ge 0} \Lambda_n$.

\begin{proposition}
The module $L_{\lambda}(t)$ is simple or zero, and every simple object of $\uRep(\GL_t)$ is isomorphic to $L_{\lambda}(t)$ for exactly one $\lambda \in \Lambda$.
\end{proposition}

\begin{proof}
For a finite dimensional $\bF$-vector space $V$, regard $V$ as $\bZ/2$-graded and concentrated in degree~0; in this way, it is an object of $\fC$. Let $A(V)=\cC(X(V))$, and let $B(V)$ be defined similarly to $B_n$ above. If $V=\bF^n$ then $X(V)=X_{0,n,0}$ and $A(V)=A_n$ and $B(V)=B_n$.

Every orbit on $X(V_1) \times X(V_2)$ has the form $X(W)$, where $W$ is an amalgamation of $V_1$ and $V_2$, that is, there is a surjection $V_1 \oplus V_2 \to W$ that is injective on each $V_i$. Fix an orbit $X(W)$, and let $U$ be the fiber product of $V_1$ and $V_2$ over $W$; concretely, if we identify $V_1$ and $V_2$ with subspaces of $W$ then $U$ is just their intersection. Giving a linear map $W \to \bV_0$ is equivalent to giving maps on each $V_i$ that agree on $U$. It follows that the natural square
\begin{displaymath}
\xymatrix{
X(W) \ar[r]^{g'} \ar[d]_{f'} & X(V_1) \ar[d]^f \\
X(V_2) \ar[r]^g & X(U) }
\end{displaymath}
is cartesian. We thus have
\begin{displaymath}
f'_* (g')^* = g^* f_*
\end{displaymath}
as maps $A(V_1) \to A(V_2)$ by \cite[Proposition~4.12]{repst}. The maps on the left side form a basis of $\Hom(A(V_1), A(V_2))$, as $X(W)$ varies over all orbits. We thus see that the maps on the right side form a basis too.

Specialize now to the case where $V_1=V_2=V$. If $U$ is a proper subspace of $V$ then the map $f_*$ above restricts to~0 on $B(V)$. Thus to obtain a non-zero map on $B(V)$, we require $U=V$, which means that the two maps $V \to W$ are both isomorphisms. In this case, the map on $A(V)$ is just the one induced by an automorphism of $V$. We have thus shown that the natural map
\begin{displaymath}
k[\GL(V)] \to \End(B(V))
\end{displaymath}
is surjective. Note that every endomorphism of $B(V)$ is induced from an endomorphism of $A(V)$ by semi-simplicity. It is now clear that each $L_{\lambda}(t)$ is either~0 or simple.

Every simple object is a constituent of some $A_n$ by Proposition~\ref{prop:GL-gen}, as the tensor powers of $\cC(\bV_0)$ decompose into sums of $A_n$'s. Now, we have an exact sequence
\begin{displaymath}
0 \to B_n \to A_n \to A_{n-1}^{\oplus r(n)}
\end{displaymath}
for some $r(n)$. Thus, by induction, we see that every simple is a constituent of some $B_n$. It follows that every simple is isomorphic to some $L_{\lambda}(t)$.

To complete the proof, we must show that $L_{\lambda}(t)$ and $L_{\mu}(t)$ are non-isomorphic when $\lambda, \mu \in \Lambda$ are distinct. Suppose $\lambda \in \Lambda_n$ and $\mu \in \Lambda_m$. If $n=m$ then this follows from the construction of the simples. Now suppose $n \ne m$; say $n>m$. Then $\Hom(L_{\lambda}(t), A_m)=0$ while $\Hom(L_{\mu}(t), A_m) \ne 0$. Thus $L_{\lambda}(t)$ and $L_{\mu}(t)$ are non-isomorphic. The result follows.
\end{proof}

\begin{remark}
The above proof is modeled on that of \cite[Theorem~6.1]{Knop2}.
\end{remark}

\begin{remark} \label{rmk:simple-nonzero}
In fact, each $L_{\lambda}(t)$ is non-zero. We briefly sketch the proof. By the Tits deformation theorem, it suffices to treat the case where $t$ is transcendental. For $g \in H_n$, a simple computation shows
\begin{displaymath}
\utr(g \vert B_n) = \delta_g t^n + O(t^{n-1}),
\end{displaymath}
where $\delta_g$ is~1 if $g=1$ and~0 otherwise. Since $L_{\lambda}(t)=e_{\lambda} B_n$, where $e_{\lambda}$ is the idempotent of $k[H_n]$ corresponding to $\lambda$, we find
\begin{displaymath}
\udim(L_{\lambda}(t)) = \utr(e_{\lambda} \vert B_n) = \frac{\dim{S_{\lambda}}}{\# H_n} t^n + O(t^{n-1}).
\end{displaymath}
Since this is non-zero, it follows that $L_{\lambda}$ is non-zero.
\end{remark}

We define the \defn{level} of a simple $L_{\lambda}(t)$ to be the $n$ for which $\lambda \in \Lambda_n$. We define the \defn{level} of an object of $\uRep(\GL_t)$ to be the supremum of the levels of simples it contains; the level of the zero representation is $-\infty$.

\begin{proposition} \label{prop:GL-res}
For $\lambda \in \Lambda_n$, the restriction of $L_{\lambda}(t)$ to $\GL_{t/q}$ decomposes as
\begin{displaymath}
L_{\lambda}(t/q)^{\oplus q^n} \oplus X,
\end{displaymath}
where $X$ has level $<n$.
\end{proposition}

\begin{proof}
Write $A_n(t)$ in place of $A_n$ to indicate the parameter. We have $\bV_0 = \bF \oplus \bV_0(1)$ as $G(1)$-sets, which shows that the restriction of $A_1(t)$ to $\GL_{t/q}$ is isomorphic to $A_1(t/q)^{\oplus q}$; note that $A_1(t)$ is simply the Schwartz space on $\bV_0$. Now, we have $A_1(t)^{\otimes n} = A_n(t) \oplus X$ where $X$ is a sum of $A_m(t)$'s with $m<n$. We thus see that the restriction of $A_n(t)$ to $\GL_{t/q}$ is $A_n(t/q)^{\oplus q^n} \oplus X$, where $X$ is as before. Taking isotypic components for $H_n$ yields the result.
\end{proof}

\begin{remark} \label{rmk:GL-K-graded}
Let $\rK=\rK(\uRep(\GL_t))$ be the Grothendieck group at a regular parameter value; this is independent of $t$ (as a ring) by the Tits deformation theorem. Restriction from $\GL_t$ to $\GL_{t/q}$ thus induces a ring homomorphism $\delta \colon \rK \to \rK$. It follows from Proposition~\ref{prop:GL-res} that $\rK \otimes \bQ$ has a basis consisting of eigenvectors of $\delta$, and the eigenvalues are powers of $q$. In particular, $K$ is naturally a graded ring; the degree $n$ piece is the $q^n$ eigenspace of $\delta$. This is similar to what we saw for the Burnside ring (Remark~\ref{rmk:GL-grading}).
\end{remark}

\subsection{Central characters} \label{ss:GL-char}

We now investigate characters and central characters in $\uRep(\GL_t)$. We maintain the assumption that $t$ is a regular parameter. The following is our main result:

\begin{theorem} \label{thm:GL-char}
Every non-trivial simple has non-trivial central character.
\end{theorem}

Before proving the theorem, we give two corollaries:

\begin{corollary}
Non-isomorphic simples have distinct central characters.
\end{corollary}

\begin{proof}
This follows from Proposition~\ref{prop:cent-char}.
\end{proof}

\begin{corollary} \label{cor:GL-char-2}
The simple characters are linearly independent.
\end{corollary}

\begin{proof}
This follows from Corollary~\ref{cor:cent-char}.
\end{proof}

We now start the proof of the theorem. We begin by handling the case where $t$ is transcendental. Consider the following condition:
\begin{description}[align=right,labelwidth=1.5cm,leftmargin=!]
\item[$\Sigma(n)$ ] If $t$ is transcendental over $\bQ$ and $L$ is a non-trivial simple object of $\uRep(\GL_t)$ level $\le n$ then the central character of $L$ is non-trivial.
\end{description}
We must prove $\Sigma(n)$ for all $n$. The proof will be by induction. We begin with the base case:

\begin{lemma} \label{lem:GL-char-1}
The statement $\Sigma(1)$ is true; moreover, the characters of level $\le 1$ irreducible representations are linearly independent.
\end{lemma}

\begin{proof}
The group $\bF^*$ acts on $\bV$ by scalar multiplication, and this induces an action of $\bF^*$ on the Schwartz space $\cC(\bV \setminus \{0\})$. For a character $\omega$ of $\bF^*$, let $M_{\omega}$ be the $\omega$-isotypic piece of this action. If $\omega$ is non-trivial then $M_{\omega}$ is simple, while $M_1=\bone \oplus M'_1$ for a simple $M'_1$. The simples $M_{\omega}$ (with $\omega$ non-trivial) and $M'_1$ are exactly the level one simples. Note that $M_{\omega}$ has dimension $(t-1)/(q-1)$.

The center of $G$ acts on $M_{\omega}$ through $\omega$. Thus if $a \in \bF^*$ is regarded as a diagonal matrix then for any c-smooth element $g \in G$, we have
\begin{displaymath}
\chi_{M_{\omega}}(ag) = \omega(a) \chi_{M_{\omega}}(g).
\end{displaymath}
From this, we see that a non-trivial linear dependency between level $\le 1$ characters cannot involve any $M_{\omega}$ with $\omega$ non-trivial. It remains to show that the characters of the trivial representation and $M'_1$ are linearly independent. Let $g \in G$ be the element given by $ge_1=e_1+e_2$ and $ge_i=e_i$ for all $i \ge 2$. A simple computation shows that $\chi_{M_1}(g)=(t/q-1)/(q-1)$. Since $\chi_{M_1}(1)=(t-1)/(q-1)$, we see that $\chi_{M_1}$ is not the constant function on $\GG$, and so $\chi_{M'_1}$ is not either. Thus $\chi_{M'_1}$ and $\chi_{\bone}$ are linearly independent.
\end{proof}

\begin{lemma} \label{lem:GL-char-8}
Suppose $\lambda \in \Lambda_2$ and $t$ is transcendental. Then the restriction of $L_{\lambda}(t)$ to $\GL_{t/q}$ contains a level one summand.
\end{lemma}

\begin{proof}
Let $P \subset G$ be the stabilizer of $e_1 \in \bV_0$. Note that $G(1)$ is contained in $P$, and the quotient $P/G(1)$ is isomorphic to $\bV_1(1) \setminus \{0\}$ as a $G(1)$-set. It follows that $\cC(P/G(1))$ is isomorphic, in $\uRep(G(1), \mu_t)=\uRep(\GL_{t/q})$, to $\cC(\bV_1(1) \setminus \{0\})$, which, in turn is isomorphic to $\cC(\bV_0(1) \setminus \{0\})$; denote this module by $A_1(t/q)$. As we have seen, $A_1(t/q)$ contains a unique copy of the trivial representation, and its other summands are irreducible of level~1.

Suppose now that the restriction of $L_{\lambda}(t)$ to $\GL_{t/q}$ contains a trivial summand. Let $x$ be a non-zero $G(1)$-invariant vector in $L_{\lambda}(t)$. Then there is a map in $\uRep(P, \mu_t)$
\begin{displaymath}
\cC(P/G(1)) \to L_{\lambda}(t), \qquad \phi \mapsto \int_{P/G(1)} \phi(g) gx dg.
\end{displaymath}
If this map factored through the trivial quotient of $\cC(P/G(1))$ then $P$ would act trivially on $x$. However, this would mean that $L_{\lambda}(t)$ occurs in $\Ind_P^G(k) \cong \cC(\bV_0)$, contradicting the fact that $L_{\lambda}(t)$ has level~2. We thus have a map $A_1(t/q) \to L_{\lambda}(t)$ in $\uRep(\GL_{t/q})$ that does not factor through the trivial quotient of $A_1(t/q)$, and so the result follows.

Now suppose that the restriction of $L_{\lambda}(t)$ to $\GL_{t/q}$ contains no summand of level $\le 1$; we will obtain a contradiction. By Proposition~\ref{prop:GL-res}, we see that the restriction is isomorphic to $L_{\lambda}(t/q)^{\oplus q^2}$. Since $t$ is transcendental, we can iterate this argument; it follows that the restriction of $L_{\lambda}(t)$ to $\GL_{t/q^n}$ is isomorphic to $L_{\lambda}(t/q)^{\oplus nq^2}$. However, this means that $L_{\lambda}(t)$ has no invariant vector under $G(n)$ for any $n \ge 1$, a contradiction.
\end{proof}

\begin{lemma}
The statement $\Sigma(2)$ holds.
\end{lemma}

\begin{proof}
Suppose by way of contradiction that the central character of $L_{\lambda}(t)$ is trivial. Consider the restriction to $\GL_{t/q}$:
\begin{displaymath}
L_{\lambda}(t) = L_{\lambda}(t/q)^{\oplus q^2} \oplus \bone^{\oplus m_0} \oplus \bigoplus_{\mu \in \Lambda_1} L_{\mu}(t/q)^{\oplus m(\mu)}.
\end{displaymath}
Since $t$ is transcendental, it follows that the central character of $L_{\lambda}(t/q)$ is also trivial. We thus see that the characters of $L_{\lambda}(t)$ and $L_{\lambda}(t/q)$ are a multiple of the trivial character. It follows that
\begin{displaymath}
\sum_{\mu \in \Lambda_1} m(\mu) \chi_{\mu,t/q} = c \chi_{\bone},
\end{displaymath}
for some scalar $c$. Since some $m(\mu)$ is non-zero (Lemma~\ref{lem:GL-char-8}), this contradicts the linear independence of level $\le 1$ characters (Lemma~\ref{lem:GL-char-1}).
\end{proof}

\begin{lemma} \label{lem:GL-char-2}
Assume $\Sigma(n)$. If $M \in \uRep(\GL_t)$ has level $\le n$ and $\chi_M=0$ then $M=0$.
\end{lemma}

\begin{proof}
Write $M = \bigoplus_{i=1}^r L_i^{\oplus m_i}$ where the $L_i$'s are distinct simples and $L_1$ is the trivial representation. Then $0 = \chi_M = \sum_{i=1}^r m_i \chi_{L_i}$. By Proposition~\ref{prop:chi-psi}, we have
\begin{displaymath}
0 = \sum_{i=1}^r m_i \udim(L_i) \cdot \psi_{L_i}.
\end{displaymath}
Since the $\psi_{L_i}$'s for $i \ge 2$ are not the trivial character by $\Sigma(n)$ and distinct central characters are linearly independent (Lemma~\ref{lem:char-ind}), the above relation implies $m_1=0$, i.e., the trivial representation is not a summand of $M$. We thus find that the invariant space $M^G$ vanishes.

We now restrict to $G(n)$ and apply the same argument. The restriction of $M$ to $G(n)=\GL_{t/q^n}$ has level $\le n$ by Proposition~\ref{prop:GL-res}. It also has vanishing character: indeed, if $g \in G(n)$ is c-smooth then it is also c-smooth in $G$, and so the trace of $g$ on $M$ is $\chi_M(g)=0$. We thus see that $M^{G(n)}=0$. Since this holds for all $n$, we find $M=0$. Indeed, by definition of $\uRep(\GL_t)$, every element of $M$ is fixed by some open subgroup, and so $M$ is the union of the $M^{G(n)}$ for $n \ge 0$.
\end{proof}

\begin{lemma} \label{lem:GL-char-3}
The statement $\Sigma(n)$ holds for all $n$.
\end{lemma}

\begin{proof}
We have already proved $\Sigma(n)$ for $n \le 2$. Suppose now that $\Sigma(n)$ holds for some $n \ge 2$. We will prove $\Sigma(n+1)$. Thus let $\lambda \in \Lambda_{n+1}$, and suppose by way of contradiction that $L_{\lambda}(t)$ has trivial central character.

First note that the trivial representation does not appear in the restriction of $L_{\lambda}(t)$ to $G(1)$. Indeed, suppose that it did. Then $L_{\lambda}(t)$ would appear in the induction of the trivial representation, which is $\cC(G/G(1))=\cC(X_{1,0,0})$. By Proposition~\ref{prop:GL-gen} (or its proof), this Schwartz space is a summand of $\cC(\bV_0)^{\otimes 2}$, which is a sum of Schwartz spaces $\cC(X_{0,m,0})$ with $m \le 2$. We thus find that $L_{\lambda}(t)$ has level $\le 2$, a contradiction.

By Proposition~\ref{prop:GL-res}, the restriction of $L_{\lambda}(t/q)$ to $\GL_{t/q}$ has the form $L_{\lambda}(t/q)^{\oplus q^{n+1}} \oplus X$, where $X$ has level $\le n$. Write $X=\bigoplus_{i=1}^m L_i$ with $L_i$ simple. The previous paragraph shows that each $L_i$ is non-trivial. We have
\begin{displaymath}
\chi_{\lambda,t}=q^{n+1} \chi_{\lambda,t/q} + \sum_{i=1}^m \chi_{L_i},
\end{displaymath}
where $\chi_{\lambda,t}$ denotes the character of $L_{\lambda}(t)$. Since the central character of $L_{\lambda}(t)$ is trivial, Proposition~\ref{prop:chi-psi} shows that $\chi_{\lambda,t}$ is the constant function on $\GG$ that takes the value $\udim(L)$ at every point; in other words, $\chi_{\lambda,t}=\udim(L) \cdot \chi_{\bone}$, where $\bone$ is the trivial representation. This continues to hold when we restrict to $G(1)$. Since $t$ is transcendental, it follows that the central character of $L_{\lambda}(t/q)$ is also trivial, and so its character is also constant. We thus find
\begin{displaymath}
\udim(L_{\lambda}(t)) \chi_{\bone}=q^{n+1} \udim(L_{\lambda}(t/q)) \chi_{\bone} + \sum_{i=1}^m \chi_{L_i}.
\end{displaymath}
Applying Proposition~\ref{prop:chi-psi}, we find
\begin{displaymath}
\udim(L_{\lambda}(t)) \epsilon=q^{n+1} \udim(L_{\lambda}(t/q)) \epsilon + \sum_{i=1}^m \udim(L_i) \psi_{L_i}.
\end{displaymath}
Since the $\psi_{L_i}$ are all non-trivial and central characters are linearly independent (Lemma~\ref{lem:char-ind}), the sum on the right must vanish. Thus $\chi_X=0$, and so $X=0$ by Lemma~\ref{lem:GL-char-2}.

We have thus shown that the restriction of $L_{\lambda}(t)$ to $\GL_{t/q}$ is $L_{\lambda}(t/q)^{\oplus q^{n+1}}$. Since $t$ is transcendental, we can apply this iteratively, and conclude that the restriction of $L_{\lambda}(t)$ to $\GL_{t/q^r}$ is $L_{\lambda}(t/q^r)^{\oplus rq^{n+1}}$. Thus $L_{\lambda}(t)$ contains no $G(r)$-invariant vector, which is a contradiction.
\end{proof}

We have now proved Theorem~\ref{thm:GL-char} when $t$ is transcendental. It remains to prove the theorem for regular algebraic values. For this, we require the following observation:

\begin{lemma} \label{lem:char-poly}
Given $\lambda \in \Lambda$ and $g \in \GG$, there exists a rational function $p_{\lambda,g}(T) \in k(T)$, which has poles only at $\{0,q^n\}_{n \in \bN}$, such that $\chi_{\lambda,t}(g)=p_{\lambda,g}(t)$ for all regular parameters $t \in k$.
\end{lemma}

\begin{proof}
We first make a different claim, namely, that for $g \in \GG$ and $h \in H_n$, there exists a rational function $p'_{h,g}(T) \in k(T)$ such that the trace of $(h,g)$ on $A_n$ is given by $p'_{h,g}(t)$ for all regular parameters $t$. Indeed, $A_n=\cC(X_{0,n,0})$ is a permutation representation, and so the trace of $(h,g)$ is $\mu_t(Y)$, where $Y$ is the set of $(h,g)$-fixed points on $X_{0,n,0}$, and this varies with $t$ in the required way by our description of measures. By finite group character theory, it now follows that there is a rational function $p''_{\lambda,g}(T)$ such that the trace of $g$ on the $S_{\lambda}$-isotypic piece of $A_n$ is given by $p_{\lambda,g}(t)$. But, as we have seen, this isotypic piece decomposes as a direct sum of $L_{\lambda}(t)$ and representations of lower level, and so the result follows by induction.
\end{proof}

We finally finish the proof of the theorem.

\begin{proof}[Proof of Theorem~\ref{thm:GL-char}]
Consider the following statement:
\begin{description}[align=right,labelwidth=1.5cm,leftmargin=!]
\item[$\Pi(n)$ ] The characters of simple objects of $\uRep(\GL_t)$ of level $\le n$ are $k$-linearly independent, for any regular parameter $t$.
\end{description}
It suffices to prove $\Pi(n)$ for all $n$. We proceed by induction on $n$. Thus suppose $\Pi(n)$ holds, and let us prove $\Pi(n+1)$.

Enumerate $\Lambda_{n+1}$ as $\lambda_1, \ldots, \lambda_r$. We know that the characters $\chi_{\lambda_i,t}$ are linearly independent when $t$ is transcendental. It follows that we can find $g_1, \ldots, g_r \in \GG$ such that the $r \times r$ matrix $\chi_{\lambda_i,t}(g_j)$ is invertible when $t$ is transcendental. It follows that the determinant of the matrix $p_{\lambda_i,g_j}(T)$ is a non-zero rational function. Thus the matrix $\chi_{\lambda_i,t}(g_j)$ is invertible for all but finitely many $t$. Let $S_0 \subset k$ be the finite set of values for which this matrix is not invertible. Then the characters $\chi_{\lambda_i,t}$ are linearly independent for all regular values $t \not\in S_0$.

Let $S$ be the set of regular values $t \in k$ such that the characters $\chi_{\lambda_i,t}$ are not linearly independent. We must show that $S$ is empty. By the previous paragraph, $S$ is contained in $S_0$, and thus finite. Suppose by way of contradiction that $S$ is non-empty. Choose $t \in S$ such that $t/q \not\in S$, which exists since $S$ is finite. The characters $\chi_{\lambda_i,t/q}$ are then linearly independent. However, there is an upper triangular system relating $\chi_{\lambda_i,t}$ and $\chi_{\lambda_i,t/q}$ (Proposition~\ref{prop:GL-res}), and so the $\chi_{\lambda_i,t}$ are also linearly independent. We thus see that $S$ is empty, which completes the proof.
\end{proof}

\begin{remark}
It would be interesting to find explicit formulas for the character $\chi_{\lambda,t}$. This should be related to Green's classical formulas for the characters of $\GL_n(\bF)$.
\end{remark}

\subsection{Ultraproduct categories} \label{ss:GL-ultra}

Let $k_n$ be an algebraically closed field for each $n \in \bN$, fix a non-principal ultrafilter on $\bN$, and let $k$ be the ultraproduct of the $k_n$'s, which we assume to be of characteristic~0. Let $t$ be the element of $k$ defined by the sequence $(q^n)_{n \ge 0}$. Let $H_n=\GL_n(\bF)$, let $\cT_n=\Rep_{k_n}(H_n)$, let $\tilde{\cT}$ be the ultraproduct of the $\cT_n$'s, and let $\cT$ be the subcategory of $\tilde{\cT}$ defined in \S \ref{ss:ultra}. Let
\begin{displaymath}
\Phi \colon \uRep(\GL_t) \to \cT
\end{displaymath}
be the tensor functor from \S \ref{ss:ultra-ab}. We can finally prove one of our main results:

\begin{theorem} \label{thm:GL-ultra}
The above functor $\Phi$ is an equivalence.
\end{theorem}

\begin{proof}
This follows from Theorems~\ref{thm:ultra-equiv2} and~\ref{thm:GL-char} in the regular case. We now treat the quasi-regular case. Let $G'=G(r)$ be such that $\mu_t$ is regular on $G'$. Of course, $\{X_{\ell,0,0}\}_{\ell \ge r}$ is a smooth approximation of $G$. We have $X_{\ell,0,0}=G/G(\ell)$. Letting $X'_{\ell,0,0}=G'/G(\ell)$, we see that $\{X'_{\ell,0,0}\}_{\ell \ge r}$ is a compatible smooth approximation of $G'$. We have just seen that the natural functor
\begin{displaymath}
\Phi' \colon \uRep(\GL_{t/q^r}) \to \cT'
\end{displaymath}
is an equivalence, where $\cT'$ is the ultraproduct category for $\GL_{t/q^r}$. Thus $\Phi$ is an equivalence by Proposition~\ref{prop:ultra-equiv3}.
\end{proof}

This proves the $\GL$ versions of Conjectures~\ref{conj:interp1} and~\ref{conj:interp2}.

\section{The symplectic group} \label{s:Sp}

\subsection{Set-up}

Let $\bF$ be a finite field of cardinality $q$. Let $\hat{\fC}$ be the following category. Objects are $\bF$-vector spaces $V$ equipped with an alternating bilinear form $\langle, \rangle$. A morphism $f \colon V \to W$ is an injective linear map that respects the form. Let $\fC$ be the full subcategory spanned by finite dimensional objects.

Let $\bV$ be the vector space with basis $\{e_i,f_i\}_{i \ge 1}$. Define an alternating form on $\bV$ by
\begin{displaymath}
\langle e_i, e_j \rangle = \langle f_i, f_j \rangle = 0, \qquad
\langle e_i, f_j \rangle = \delta_{i,j}.
\end{displaymath}
We regard $\bV$ as an object of $\hat{\fC}$, and let $G$ be its automorphism group, which is one version of the infinite symplectic group. The action of $G$ on $\bV$ is oligomorphic, and $\fC$ is the category of structures for $(G, \bV)$ as defined in \S \ref{ss:struct}.

Let $V_{m,n}$ be an $\bF$-vector space of dimension $2m+n$ with a non-degenerate symplectic form on the $2m$ dimensional piece, and the zero form on the $n$-dimensional piece. To be concrete, we can take $V_{m,n}$ to be the subspace of $\bV$ spanned by $e_1, \ldots, e_{n+m}$ and $f_1, \ldots, f_n$. Let $X_{r,s}$ be the space of embeddings $V_{m,n} \to \bV$. One easily sees that every object of $\fC$ is isomorphic to a unique $V_{m,n}$, and so the $X_{m,n}$ are exactly the $\bV$-smooth transitive $G$-sets by Proposition~\ref{prop:struct-equiv}.

\subsection{Enumeration}

Let $N_{a,b}^{m,n}$ be the number of embeddings $V_{a,b} \to V_{m,n}$.

\begin{proposition}
We have the following recurrences for $N$:
\begin{align*}
N_{a,b}^{m,n} &= (q^{2m+n}-q^n) q^{2m+n-1} N_{a-1,b}^{m-1,n} \\
N_{a,b}^{m,n} &= (q^{2m+n}-q^n) q^{2a+b-1} N^{m-1,n}_{a,b-1} + (q^n-1) q^{2a+b-1} N_{a,b-1}^{m,n-1}.
\end{align*}
These are valid for $a \ge 1$ and $b \ge 1$ respectively.
\end{proposition}

\begin{proof}
(a) Giving an embedding of $V_{a,b}$ into $V_{m,n}$ is equivalent to giving an embedding $i$ of $V_{1,0}$ together with an embedding of $V_{a-1,b}$ into the orthogonal complement of $i$. The latter is isomorphic to $V_{m-1,n}$. We thus find
\begin{displaymath}
N_{a,b}^{m,n} = N_{1,0}^{m,n} \cdot N_{a-1,b}^{m-1,n}.
\end{displaymath}
To complete the proof of this identity, we must show
\begin{displaymath}
N_{1,0}^{m,n} = (q^{2m+n}-q^n) q^{2m+n-1}.
\end{displaymath}
Let $x$ and $y$ be a symplectic basis of $V_{1,0}$. Then $i(x)$ can be any vector in $V_{m,n}$ not in the nullspace. There are thus $q^{2m+n}-q^n$ choices for $x$. Fixing a choice for $i(x)$, we see that $i(y)$ can be any vector that pairs to~1 with $i(x)$. The set of choices forms a torsor for the orthogonal complement of $i(x)$, which is a vector space of dimension $2m+n-1$. There are thus $q^{2m+n-1}$ choices for $i(y)$, as required.

(b) Write $V_{a,b} = V_{0,1} \oplus V_{a,b-1}$, and let $x$ be a basis for $V_{0,1}$. Giving an embedding $i$ of $V_{a,b}$, we let $y$ be the image of $x$, and let $j$ be the restriction of $i$ to $V_{a,b-1}$. Thus $i$ is determined by $(y,j)$. We count the number of $i$'s by counting the number of $(y,j)$'s, by picking $y$ first.

First suppose that $y$ belongs to the nullspace; there are $q^n-1$ choices for such $y$. The pair $(y,j)$ comes from an $i$ if and only if the induced map $\ol{j} \colon V_{a,b-1} \to V_{m,n}/\bF y$ is an embedding. The quotient space is isomorphic to $V_{m,n-1}$, and so there are $N_{a,b-1}^{m,n-1}$ choices for $\ol{j}$. Each $\ol{j}$ admits $q^{2a+b-1}$ lifts to a choice of $j$. We thus find that there are $(q^n-1)q^{2a+b-1}N_{0,b-1}^{m,n-1}$ choices of $i$ in this case.

Now suppose that $y$ is not a null vector; there are $q^{2m+n}-q^n$ choices for such $y$. Then $j$ must map into the orthogonal complement $V'$ of $y$. The pair $(y,j)$ comes from an $i$ if and only if the induced map $\ol{j} \colon V_{a,b-1} \to V'/\bF y$ is an embedding. The quotient is isomorphic to $V_{n-1,m}$, and again, each $\ol{j}$ admits $q^{2a+b-1}$ lifts. We thus find $(q^{2m+n}-q^n) q^{2a+b-1} N_{a,b-1}^{m-1,n}$ choices for $i$ in this case, and so the formula follows.
\end{proof}

We define polynomials $Q_{a,b} \in \bZ_q[t,u]$ by putting $Q_{0,0}=1$ and imposing the following recurrences (valid for $a \ge 1$ and $b \ge 1$ respectively):
\begin{align*}
Q_{a,b}(t,u) &= q^{-1} (t-1)tu^2 Q_{a-1,b}(q^{-2}t, u) \\
Q_{a,b}(t,u) &= q^{2a+b-1} (t-1)u Q_{a,b-1}(q^{-2} t, u) + q^{2a+b-1} (u-1) Q_{a,b-1}(t, q^{-1} u).
\end{align*}
Similar to the $\GL$ case, we have:

\begin{proposition} \label{prop:Sp-Q}
The $Q$ polynomials are well-defined, and we have
\begin{displaymath}
N_{a,b}^{m,n} = Q_{a,b}(q^{2m}, q^n)
\end{displaymath}
for all $a,b,m,n \in \bN$.
\end{proposition}

We now define polynomials $P_{a,b} \in \bZ_q[t]$ by $P_{a,b}(t)=Q_{a,b}(t,1)$. These polynomials satisfy a simple recurrence, which is easily solved; we record the answer here:

\begin{proposition} \label{prop:Sp-P}
We have
\begin{displaymath}
P_{a,b}(t) = q^{-e} t^a \prod_{i=0}^{a+b-1} (t-q^{2i}), \qquad
e = 2 \binom{a+b}{2} + a^2 - \binom{b}{2}.
\end{displaymath}
\end{proposition}

\subsection{Burnside rings}

Let $\BB=\BB(G, \bV)$ be the Burnside ring for $G$ relative to the stabilizer class $\sE(\bV)$. Define $x_{a,b}$ to be the class $\lbb X_{a,b} \rbb$ in $\BB$. These elements form a $\bZ$-basis of $\BB$. We now determine the ring structure.

\begin{proposition} \label{prop:Sp-burnside}
We have a ring isomorphism
\begin{displaymath}
\phi \colon \bZ_q[X, Y] \to \BB[1/q], \qquad X \mapsto x_{1,0}, \quad Y \mapsto x_{0,1}.
\end{displaymath}
\end{proposition}

\begin{proof}
The ring $\BB$ carries a filtration, where $F^n \BB$ is the $\bZ$-span of the classes $x_{a,b}$ with $2a+b \le n$. We define a filtration on $\bZ[X,Y]$ by taking $F^n \bZ[X,Y]$ be the $\bZ$-span of the monomials $X^a Y^b$ with $2a+b \le n$. Then $\phi$ is a map of filtered rings, and the filtered pieces are free $\bZ_q$-modules of the same finite rank. Thus it suffices to check that $\phi$ is injective modulo $\ell$, for all $\ell \nmid q$. The additive map
\begin{displaymath}
\psi \colon \BB \to \bZ_q[t,u], \qquad x_{a,b} \mapsto Q_{a,b}
\end{displaymath}
is a ring homomorphism. Since $\psi(x_{1,0})=Q_{1,0}$ and $\psi(x_{0,1})=Q_{0,1}$ are algebraically independent modulo $\ell$, it follows that $\psi \circ \phi$ is injective modulo $\ell$, and so $\phi$ is as well.
\end{proof}

Let $\hat{\BB}$ be the infinitesimal Burnside ring for $G$ relative to $\sE(\bV)$. We now analyze this ring. Let $G(n) \subset G$ be the open subgroup fixing each $e_i$ and $f_i$ for $1 \le i \le n$, and let $\bV(n) \subset \bV$ be the span of the $e_i$ and $f_i$ with $i>n$. Then $(G(n), \bV(n))$ is isomorphic to $(G, \bV)$. We define a ring homomorphism
\begin{displaymath}
\delta \colon \BB \to \BB
\end{displaymath}
by restricting to $G(1)$ and then identifying $G(1)$ with $G$. We now compute this map.

\begin{proposition} \label{prop:Sp-delta}
We have
\begin{align*}
\delta(x_{1,0}) &= q^2 x_{1,0} + (q^3-q) (x_{0,1}+1)^2 \\
\delta(x_{0,1}) &= q^2 x_{0,1}+q^2-1.
\end{align*}
In particular, $\delta$ is an isomorphism after inverting $q$.
\end{proposition}

\begin{proof}
We have $\bV=\bF^2 \oplus \bV(1)$, and so the restriction of $\bV$ to $G(1)$ is isomorphic to $q^2$ copies of $\bV(1)$. We thus see that $\delta([\bV])=q^2 [\bV]$. Since $[\bV]=1+x_{0,1}$, we obtain the stated formula for $\delta(x_{0,1})$.

Now, $X_{1,0}$ is the set of pairs $(v_1,v_2) \in \bV^2$ such that $\langle v_1, v_2 \rangle=1$. Appealing to the decomposition $\bV=\bF^2 \oplus \bV(1)$, we have
\begin{displaymath}
X_{1,0} = \coprod_{u_1,u_2 \in \bF^2} Z(u_1,u_2),
\end{displaymath}
where $Z(u_1,u_2)$ is the set of pairs $(v_1,v_2) \in \bV(1)^2$ such that
\begin{displaymath}
\langle u_1, u_2 \rangle + \langle v_1, v_2 \rangle = 1.
\end{displaymath}
Now, if $\langle u_1, u_2 \rangle \ne 1$ then $Z(u_1,u_2)$ is isomorphic to $X_{1,0}(1)$. Suppose now that $\langle u_1, u_2 \rangle=1$. Then $Z(u_1,u_2)$ is the set of pair $(v_1,v_2) \in \bV(1)^2$ such that $\langle v_1, v_2 \rangle = 0$. The complement of this set in $\bV(1)^2$ is the set of pairs where the pairing is non-zero, which is isomorphic to $q-1$ copies of $X_{1,0}(1)$. There are $(q^2-1)q$ pairs $(u_1,u_2)$ where the pairing is~1, and $q^4-(q^2-1)q$ where the pairing is not~1. We thus find
\begin{displaymath}
\delta(x_{1,0})=(q^4-q^3+q) x_{1,0} + (q^3-q)((1+x_{0,1})^2-(q-1)x_{1,0}),
\end{displaymath}
which yields the stated formula. Since $\BB[1/q]$ is a polynomial ring in $x_{1,0}$ and $x_{0,1}$, it is clear that $\delta$ is an automorphism of $\BB[1/q]$.
\end{proof}

Since $\hat{\BB}$ is the direct limit of iterates of $\delta$, we find:

\begin{corollary} \label{cor:Sp-inf-burnside}
The map $\BB[1/q] \to \hat{\BB}[1/q]$ is an isomorphism.
\end{corollary}

\begin{remark} \label{rmk:Sp-grading}
Define elements of $\BB[1/q]$ by
\begin{displaymath}
w_1=x_{1,0}-q^{-1} x_{0,1} (x_{0,1}+1), \qquad
w_2=x_{0,1}+1.
\end{displaymath}
As in Remark~\ref{rmk:GL-grading}, these are eigenvectors for $\delta$ of eigenvalue $q^2$, and generate $\BB[1/q]$. Thus $\BB[1/q]$ has a canonical grading.
\end{remark}

\subsection{Transitive sets and smooth approximation} \label{ss:Sp-trans}

We now classify the transitive $G$-sets.

\begin{proposition} \label{prop:Sp-trans}
Every transitive $G$-set is isomorphic to one of the form $X_{m,n}/\Gamma$, where $\Gamma$ is a subgroup of $\Aut(X_{m,n})$.
\end{proposition}

\begin{proof}
The proof is similar to that of Proposition~\ref{prop:GL-trans}. We follow the same strategy of classifying equivalence relations, and use the same definition of defect. Lemmas~\ref{lem:GL-trans-1}, \ref{lem:GL-trans-2}, and~\ref{lem:GL-trans-3} remain valid with essentially identical proofs. The proof of Lemma~\ref{lem:GL-trans-4} requires some additional remarks.

Let $V$ be an object of $\fC$ and let $R$ be an equivalence relation on $X(V)$ of positive defect. Then there is a codimension~1 subspace $W$ of $V$ and a pair $(i,j) \in R$ such that $i(x)=j(x)$ for all $x \in W$. We must show that $R$ contains $R_W$. Let $x_1, \ldots, x_n$ be a basis of $V$ such that $x_1, \ldots, x_{n-1}$ is a basis of $W$. In the $\GL$-case, the pairing between $i(x_n)$ and $j(x_n)$ vanished since the two elements had the same parity. This need not be the case anymore, and introduces additional complications: indeed, this constrains which pairs $(i',j')$ can be in the $G$-orbit of $(i,j)$. We explain how to deal with this issue.

Let $a \in \bF$ be given. Let $U=i(V)+j(V)$, and define a linear functional $\lambda_a \colon U \to \bF$ by
\begin{displaymath}
\lambda_a(i(x_{\ell}) \rangle = \langle j(x_n), i(x_{\ell}) \rangle, \quad
\lambda_a(i(x_n) \rangle = a, \quad
\lambda_a(j(x_n) \rangle = \langle j(x_n), i(x_n) \rangle.
\end{displaymath}
Here, and in what follows, $\ell$ is an index in $\{1, \ldots, n-1\}$. Since the pairing on $\bV$ is non-degenerate, there is $y_a \in \bV$ such that $\lambda_a(-)=\langle y_a, - \rangle$. Let $U_a$ be the subspace of $\bV$ spanned by the $i(x_{\ell})$, $j(x_n)$, and $y_a$. We have a linear isomorphism $U \to U_a$ given by
\begin{displaymath}
i(x_{\ell}) \mapsto i(x_{\ell}), \quad
i(x_n) \mapsto j(x_n), \quad
j(x_n) \mapsto y_a.
\end{displaymath}
This map is an isomorphism in $\fC$ by our choice of $y_a$. It follows that there is $g_a \in G$ inducing this isomorphism. We have $j=g_a i$. Put $j_a=g_a j$; note $j_a(x_n)=y_a$. Since $(i,j)$ belongs to $R$, so does $g_a(i,j)=(j,j_a)$, and thus so does $(i,j_a)$ by transitivity. The pair $(i,j_a)$ has defect~1, the embeddings $i$ and $j_a$ agree on $W$, and they satisfy $\langle j_a(x_n), i(x_n) \rangle = a$.

We can now resume following the proof of Lemma~\ref{lem:GL-trans-4}, using an appropriate $j_a$ when required. To say a bit more, suppose $j' \colon V \to \bV$ is an embedding that agrees with $i$ on $W$, and such that $(i,j')$ has defect~1. Let $a=\langle j'(x_n), i(x_n) \rangle$. Then, following the proof of Lemma~\ref{lem:GL-trans-4}, we see that $(i,j')$ belongs to the $G$-orbit of $(i,j_a)$, and thus belongs to $R$. The remainder of the proof is exactly the same.
\end{proof}

\begin{corollary}
The stabilizer class $\sE(\bV)$ is large.
\end{corollary}

We next show that $G$ is smoothly approximable.

\begin{proposition} \label{prop:Sp-smooth}
The $G$-set $X_{a,b}$ is simply connected for all $a$, $b$. The family $\{X_{m,0}\}_{m \ge 0}$ is a smooth approximation of $G$.
\end{proposition}

\begin{proof}
The proofs are similar to those of Propositions~\ref{prop:GL-simply-conn} and~\ref{prop:GL-smooth}.
\end{proof}

\subsection{Measures}

Let $N^{*,0}_{a,b}$ be the function on $\bN$ given by $m \mapsto N^{m,0}_{a,b}$. Since $\{V_{m,0}\}_{m \ge 0}$ is a smooth approximation of $G$ (Proposition~\ref{prop:Sp-smooth}), by \S \ref{ss:count} we obtain a counting measure
\begin{displaymath}
\Theta(G; \bV) \to \Fun^{\circ}(\bN, \bZ), \qquad x_{a,b} \mapsto N^{*,0}_{a,b}.
\end{displaymath}
Since $N^{m,0}_{a,b}=P_{a,b}(q^m)$, it follows that we have a measure
\begin{displaymath}
\phi \colon \Theta(G; \bV) \to \bZ_q[t], \qquad
\phi(x_{a,b})=P_{a,b}(t).
\end{displaymath}
This measure is universal:

\begin{proposition} \label{prop:Sp-Theta}
The map $\phi$ is an isomorphism after inverting $q$.
\end{proposition}

\begin{proof}
The proof is similar to that of Theorem~\ref{thm:GL-measure}. Let $a=[X_{1,0}]$ and $b=[X_{0,1}]$. These classes generate $\Theta(G; \bV)[1/q]$ as a $\bZ_q$-algebra. We now produce a relation between them

Identify $X_{1,0}$ with the set of pairs $(v_1,v_2)$ in $\bV^2$ such that $\langle v_1, v_2 \rangle=1$, and $X_{0,1}$ with $\bV \setminus \{0\}$. We have a map of $G$-sets $X_{1,0} \to X_{0,1}$ via $(v_1,v_2) \mapsto v_1$. Let $F$ be the fiber of this map over $e_1$. Then
\begin{displaymath}
[X_{1,0}]=[F] [X_{0,1}]
\end{displaymath}
holds in $\Theta(G, \bV)$. Now, $F$ is identified with the set of vectors of the form $ae_1+f_1+x$, where $a \in \bF$ is arbitrary and $x \in \bV(1)$. Thus $[F]=q[\bV(1)]$. On the other hand, $[\bV]=q^2 [\bV(1)]$, and so $[F]=q^{-1} [\bV]$. As $[\bV]=b+1$, we find
\begin{displaymath}
a=q^{-1}b(b+1)
\end{displaymath}
In particular, $b$ generates $\Theta(G, \bV)[1/q]$ as a $\bZ_q$-algebra. Since $\mu_t(b)=P_{0,1}=t-1$, the result follows.
\end{proof}

\begin{corollary}
The map $\phi$ induces an isomorphism $\Theta(G) \otimes \bQ \to \bQ[t]$.
\end{corollary}

\begin{proof}
This follows since $\sE(\bV)$ is large; see \S \ref{ss:stab} and \S \ref{ss:Sp-trans}.
\end{proof}

Given a field $k$ of characteristic~0 and an element $t \in k$, we let
\begin{displaymath}
\mu_t \colon \Theta(G) \to k
\end{displaymath}
by the measure obtained by composing the isomorphism $\phi$ with the homomorphism $\bQ[t] \to k$ mapping $t$ to $t$. The analysis of $\Theta(G)$ shows that $\mu_t$ is the unique $k$-valued measure satisfying $\mu_t(X_{0,1})=t-1$, or equivalently, satisfying $\mu_t(\bV)=t$. The restriction of $\mu_t$ to $G(n)$ is $\mu_{t/q^{2n}}$, after identifying $G(n)$ with $G$.

\begin{proposition} \label{prop:Sp-reg}
The measure $\mu_t$ is regular if and only if $t \not\in \{0,q^{2i}\}_{i \ge 0}$, and is quasi-regular if and only if $t \ne 0$.
\end{proposition}

\begin{proof}
The proof is similar to that of Proposition~\ref{prop:GL-reg}.
\end{proof}

We now discuss ultraproduct measures, as in \S \ref{ss:GL-ultra-meas}. The following is the main result we need along these lines:

\begin{proposition} \label{prop:Sp-ultra-meas}
Let $k$ be a field of characteristic~0 and let $t \in k$ be non-zero. Then $\mu_t$ is equivalent to some ultraproduct measure, and in particular, satisfies (Nil).
\end{proposition}

\begin{proof}
Let $\kappa_n$, for $n \in \bN$, be a field, let $\cF$ be a non-principal ultrafilter on $\bN$, and let $\kappa$ be the ultraproduct field, which we assume to have characteristic~0. Let $\nu$ be the ultraproduct measure associated to the smooth approximation $\{X_{m,0}\}_{m \ge 0}$, as in \S \ref{ss:ultra}. Similar to Proposition~\ref{prop:GL-ultra-meas-param}, we have $\nu=\mu_{\tau}$, where $\tau$ is the element of $\kappa$ defined by the sequence $(q^{2n})_{n \ge 0}$. The remainder of the proof is similar to that of Proposition~\ref{prop:ultra-equiv}.
\end{proof}

\subsection{Tensor categories}

For the remainder of \S \ref{s:Sp}, we fix an algebraically closed field $k$ of characteristic~0 and a non-zero element $t \in k$. We say that $t$ is a \defn{regular parameter} if $\mu_t$ is regular, which simply means that $t$ is not an even power of $q$. Let $\uPerm_k(\Sp_t)$ be the category $\uPerm_k(G, \mu_t)$, and let $\uRep_k(\Sp_t)$ be the category $\uRep_k(G, \mu_t)$. We omit the $k$ when possible.

\begin{theorem}
For $t \ne 0$, the category $\uRep(\Sp_t)$ is pre-Tannakian, and the abelian envelope of $\uPerm(\Sp_t)$. If $t$ is a regular parameter then $\uRep(\Sp_t)$ is semi-simple.
\end{theorem}

\begin{proof}
This follows from the general theory in \S \ref{ss:abenv}, combined with Proposition~\ref{prop:Sp-reg} and Proposition~\ref{prop:Sp-ultra-meas}.
\end{proof}

Let $\bV_0$ be the subspace of $\bV$ spanned by the $e_i$'s, and let $\bV_1$ be the subspace spanned by the $f_i$'s. The symplectic form on $\bV$ induces a pairing $\bV_0 \times \bV_1 \to \bF$. Thus $\bV$ can be regarded as an object in the category of structures for $\GL$ defined in \S \ref{ss:GL-setup}. Let $H$ be the automorphism group of $\bV$ in this category (i.e., the infinite general linear group). Elements of $H$ preserve the symplectic form on $\bV$, and so $H$ is a subgroup of $G$.

\begin{proposition} \label{prop:Sp-res-GL}
Let $u \in k$ be a square root of $t$. Then inclusion $H \subset G$ induces a faithful and exact tensor functor
\begin{displaymath}
\uRep(\Sp_t) \to \uRep(\GL_u), \qquad \cC(\bV) \to \cC(\bV)
\end{displaymath}
\end{proposition}

\begin{proof}
Since $\bV$ is a finitary and smooth $H$-set, there is a restriction functor $\bS(G) \to \bS(H)$. Let $G'$ be the set of elements $g \in G$ whose centralizer in $H$ is open. If $g \in G$ fixes all but finitely many basis vectors then $g$ clearly belongs to $G'$. Since the set of such elements is dense in $G$, it follows that $G'$ is dense in $G$. Thus Proposition~\ref{prop:pullback-crit} shows that the condition from Proposition~\ref{prop:pullback} is fulfilled, and so any measure on $H$ pulls back to one on $G$. Since the measure $\mu_u$ on $H$ assigns to $\bV$ the measure $u^2=t$, it follows the pull-back of $\mu_u$ to $G$ assigns to $\bV$ the measure $t$, and is thus the measure $\mu_t$. The result thus follows from Proposition~\ref{prop:res-ab}.
\end{proof}

\begin{remark}
Recall that $t$ is interpolating the value $q^{2n}$ in the finite case. This functor is interpolating the restriction functors $Rep(Sp_{2n}(\mathbb{F}_q)) \to Rep(GL_n(\mathbb{F}_q))$ for the natural embedding of $GL_n(\mathbb{F}_q)$ in $Sp_{2n}(\mathbb{F}_q)$.
\end{remark}

\subsection{Ultraproduct categories}

We now examine the ultraproduct categories associated to the symplectic group. As in the $\GL$ case, we need to first understand central characters:

\begin{theorem} \label{thm:Sp-ultra}
Let $t$ be a regular parameter and let $L$ be a non-trivial simple object in $\uRep(\Sp_t)$. Then the central character of $L$ is non-trivial.
\end{theorem}

\begin{proof}
Let $u$ be a square root of $t$. Since $t$ is not an even power of $q$, it follows that $u$ is not a power of $q$, and is thus a regular parameter for $\GL$. We have a restriction functor
\begin{displaymath}
\uRep(\Sp_t) \to \uRep(\GL_u)
\end{displaymath}
from Proposition~\ref{prop:Sp-res-GL}.

Suppose that $L$ is a simple of $\uRep(\Sp_t)$ with trivial central character. We show that $L$ is trivial. By Proposition~\ref{prop:chi-psi}, the character of $L$ is constant, equal to $\udim(L)$. It follows that the same is true for the character of $L$ as a representation of $\GL_u$. Since the characters of simples of $\GL_u$ are linearly independent (Corollary~\ref{cor:GL-char-2}), it follows that $L$ only contains copies of the trivial representation of $\GL_u$. In particular, $L$ is a finite dimensional vector space. Here we treat $L$ as a smooth module over the completed group algebra $A(G)$.

The trace of any element of $H_n=\Sp_{2n}(\bF)$ on $L$ is equal to $\udim(L)$; note that since $L$ is finite dimensional, the categorical trace coincides with the usual trace. Thus $L$, as an ordinary representation of $H_n$, has constant character. It follows that $H_n$ acts trivially on $L$. Thus $G$ acts trivially on $L$: indeed, the union of the $H_n$'s is a dense subgroup of $G$, and the action of $G$ is smooth. Let $x_1, \ldots, x_d$ be a basis for $L$. Since $L$ is a smooth module over $A(G)$, for each $i$ there is an open subgroup $U_i$ such that $A(U_i)$ acts trivially on $x_i$. Let $U=U_1 \cap \cdots \cap U_d$. Then $A(U)$ acts trivially on $L$. Since $A(G)$ is generated by $A(U)$ and $G$ \cite[Proposition~11.20]{repst}, it follows that $A(G)$ acts trivially on $L$, completing the proof.
\end{proof}

Now, let $k_n$ be an algebraically closed field for each $n \in \bN$, fix a non-principal ultrafilter on $\bN$, and let $k$ be the ultraproduct of the $k_n$'s, which we assume to be of characteristic~0. Let $t$ be the element of $k$ defined by $(q^{2n})_{n \ge 0}$. Let $H_n=\Sp_{2n}(\bF)$, let $\cT_n=\Rep_{k_n}(H_n)$, let $\tilde{\cT}$ be the ultraproduct of the $\cT_n$'s, and let $\cT$ be the subcategory $\tilde{\cT}$ defined in \S \ref{ss:ultra}. Let
\begin{displaymath}
\Phi \colon \uRep(\Sp_t) \to \cT
\end{displaymath}
be the tensor functor from \S \ref{ss:ultra-ab}. We can finally prove one of our main results:

\begin{theorem}
The above functor $\Phi$ is an equivalence.
\end{theorem}

\begin{proof}
The proof is similar to that of Theorem~\ref{thm:GL-ultra}.
\end{proof}

This proves Conjectures~\ref{conj:interp1} and~\ref{conj:interp2}.

\begin{remark}
From here a next natural question is to classify the simple objects in these categories.  This seems likely to be doable using similar methods to the $\GL$ case, but we do not pursue it further here.   This will also be the case for the orthogonal and unitary cases below.
\end{remark}

\section{The orthogonal group} \label{s:O}

\subsection{Set-up}

Let $\bF$ be a finite field of odd cardinality $q$. Fix a non-square $\pi \in \bF^{\times}$, and let $\chi \colon \bF^* \to \{\pm 1\}$ be the quadratic character. Let $\hat{\fC}$ be the following category. Objects are quadratic spaces, i.e., $\bF$-vector spaces equipped with a symmetric bilinear form. Morphisms are injective linear maps that respect the forms. We let $\fC$ be the full subcategory spanned by finite dimensional spaces.

Let $\bV$ be the vector space with basis $\{e_i, f_i\}_{i \ge 1}$. Define a symmetric bilinear form on $\bV$ by
\begin{displaymath}
\langle e_i, e_j \rangle = \langle f_i, f_j \rangle = 0, \qquad
\langle e_i, f_j \rangle = 1.
\end{displaymath}
Thus $\bV$ is an orthogonal direct sum of countably many copies of the hyperbolic plane. We regard $\bV$ as an object of $\hat{\fC}$, and let $G$ be its automorphism group, which is one version of the infinite orthogonal group. The action of $G$ on $\bV$ is oligomorphic, and $\fC$ is the category of structures for $(G, \bV)$ as defined in \S \ref{ss:struct}

For a finite dimensional quadratic space $V$, we define the \defn{modified discriminant} of $V$ to be $(-1)^r \delta$, where $\delta$ is the usual discriminant, and $\dim(V)$ is equal to $2r$ or $2r+1$. If $U$ is a hyperbolic plane then $V$ and $V \oplus U$ have equal modified discriminant; this is the point of the definition. Let $V_{m,\delta}$ be the non-degenerate quadratic space of dimension $m$ and modified discriminant $\delta \in \bF^{\times}/(\bF^{\times})^2$. Let $V_{m,\delta,n}$ be the orthogonal direct sum of $V_{m,\delta}$ and an $n$-dimensional space with zero form, and let $X_{m,\delta,n}$ be the space of embeddings $V_{m,\delta,n} \to \bV$. One easily sees that every object of $\fC$ is isomorphic to a unique $V_{m,\delta,n}$ (unless $m=0$ in which case $\delta$ is irrelevant), and so the $X_{m,\delta,n}$ are the transitive $\bV$-smooth $G$-sets by Proposition~\ref{prop:struct-equiv}.

\subsection{Enumeration of norms}

For a quadratic space $V$, let $N_a(V)$ be the number of elements $v \in V$ with $\langle v, v \rangle = a$. We will often use the following identity:
\begin{displaymath}
N_a(V \oplus W) = \sum_{a=b+c} N_b(V) N_c(W).
\end{displaymath}

\begin{proposition}
For $a \ne 0$, we have
\begin{align*}
N_a(V_{2r,\delta}) &= q^{2r-1}-\chi(\delta) q^{r-1} \\
N_0(V_{2r,\delta}) &= q^{2r-1}+\chi(\delta) (q^r-q^{r-1}) \\
N_a(V_{2r+1,\delta}) &= q^{2r} + \chi(a \delta) q^r \\
N_0(V_{2r+1,\delta}) &= q^{2r}
\end{align*}
\end{proposition}

\begin{proof}
We break the analysis into three cases:

\textit{Case 1: split.}
The space $V_{2r,1}$ is split. If we multiple the form by a non-zero scalar, the space is still split. It follows that $N_a(V_{2r,1})$ is independent of $a$, for $a \ne 0$. Since each of the $q^{2r}$ elements of $V_{2r,1}$ has some norm, we find
\begin{displaymath}
N_0(V_{2r,1}) + (q-1) N_1(V_{2r,1}) = q^{2r}.
\end{displaymath}
Now, we have $V_{2r,1}=V_{2,1} \oplus V_{2r-2,1}$, and so
\begin{align*}
N_1(V_{2r,1})
&= \sum_{1=b+c} N_b(V_{2,1}) N_c(V_{2r-2,1}) \\
&= N_0(V_{2,1}) N_1(V_{2r-2,1}) + N_1(V_{2,1}) N_0(V_{2r-2,1}) + (q-2) N_1(V_{2,1}) N_1(V_{2r-2,1}).
\end{align*}
A simple computation shows $N_1(V_{2,1})=q-1$ and $N_0(V_{2,1}) = 2q-1$. Using our relation between $N_0$ and $N_1$, the above equation yields a recurrence for $N(V_{2r,1})$, which leads to our stated value. We can then solve for $N_0(V_{2r,1})$ using the relation.

\textit{Case 2: even dimension and non-split.}
We have $V_{2r,\pi}=V_{2,\pi} \oplus V_{2r-2,1}$, and so
\begin{displaymath}
N_a(V_{2r,\pi}) = \sum_{a=b+c} N_b(V_{2,\pi}) N_c(V_{2r-2,1}).
\end{displaymath}
Now, $N_b(V_{2,\pi})$ is $q+1$ if $b \ne 0$, and~1 if $b=0$; this is easy to see since the form here is the norm form for the extension $\bF_{q^2}/\bF_q$. For $a \ne 0$, we have
\begin{displaymath}
N_a(V_{2r,\pi}) = (q+1) N_0(V_{2r-2,1}) + N_1(V_{2r-2,1}) + (q-2)(q+1) N_1(V_{2r-2,1}).
\end{displaymath}
Here the first and second terms come from taking $(b,c)$ to be $(a,0)$ and $(0,a)$, and the third term is all that remains. This simplifies to the stated formula. For $a=0$, we have
\begin{displaymath}
N_0(V_{2r,\pi}) = N_0(V_{2r-2,1}) + (q-1) (q+1) N_1(V_{2r-2,1}).
\end{displaymath}
Here the first term comes from $(b,c)=(0,0)$, and the second term is all that remains. This again simplifies to the stated formula.

\textit{Case 3: odd dimension}
We have $V_{2r+1,\delta} = V_{1,\delta} \oplus V_{2r,1}$, and so
\begin{displaymath}
N_a(V_{2r+1,\delta}) = \sum_{a=b+c} N_b(V_{1,\delta}) N_c(V_{2r,1}).
\end{displaymath}
Now, $N_b(V_{1,\delta})$ is $1+\chi(b \delta)$ if $b \ne 0$, and is~1 if $b=0$. If $a \ne 0$, we find
\begin{displaymath}
N_a(V_{2r+1,\delta})
= N_1(V_{2r,1}) + 2\sum_{b \in \delta (\bF^{\times})^2} N_{a-b}(V_{2r,1})
\end{displaymath}
Now, if $a \delta$ is not a square then $a-b$ is non-zero for all $b$ above; if $a \delta$ is a square, then there is a unique $b$ for which $a-b=0$. We thus find
\begin{displaymath}
N_a(V_{2r+1,\delta}) = N_1(V_{2r,1}) + (q-1) N_1(V_{2r,1}) + (1+\chi(a \delta) (N_0(V_{2r,1})-N_1(V_{2r,1}))
\end{displaymath}
This simplifies to the stated formula. For $a=0$, we have
\begin{displaymath}
N_0(V_{2r+1,\delta}) = N_0(V_{2r,1}) + (q-1) N_1(V_{2r,1}).
\end{displaymath}
This again simplifies to the stated formula.
\end{proof}

\subsection{Enumeration of embeddings}

Let $N_{a,\epsilon,b}^{\ell,\delta,n}$ be the number of embeddings $V_{a,\epsilon,b} \to V_{\ell,\delta,n}$. Note that
\begin{displaymath}
N_{1,\epsilon,0}^{\ell,\delta,n} = q^n \cdot N_{\epsilon}(V_{\ell,\delta}), \qquad
N_{0,0,1}^{\ell,\delta,n} = q^n \cdot N_0(V_{\ell,\delta})-1.
\end{displaymath}
There is one additional base case we require:

\begin{proposition}
We have:
\begin{align*}
N_{2,\pi,0}^{2r,\delta,n} &= (q^{2r}-\chi(\delta) q^r)(q^{2r-1}-\chi(\delta) q^r) q^{2n-2} \\
N_{2,\pi,0}^{2r+1,\delta,n} &= (q^{4r}-q^{2r}) q^{2n-1}
\end{align*}
\end{proposition}

\begin{proof}
The space $V_{2,\pi,0}$ has (ordinary) discriminant $-\pi$, and thus has an orthogonal basis consisting of two vectors with norms 1 and $-\pi$. To give an embedding of $N_{2,\pi,0}$ into a quadratic space $V$ thus amounts to giving a norm~1 vector $v$ in $V$, and a norm~$(-\pi)$ vector $w$ in the orthogonal complement $V'$ of $v$. Since the isomorphism type of $V'$ is independent of $v$, we see that the number of such embeddings is $N_1(V) N_{-\pi}(V')$. If $V=V_{2r,\delta,n}$ then $V'=V_{2r-1,-\delta,n}$ while if $V=V_{2r+1,\delta,n}$ then $V'=V_{2r,\delta,n}$; note that the sign in the first case comes since we are using the modified discriminant. The result now follows from our previous computations.
\end{proof}

We now derive some recurrences for $N$.

\begin{proposition}
We have:
\begin{align*}
N_{a,\epsilon,b}^{\ell,\delta,n} &= (N_0(V_{\ell,\delta})-1) q^{\ell+2n-2} N_{a-2,\epsilon,b}^{\ell-2,\delta,n} \\
N_{a,\epsilon,b}^{\ell,\delta,n} &= (N_0(V_{\ell,\delta})-1) q^{n+a+b-1} N_{a,\epsilon,b-1}^{\ell-2,\delta,n} + (q^n-1) q^{a+b-1} N_{a,\epsilon,b-1}^{\ell,\delta,n-1}
\end{align*}
For the first identity we require $a \ge 3$, or $a=2$ and $\epsilon=1$; for the second, we require $b \ge 1$.
\end{proposition}

\begin{proof}
(a) We have $V_{a,\epsilon,b}=V_{a-2,\epsilon,b} \oplus V_{2,1,0}$. To give an embedding of this space into $V_{\ell,\delta,n}$, we give an embedding $i$ of $V_{2,1,0}$ together with an embedding of $V_{a-2,\epsilon,b}$ into the orthogonal complement of $i$, which is isomorphic to $V_{\ell-2,\delta,n}$. We thus find
\begin{displaymath}
N_{a,\epsilon,b}^{\ell,\delta,n} = N_{2,1,0}^{\ell,\delta,n} \cdot N_{a-2,\epsilon,b}^{\ell-2,\delta,n}.
\end{displaymath}
It remains to compute the first factor.

Let $X$ be the set of pairs $(v,w)$ of vectors in $V_{\ell,\delta,n}$ such that $v$ and $w$ are isotropic, and $\langle v, w \rangle=1$. Then $N_{2,1,0}^{\ell,\delta,n}$ is the cardinality of $X$. Define $Y$ like $X$, except drop the condition that $w$ be isotropic. The additive group $\bF$ acts on $Y$ by $a \cdot (v,w) = (v,w+av)$. This action is free since $v$ and $w$ are linearly independent (as $\langle v,v \rangle=0$ but $\langle v,w \rangle=1$). Every oribt meets $X$ exactly once. We thus see that $\# X=q^{-1} \#Y$.

We now count the size of $Y$. The vector $v$ can be any isotropic vector of $V_{\ell,\delta,n}$ not in the nullspace. The number of such vectors is $(N_0(V_{\ell,\delta})-1) q^n$. Given $v$, the only constraint on $w$ is $\langle v , w \rangle=1$. The set of such vectors is a torsor for the orthogonal complement of $v$, which is a vector space of dimension $\ell+n-1$; there are thus $q^{\ell+n-1}$ choices for $w$. The result follows.

(b) Write $V_{a,\epsilon,b} = V_{a,\epsilon,b-1} \oplus V_{0,1,1}$, and let $x$ be a basis of $V_{0,1,1}$. Giving an embedding $i$ of $V_{a,\epsilon,b}$ into $V_{\ell,\delta,n}$, let $y$ be the image of $x$ and let $j$ be the restriction of $i$ to $V_{a,\epsilon,b-1}$. Then $i$ is determined by the pair $(y,j)$. Such a pair comes from an $i$ if and only if $y$ is non-zero isotropic, and the image of $j$ is orthogonal to and linearly independent from $y$. We count how many pairs $(y,j)$ occur.

First consider the case where $y$ belongs to the nullspace of $V_{\ell,\delta,n}$. There are $q^n-1$ choices for such $y$. A pair $(y,j)$ comes from an $i$ if and only if the induced map $\ol{j} \colon V_{a,\epsilon,b-1} \to V_{\ell,\delta,n}/\bF y$ is an embedding in $\fC$. The quotient space is isomorphic to $V_{\ell,\delta,n-1}$, and so there are $N_{a,\epsilon,b-1}^{\ell,\delta,n-1}$ choices for $\ol{j}$. Each $\ol{j}$ admits exactly $q^{a+b-1}$ lifts to a $j$. Thus the total number of $i$'s in this case is the second term in the stated formula.

Now suppose $y$ does not belong to the nullspace. There are $(N_0(V_{\ell,\delta})-1)q^n$ choices for such $y$. Let $V'$ be the orthogonal complement of $y$. A pair $(y,j)$ comes from an $i$ if and only if $j$ maps into $V'$ and the induced map $\ol{j} \colon V_{a,\epsilon,b-1} \to V'/\bF y$ is an embedding. The quotient space is always isomorphic to $V_{\ell-2,\delta,n}$, and so there are $N_{a,\epsilon,b-1}^{\ell-2,\delta,n}$ choices for $\ol{j}$. Each $\ol{j}$ admits $q^{a+b-1}$ lifts to a $j$. Thus the total number of $i$'s in this case is the first term in the stated formula. The result follows.
\end{proof}

We now define polynomials $Q^*_{a,\epsilon,b}$ in $\bZ_q[t,u]$, where $*$ is either ``even'' or ``odd.'' To begin, we put $Q^*_{0,1,0} = 1$ and
\begin{align*}
Q^{\rm even}_{2,\pi,0} &= q^{-3} (t-1)(t-q) t^2 u^2 &
Q^{\rm even}_{1,\epsilon,0} &= q^{-1} (t-1) tu \\
Q^{\rm odd}_{2,\pi,0} &= q^{-1}(t^2-1) t^2 u^2 &
Q^{\rm odd}_{1,\epsilon,0} &= (t+\chi(\epsilon)) tu
\end{align*}
We then impose the following recurrences: 
\begin{align*}
Q^{\rm even}_{a,\epsilon,b}(t,u) &= q^{-3} (t+q)(t-1)t^2 u^2 Q^{\rm even}_{a-2,\epsilon,b}(q^{-1}t, u) \\
Q^{\rm even}_{a,\epsilon,b}(t,u) &=  q^{a+b-2} (t+q)(t-1)u Q^{\rm even}_{a,\epsilon,b-1}(q^{-1} t, u) + (u-1) q^{a+b-1} Q^{\rm even}_{a,\epsilon,b-1}(t, q^{-1} u) \\
Q^{\rm odd}_{a,\epsilon,b}(t,u) &= q^{-1} (t^2-1) t^2 u^2 Q^{\rm odd}_{a-2,\epsilon,b}(q^{-1} t, u) \\
Q^{\rm odd}_{a,\epsilon,b}(t,u) &= q^{a+b-1} (t^2-1) u Q^{\rm odd}_{a,\epsilon,b-1}(q^{-1} t, u) + (u-1) q^{a+b-1} Q^{\rm odd}_{a,\epsilon,b-1}(t, q^{-1} u)
\end{align*}
For the first and third, we require $a \ge 3$, or $a=2$ and $\epsilon=1$; for the second and fourth, we require $b \ge 1$.

\begin{proposition}
The $Q$ polynomials are well-defined, and we have
\begin{displaymath}
N_{a,\epsilon,b}^{2r,\delta,n} = Q^{\rm even}_{a,\epsilon,b}(\chi(\delta) q^r, q^n) \qquad
N_{a,\epsilon,b}^{2r+1,\delta,n} = Q^{\rm odd}_{a,\epsilon,b}(\chi(\delta) q^r, q^n)
\end{displaymath}
for all parameter values.
\end{proposition}

\begin{proof}
The proof is similar to that of Proposition~\ref{prop:GL-Qpoly}.
\end{proof}

We also define polynomials $P^*_{a,\delta,b}$ in $\bZ_q[t]$ by
\begin{displaymath}
P^*_{a,\epsilon,b}(t)=Q^*_{a,\epsilon,b}(t,1).
\end{displaymath}
We compute these polynomials explicitly:

\begin{proposition}
We have
\begin{displaymath}
P^{\rm even}_{a,\epsilon,b}(t) = q^{-e} t^a \prod_{i=1}^{a+2b} (t-\rho_i), \qquad
P^{\rm odd}_{a,\epsilon,b}(t) = q^{-f} t^a \prod_{i=1}^{a+2b} (t-\sigma_i),
\end{displaymath}
where
\begin{align*}
e &= \binom{a+b+1}{2} & f &= \binom{a+b}{2} \\
\{ \rho_1, \ldots, \rho_{2b} \} &= \{1, \pm q, \pm q^2, \ldots, \pm q^{b-1}, -q^b \} &
\{ \sigma_1, \ldots, \sigma_{2b} \} &= \{\pm 1, \ldots, \pm q^{b-1}\} \\
\rho_{2b+1} &= q^b &
\sigma_{2b+1+i} &= (-1)^{i+1} \chi(\epsilon) q^{b+\lfloor i/2 \rfloor} \\
\rho_{2b+2+i} &= (-1)^{i+1} \chi(\epsilon) q^{b+1+\lfloor i/2 \rfloor}.
\end{align*}
for $i \ge 0$.
\end{proposition}

Note that if $a=2m-1$ is odd and positive then
\begin{displaymath}
\{ \rho_1, \ldots, \rho_{2b+a} \} = \{ 1, \pm q, \ldots, \pm q^{b+m-1} \},
\end{displaymath}
while if $a=2m$ is even and positive then there is one additional $\rho$ root, namely $-\chi(\epsilon) q^{b+m}$. Similarly, if $a=2m$ is even then
\begin{displaymath}
\{ \sigma_1, \ldots, \sigma_{2b+a} \} = \{ \pm 1, \ldots, \pm q^{b+m-1} \},
\end{displaymath}
while if $a=2m+1$ is odd then there is one additional $\sigma$ root, namely $-\chi(\epsilon) q^{b+m}$. In particular, $P^{\rm even}_{a,\epsilon,b}$ is independent of $\epsilon$ when $a$ is odd, and $P^{\rm odd}_{a,\epsilon,b}$ is independent of $\epsilon$ when $a$ is even.

\subsection{Burnside rings}

Let $\BB=\BB(G, \bV)$ be the Burnside ring for $G$ relative to $\sE(\bV)$. Define $x_{a,\epsilon,b}$ to be the class $\lbb X_{a,\epsilon,b} \rbb$; these elements form a $\bZ$-basis of $\BB$. We now determine the ring structure when $2q$ is inverted (note that in all other cases we only invert $q$):

\begin{proposition} \label{prop:O-burnside}
We have a ring isomorphism
\begin{displaymath}
\phi \colon \bZ_{2q}[X,Y,Z]/(F) \to \BB[1/2q], \qquad
X \mapsto x_{1,1,0}, \quad Y \mapsto x_{1,\pi,0}, \quad Z \mapsto x_{0,0,1}
\end{displaymath}
where
\begin{displaymath}
F = (1-\tfrac{1}{2}X-\tfrac{1}{2}Y+Z)(X-Y).
\end{displaymath}
\end{proposition}

\begin{proof}
The ring $\BB$ is filtered, where $F^n \BB$ is spanned by the classes $x_{a,\epsilon,b}$ with $a+b \le n$. Define a ring homomorphism
\begin{displaymath}
\tilde{\phi} \colon \bZ_{2q}[X,Y,Z] \to \BB[1/2q], \qquad
X \mapsto x_{1,1,0}, \quad Y \mapsto x_{1,\pi,0}, \quad Z \mapsto x_{0,0,1}.
\end{displaymath}
This is a homomorphism of filtered rings, if we define $X$, $Y$, and $Z$ to have filtration level~1. We have ring homomorphisms
\begin{displaymath}
\psi^* \colon \BB[1/2q] \to \bZ_{2q}[t,u], \qquad x_{a,\epsilon,b} \mapsto Q^{\ast}_{a,\epsilon,b},
\end{displaymath}
where $\ast$ is ``even'' or ``odd.'' Let
\begin{displaymath}
\psi \colon \BB[1/2q] \to \bZ_{2q}[t,u] \times \bZ_{2q}[t,u]
\end{displaymath}
be the product map $\psi^{\rm even} \times \psi^{\rm odd}$. Explicitly, we have
\begin{align*}
\psi(\tilde{\phi}(X)) &= (q^{-1}(t^2-t)u, t^2u+tu), \\
\psi(\tilde{\phi}(Y)) &= (q^{-1}(t^2-t)u, t^2u-tu), \\
\psi(\tilde{\phi}(Z)) &= (q^{-1}(t^2-t)u+tu-1, t^2u-1).
\end{align*}
Put
\begin{displaymath}
F^{\rm even} = X-Y, \qquad F^{\rm odd} = 1-\tfrac{1}{2}X-\tfrac{1}{2}Y+Z,
\end{displaymath}
so that $F=F^{\rm even} \cdot F^{\rm odd}$. The above formulas show that $F^*$ belongs to the kernel of $\psi^* \circ \tilde{\phi}$, and so $F$ belongs to the kernel of $\psi \circ \tilde{\phi}$.

Fix a prime $\ell \nmid 2q$. We claim that the map
\begin{displaymath}
\psi^* \circ \tilde{\phi} \colon \bZ_{2q}[X,Y,Z]/(F^*) \to \bZ_{2q}[t,u]
\end{displaymath}
is injective modulo $\ell$. Indeed, the source is simply the polynomial ring on $Y$ and $Z$, and the images of $Y$ and $Z$ in the target are algebraically independent modulo $\ell$, so the claim follows. It follows that the map
\begin{displaymath}
\psi \circ \tilde{\phi} \colon \bZ_{2q}[X,Y,Z]/(F) \to \bZ_{2q}[t,u] \times \bZ_{2q}[t,u]
\end{displaymath}
is injective modulo $\ell$.

Let $S$ be the set of all monomials of the form $X^aY^bZ^c$ with $a \in \{0,1\}$. Then $S$ forms a basis for $\bZ_{2q}[X,Y,Z]/(F)$ modulo $\ell$, as $F$ is quadratic degree polynomial in $X$ and the coefficient of $X^2$ is a unit. It follows that the image of $S$ under $\psi \circ \tilde{\phi}$ is linearly independent modulo $\ell$, and so the image of $S$ under $\tilde{\phi}$ is also linearly independent modulo $\ell$. Now, the set $F^n S$ of monomials $\{X^aY^bZ^c\}$ with $a \in \{0,1\}$ and $a+b+c \le n$ has cardinality $2n+1$, which coincides with the dimension of the filtered piece $F^n \BB$. It follows that $\tilde{\phi}(F^n S)$ forms a basis of $F^n \BB$ modulo $\ell$, and so the map
\begin{displaymath}
F^n \tilde{\phi} \colon F^n \bZ_{2q}[X,Y,Z] \to F^n \BB[1/2q]
\end{displaymath}
is surjective modulo $\ell$. Since this holds for all $\ell \nmid 2q$, it follows that $F^n \tilde{\phi}$ is surjective. This implies that $\tilde{\phi}$ is surjective too.

The kernel of $\tilde{\phi}$ is contained in the kernel of $\psi \circ \tilde{\phi}$, which is $(F)$. Consider the map
\begin{displaymath}
F^2 \tilde{\phi} \colon F^2 \bZ_{2q}[X,Y,Z] \to F^2 \BB[1/2q].
\end{displaymath}
The source is a free $\bZ_{2q}$ module of rank~10, while the target is a free $\bZ_{2q}$-module of rank~9. Thus the kernel is non-trivial. Since the kernel is contained in the second filtered piece of $(F)$, it follows that the kernel contains $F$. We thus see that $\tilde{\phi}$ induces a ring homomorphism $\phi$ as in the statement of the proposition. Since $\tilde{\phi}$ is surjective, so is $\phi$; and since $\psi \circ \tilde{\phi}$ has kernel $(F)$, it follows that $\phi$ is injective, which completes the proof.
\end{proof}

Let $\hat{\BB}$ be the infinitesimal Burnside ring for $G$ relative to $\sE(\bV)$. We now analyze this ring. Let $G(n) \subset G$ be the open subgroup fixing each $e_i$ and $f_i$ for $1 \le i \le n$, and let $\bV(n) \subset \bV$ be the span of the $e_i$ and $f_i$ with $i>n$. Then $(G(n), \bV(n))$ is isomorphic to $(G, \bV)$. We define a ring homomorphism
\begin{displaymath}
\delta \colon \BB \to \BB
\end{displaymath}
by restricting to $G(1)$ and then identifying this with $G$. We now compute this map.

\begin{proposition}
We have
\begin{align*}
\delta(x_{1,1,0}) &= q x_{1,1,0} + (q-1) z \\
\delta(x_{1,\pi,0}) &=  q x_{1,\pi,0} + (q-1) z \\
\delta(x_{0,0,1}) &= qx_{0,0,1}+(q-1)(z+1)
\end{align*}
where
\begin{displaymath}
z = 1+x_{0,0,1} + \tfrac{1}{2}(q-1) x_{1,1,0} + \tfrac{1}{2}(q-1)x_{1,\pi,0}.
\end{displaymath}
In particular, $\delta$ is an isomorphism after inverting $2q$.
\end{proposition}

\begin{proof}
For a quadratic space $V$, let $Y_a(V)$ be the set of vectors in $a$ of norm $a$. We have
\begin{displaymath}
Y_a(V \oplus W) = \coprod_{a=b+c} Y_b(V) \times Y_c(W)
\end{displaymath}
for an orthogonal direct sum. If $V$ is finite then the cardinality of $Y_a(V)$ is the number $N_a(V)$ previously studied. Now, we have $\bV = W \oplus \bV(1)$, where $W$ is a two-dimensional space isomorphic to $V_{2,1,0}$. We thus find
\begin{displaymath}
\lbb Y_a(\bV) \rbb = \sum_{a=b+c} N_b(W) \lbb Y_c(\bV(1)) \rbb.
\end{displaymath}
As we have already seen, $N_b(W)$ is $q-1$ if $b \ne 0$, and is $2q-1$ if $b=0$. We thus have
\begin{displaymath}
\lbb Y_a(\bV) \rbb = q \lbb Y_a(\bV(1)) \rbb + (q-1) \sum_{c \in \bF} \lbb Y_c(\bV(1)) \rbb
\end{displaymath}
We have
\begin{displaymath}
\lbb \bV \rbb = \sum_{c \in \bF} \lbb Y_c(\bV) \rbb = \lbb Y_0(\bV) \rbb + \tfrac{1}{2} (q-1) \lbb Y_1(\bV) \rbb + \tfrac{1}{2} (q-1) \lbb Y_{\pi}(\bV) \rbb
\end{displaymath}
and similarly with $\bV(1)$ in place of $\bV$. Since
\begin{displaymath}
x_{1,1,0} = \lbb Y_1(\bV) \rbb, \qquad
x_{1,\pi,0} = \lbb Y_{\pi}(\bV) \rbb, \qquad
x_{0,0,1} = \lbb Y_0(\bV) \rbb - 1,
\end{displaymath}
the formulas for $\delta$ follow.

We now show that $\delta$ is an isomorphism. Since $z=\lbb \bV \rbb$, we have $\delta(z)=q^2 z$. This shows that $z$ belongs to the image of $\delta$ (when $q$ is inverted). From the formulas for $\delta$, we thus see that the image of $\delta$ contains each of $x_{1,1,0}$, $x_{1,\pi,0}$, and $x_{0,0,1}$ (when $q$ is inverted). Since these elements generate $\BB[1/2q]$ as a $\bZ_q$-algebra, it follows that $\delta \colon \BB[1/2q] \to \BB[1/2q]$ is surjective, and so it is an isomorphism \stacks{06RN}.
\end{proof}

As $\hat{\BB}$ is the direct limit of iterates of $\delta$, we find:

\begin{corollary}
The natural map $\BB[1/2q] \to \hat{\BB}[1/2q]$ is an isomorphism.
\end{corollary}

\begin{remark}
Let $z$ be as above, and put
\begin{displaymath}
w_1 = x_{1,1,0} - q^{-1} z, \qquad
w_2 = x_{1,\pi,0} - q^{-1} z.
\end{displaymath}
Then $w_1$, $w_2$, and $z$ generate $\BB[1/2q]$ as a $\bZ_q$-algebra. Moreover, $w_1$ and $w_2$ are $q$-eigenvectors of $\delta$, while $z$ is a $q^2$-eigenvector of $\delta$. It follows that $\delta$ induces a grading on the ring $\BB[1/2q]$.
\end{remark}

\subsection{Transitive sets and smooth approximation}

Once again, we classify transitive $G$-sets, and obtain a smooth approximation.

\begin{proposition} \label{prop:O-trans}
Every transitive $G$-set is isomorphic to one of the form $X_{m,\delta,n}/\Gamma$, where $\Gamma$ is a subgroup of $\Aut(X_{m,\delta,n})$. In particular, $\sE(\bV)$ is a large stabilizer class.
\end{proposition}

\begin{proof}
The proof is similar to the symplectic case (Proposition~\ref{prop:Sp-trans}).
\end{proof}

\begin{proposition} \label{prop:Sp-smooth}
The $G$-set $X_{m,\delta,n}$ is simply connected for all $m$, $\delta$, $n$. The family $\{X_{m,\delta}\}$, with $m \ge 0$ and $\delta \in \bF^*/(\bF^*)^2$, is a smooth approximation of $G$.
\end{proposition}

\begin{proof}
The proofs are similar to those of Propositions~\ref{prop:GL-simply-conn} and~\ref{prop:GL-smooth}.
\end{proof}

\subsection{Measures}

Let $N^{\rm even}_{a,\epsilon,b}$ be the function on $\bN$ defined by $r \mapsto N^{2r,1,0}_{a,\epsilon,b}$. Similarly, let $N^{\rm odd}_{a,\epsilon,b}$ by the function $r \mapsto N^{2r+1,1,0}_{a,\epsilon,b}$. The spaces $V_{\ell,\delta,0}$ are homogeneous objects of $\cC$. Thus, by \S \ref{ss:count}, we obtain two counting measures
\begin{displaymath}
\Theta(G; \bV) \to \Fun^{\circ}(\bN, \bZ), \qquad x_{a,\epsilon,b} \mapsto N^*_{a,\epsilon,b}
\end{displaymath}
where $*$ is either ``even'' or ``odd.'' Since $N$ can be expressed in terms of the $P$ polynomials, we obtain two measures
\begin{displaymath}
\phi^* \colon \Theta(G; \bV) \to \bZ_q[t], \qquad x_{a,\epsilon,b} \mapsto P^*_{a,\epsilon,b}(t)
\end{displaymath}
Taken together, these two measures are universal:

\begin{theorem} \label{thm:O-meas}
The map
\begin{displaymath}
\phi^{\rm even} \times \phi^{\rm odd} \colon \Theta(G, \bV)[1/2q] \to \bZ_{2q}[t] \times \bZ_{2q}[t]
\end{displaymath}
is an injection, and the image consists of pairs $(f,g)$ such that $f(0)=g(0)$. In particular, we have a ring isomorphism
\begin{displaymath}
\Theta(G, \bV)[1/2q] \cong \bZ_{2q}[t,t']/(tt').
\end{displaymath}
\end{theorem}

\begin{proof}
Let $a=[X_{1,1,0}]$ and $b=[X_{1,\pi,0}]$ and $c=[X_{0,1,1}]$, and put
\begin{displaymath}
u=c+1-\tfrac{1}{2} a-\tfrac{1}{2} b, \qquad v=\tfrac{1}{2} (a-b).
\end{displaymath}
Letting $\phi=\phi^{\rm even} \times \phi^{\rm odd}$, we have
\begin{displaymath}
\phi(u) = (t,0), \qquad \phi(v)=(0,t),
\end{displaymath}
which shows that the image of $\phi$ contains all pairs $(f,g)$ with $f(0)=g(0)$ (when $2q$ is inverted). Our results on the Burnside rings show that $a$, $b$, and $c$ generate $\Theta(G; \bV)[1/2q]$ as a $\bZ_{2q}$-algebra, and that $uv=0$ holds in this ring. Of course, this implies that $\Theta(G; \bV)[1/2q]$ is generated by $u$, $v$, and $c$. We show that $c$ is generated by $u$ and $v$, which will complete the proof.

Identify $X_{2,1,0}$ with the set of pairs $(v_1,v_2)$ in $\bV^2$ such that $v_1$ and $v_2$ are isotropic and $\langle v_1, v_2 \rangle=1$, and identify $X_{0,1,1}$ with the set of non-zero isotropic vectors in $\bV$. We have a map $X_{2,1,0} \to X_{0,1,1}$ via $(v_1, v_2) \mapsto v_1$. Let $F$ be the fiber of this map over $e_1$. We thus have the relation
\begin{displaymath}
[X_{2,1,0}] = [F] \cdot [X_{0,1,1}].
\end{displaymath}
We now analyze the structure of $F$.

Let $F'$ be the set of vectors $v \in \bV$ such that $(e_1,v) \in F$. Then the map $F \to F'$ given by $(v_1,v_2) \mapsto v_2$ is an isomorphism. We have $v \in F'$ if and only if
\begin{displaymath}
\langle e_1, v \rangle = 1, \qquad \langle v, v \rangle = 0.
\end{displaymath}
Write $v=\alpha e_1 + \beta f_1 + w$ with $w \in \bV(2)$. Then
\begin{displaymath}
\langle e_1, v \rangle = \beta, \qquad \langle v, v \rangle = \alpha \beta + \langle w, w \rangle.
\end{displaymath}
Thus $v$ belongs to $F'$ if and only if $\beta=1$ and $\alpha=-\langle w, w \rangle$. We therefore have an isomorphism
\begin{displaymath}
\bV(2) \to F', \qquad w \mapsto -\langle w, w \rangle e_1 + f_1 + w.
\end{displaymath}
Thus $F \cong \bV(2)$ as well.

Since $\lbb \bV \rbb = q^2 \lbb \bV(2) \rbb$ and $[\bV] =c+1+\tfrac{1}{2} (q-1)(a+b)$, we have
\begin{displaymath}
[X_{2,1,0}] = q^{-2}(c+1+\tfrac{1}{2} (q-1)(a+b)) c.
\end{displaymath}
On the other hand, in $\BB$ we have
\begin{displaymath}
x_{2,1,0} = q^{-2} ((q+1) (x_{0,1,1}+1) - \tfrac{1}{2}(q+1) x_{1,1,0} +\tfrac{1}{2} (q-1) x_{1,\pi,0}) x_{1,1,0},
\end{displaymath}
since the two sides have the same image under the homomorphism $\psi$ from the proof of Proposition~\ref{prop:O-burnside}. Mapping this identity to $\Theta(G; \bV)$ gives another expression for $[X_{2,1,0}]$ in terms of $a$, $b$, and $c$. Combining the two expressions for $[X_{2,1,0}]$ yields
\begin{displaymath}
((q+1)(c+1) - \tfrac{1}{2}(q+1)a +\tfrac{1}{2} (q-1)b)a
= (c+1+\tfrac{1}{2} (q-1)(a+b)) c.
\end{displaymath}
Rewriting this in terms of the generators $u$, $v$, and $c$ leads to the equation
\begin{displaymath}
c=-1+q^{-1} u(u+q-1)+v^2,
\end{displaymath}
which completes the proof.
\end{proof}

Since $\sE(\bV)$ is a large stabilizer class (Proposition~\ref{prop:O-trans}), we find:

\begin{corollary}
The map $\phi$ induces an isomorphism $\Theta(G) \to \bQ[t,t']/(tt')$.
\end{corollary}

Let $k$ be a field of characteristic~0, let $t \in k$, and let $\ast$ be ``even'' or ``odd.'' We define
\begin{displaymath}
\mu_t^{\ast} \colon \Theta(G) \to k
\end{displaymath}
be the measures obtained by composing $\phi^{\ast}$ with the map $\bZ_q[t] \to k$ mapping $t$ to $t$. These measures satisfy
\begin{align*}
\mu_t^{\rm even}(u) &=t & \mu_t^{\rm odd}(u) &=0 \\
\mu_t^{\rm even}(v) &=0 & \mu_t^{\rm odd}(v) &=t
\end{align*}
where $u$ and $v$ are as in the proof of Theorem~\ref{thm:O-meas}, and these identities uniquely characterize the measures. We also have $\mu_t^{\ast}(\bV)=t^2$. Every $k$-valued measure is of the form $\mu^{\ast}_t$ for some choice of $t$ and $\ast$, and these measures are all distinct except for the coincidence $\mu^{\rm even}_0=\mu^{\rm odd}_0$. A simple computation shows that restricting $\mu_t^{\ast}$ to $G(n)$ yields the measure $\mu^{\ast}_{t/q^n}$ after identifying $G(n)$ with $G$. The next proposition thus follows:

\begin{proposition}
Maintain the above notation.
\begin{enumerate}
\item The measure $\mu_t^{\rm even}$ is regular if and only if $t \not\in \{0,1,\pm q^i\}_{i \ge 1}$.
\item The measure $\mu_t^{\rm odd}$ is regular if and only if $t \not\in \{0, \pm q^i\}_{i \ge 0}$.
\item The measure $\mu_t^{\ast}$ is quasi-regular for all $t \ne 0$.
\end{enumerate}
\end{proposition}

\begin{proposition}
Suppose $t \ne 0$. Then $\mu_t^{\ast}$ is equivalent to some ultraproduct measure, and in particular, satisfies (Nil).
\end{proposition}

\begin{proof}
The proof is similar to previous cases, but due to the even/odd subtlety we provide an argument. The sets $\{X_{2r,1}\}_{r \ge 0}$ are a smooth approximation of $G$. For each $r \ge 0$, pick a field $\kappa_r$, let $\cF$ be a non-principal ultrafilter on $\bN$, and let $\kappa$ be the ultraproduct, which we assume to be of characteristic~0. Let $\nu$ be the ultraproduct measure. Since the number of embeddings $V_{a,\epsilon,b} \to V_{2r,1}$ is $P^{\rm even}_{a,\epsilon,b}(q^r)$, it follows that $\nu(x_{a,\epsilon,b})$ is the element of $\kappa$ defined by the sequence $(P^{\rm even}_{a,\epsilon,b}(q^r))_{r \ge 0}$. Thus, letting $\tau=(q^r)_{r \ge 0}$, we find
\begin{displaymath}
\nu(x_{a,\epsilon,b}) = P^{\rm even}_{a,\epsilon,b}(\tau) = \mu^{\rm even}_{\tau}(x_{a,\epsilon,b}).
\end{displaymath}
Since the $x_{a,\epsilon,b}$ generate $\Theta(G; \bV)[1/2q]$, it follows that $\nu=\mu^{\rm even}_{\tau}$. Thus, by choosing the $\kappa$'s and $\cF$ appropriately, for any $t \ne 0$ we can obtain a measure equivalent to $\mu_t^{\rm even}$. The odd case is similar.
\end{proof}

\begin{remark}
As far as we know, Deligne was the first to observe the phenomenon that there are two 1-parameter families of measures in this case \cite{DeligneLetter}.
\end{remark}

\subsection{Tensor categories}

Fix an algebraically closed field $k$ of characteristic~0, a non-zero element $t \in k$, and let $\ast$ be ``even'' or ``odd.'' We let $\uPerm_k(\bO^{\ast}_t)$ be the category $\uPerm_k(G, \mu_t^{\ast})$, and let $\uRep_k(\bO^{\ast}_t)$ be the category $\uRep(G, \mu_t^{\ast})$. We omit the $k$ when possible.

\begin{theorem}
For $t \ne 0$, the category $\uRep(\bO_t^{\ast})$ is pre-Tannakian, and the abelian envelope of $\uPerm(\bO_t^{\ast})$. If $t$ is a regular parameter then $\uRep(\bO_t^{\ast})$ is semi-simple.
\end{theorem}

\begin{proof}
do this
\end{proof}

Suppose now that $k$ is the ultraproduct of fields $\{k_r\}_{r \ge 0}$, with respect to some non-principal ultrafilter on $\bN$, and let $t=(q^r)_{r \ge 0}$. Let
\begin{displaymath}
H_r^{\rm even} = \Aut(V_{2r,1}) = \bO_{r,r}(\bF), \qquad
H_r^{\rm odd} = \Aut(V_{2r+1,1}) = \bO_{r+1,r}(\bF),
\end{displaymath}
be the finite Chevalley groups of type $D_r$ and $B_r$. Let $\cT^{\ast}_r=\Rep_{k_r}(H_r^{\ast})$, let $\tilde{\cT}^{\ast}$ be the ultraproduct of the $\cT^{\ast}_r$, and let $\cT^{\ast}$ be the subcategory defined in \S \ref{ss:ultra}. Let
\begin{displaymath}
\Phi \colon \uRep_k(\bO_t^{\ast}) \to \cT^{\ast}
\end{displaymath}
be the tensor functor defined in \S \ref{ss:ultra-ab}. Then, as in other cases, we find:

\begin{theorem}
The above functor $\Phi$ is an equivalence.
\end{theorem}

\section{The unitary group} \label{s:U}

\subsection{Set-up}

Let $\bF$ be a finite field of odd cardinality $q$, let $\bK$ be the quadratic extension of $\bF$, and let $\tau$ be the non-trivial Galois automorphism of $\bK/\bF$. As usual, $q$ is the cardinality of $\bF$, and so $q^2$ is the cardinality of $\bK$. Define a category $\hat{\fC}$ as follows. An object is a $\bK$-vector space $V$ equipped with a symmetric sesquilinear form
\begin{displaymath}
\langle, \rangle \colon V \times V \to \bK.
\end{displaymath}
This means that $\langle, \rangle$ is biadditive, conjugate $\bK$-linear in the first variable, $\bK$-linear in the second variable, and satisfies $\tau(\langle x, y \rangle)=\langle y, x \rangle$. A morphism $V \to W$ is a injective $\bK$-linear map that respects the forms. We let $\fC$ be the full subcategory spanned by finite dimensional spaces.

Let $\bV$ be the $\bK$-vector space with basis $\{e_i\}_{i \ge 1}$, equipped with the symmetric sesquilinear form satisfying $\langle e_i, e_j \rangle=\delta_{i,j}$. We regard $\bV$ as an object of $\hat{\fC}$, and let $G$ be its automorphism group; this is one version of the infinite unitary group over $\bF$. The action of $G$ on $\bV$ is oligomorphic, and $\fC$ is the category of structures for $(G, \bV)$ as defined in \S \ref{ss:struct}.

Define $V_m$ to be the subspace of $\bV$ spanned by $e_1, \ldots, e_m$ with the induced form, and define $V_{m,n}$ to be the orthogonal direct sum of $V_m$ with an $n$-dimensional space carrying the zero form. Let $X_{m,n}$ be the space of embeddings $V_{m,n} \to \bV$. One easily sees that every object of $\fC$ is isomorphic to a unique $V_{m,n}$, and so the $X_{m,n}$ are exactly the transitive $\bV$-smooth $G$-sets by Proposition~\ref{prop:struct-equiv}.

\subsection{Enumeration}

Let $V$ be an object of $\fC$. If $x \in V$ then $\langle x, x \rangle$ is invariant under $\tau$, and thus belongs to $\bF$. For $a \in \bF$, we let $N_a(V)$ denote the number of vectors $x \in V$ such that $\langle x, x \rangle = a$.

\begin{proposition} \label{prop:U-norm}
For $a \in \bF$ non-zero, we have
\begin{align*}
N_a(V_m) &= q^{2m-1}+(-q)^{m-1} \\
N_0(V_m) &= q^{2m-1}+(-q)^m+(-q)^{m-1}
\end{align*}
\end{proposition}

\begin{proof}
If $b \in \bK^*$ then $\langle bx, bx \rangle = \rN(b) \langle x, x \rangle$, where $\rN \colon \bK^* \to \bF^*$ is the norm map. Since the norm is surjective, it follows that $N_a(V_m)$ is independent of $a$, for $a$ non-zero. Since every element of $V_m$ has some norm in $\bF$, we thus find
\begin{displaymath}
q^{2m} = \sum_{a \in \bF} N_a(V_m) = N_0(V_m) + (q-1) N_1(V_m).
\end{displaymath}
From the decomposition $V_m=V_1 \oplus V_{m-1}$, we find
\begin{displaymath}
N_0(V_m) = \sum_{a \in \bF} N_a(V_1) N_{-a}(V_{m-1}) = N_0(V_{m-1})+(q-1)(q+1)N_1(V_{m-1}),
\end{displaymath}
where here we used $N_0(V_1)=1$ and $N_1(V_1)=q+1$. Solving this system of equations gives the stated result.
\end{proof}

Let $N_{a,b}^{m,n}$ be the number of embeddings $V_{a,b} \to V_{m,n}$. We put $N_{a,b}^{m,n}=0$ if $m$ or $n$ is negative.

\begin{proposition}
We have
\begin{align*}
N_{a,b}^{m,n} &= (q^{2m-1}+(-q)^{m-1}) q^{2n} N_{a-1,b}^{m-1,n} \\
N_{a,b}^{m,n} &= -((-q)^m-1)((-q)^{m-1}-1) q^{2(n+a+b-1)} N_{a,b-1}^{m-2,n} + (q^{2n}-1) q^{2(a+b-1)} N_{a,b-1}^{m,n-1}
\end{align*}
The first equation is valid for $a \ge 1$, and the second for $b \ge 1$.
\end{proposition}

\begin{proof}
(a) Giving an embedding of $V_{a,b}$ into $V_{m,n}$ is equivalent to giving an embedding $i$ of $V_{1,0}$ together with an embedding of $V_{a-1,b}$ into the orthogonal complement of the image of $i$, which is isomorphic to $V_{m-1,n}$. We thus find
\begin{displaymath}
N_{a,b}^{m,n} = N_{1,0}^{m,n} \cdot N_{a-1,b}^{m-1,n}.
\end{displaymath}
Since $N_{1,0}^{m,n}=N_1(V_{m,n})=q^{2n} N_1(V_m)$, the result follows.

(b) Write $V_{a,b} = V_{a,b-1} \oplus V_{0,1}$, and let $x$ be a basis of $V_{0,1}$. Given en embedding $i$ of $V_{a,b}$ into $V_{m,n}$, let $y$ be the image of $x$ and let $j$ be the restriction of $i$ to $V_{a,b-1}$. Then $i$ is determined by $(y,j)$, and such a pair comes from an $i$ if and only if $y$ is non-zero isotropic and the image of $j$ is orthogonal to and linearly independent from $y$. We count the pairs $(y,j)$.

First consider the case where $y$ belongs to the nullspace of $V_{m,n}$. There are $q^{2n}-1$ choices for such $y$. A pair $(y,j)$ comes from an $i$ if and only if the induced map $\ol{j} \colon V_{a,b-1} \to V_{m,n}/\bK y$ is an embedding in $\fC$. The quotient is isomorphic to $V_{m,n-1}$, and so there are $N_{a,b-1}^{m,n-1}$ choices for $\ol{j}$. Each $\ol{j}$ admits exactly $q^{2(a+b-1)}$ lifts to a $j$. Thus the total number of $i$'s in this case is the second term in the formula.

Now suppose that $y$ does not belong to the nullspace. There are $(N_0(V_m)-1)q^{2n}$ choices for such a $y$. Let $V'$ be the orthogonal complement of $y$. A pair $(y,j)$ comes from an $i$ if and only if $j$ maps into $V'$ and the induced map $\ol{j} \colon V_{a,b-1} \to V'/\bK y$ is an embedding. The quotient space is always isomorphic to $V_{n-2,m}$, and so there are $N_{a,b-1}^{m-2,n}$ choices for $\ol{j}$. Each $\ol{j}$ admits $q^{2(a+b-1)}$ lifts to a $j$, and so the total number of $i$'s in this case is the first term of the stated formula.
\end{proof}

Define polynomials $Q_{a,b}(t,u) \in \bZ_q[t,u]$ by $Q_{0,0}=1$ and
\begin{align*}
Q_{a,b}(t,u) =& q^{-1} (t-1) t u^2 Q_{a-1,b}(-q^{-1}t, u) \\
Q_{a,b}(t,u) =& q^{2(a+b)-3} (t-1)(t+q) u^2 Q_{a,b-1}(q^{-2} t, u) + (u^2-1)q^{2(a+b-1)} Q_{a,b-1}(t, q^{-1} u).
\end{align*}
As usual, we have:

\begin{proposition}
We have $N_{a,b}^{m,n}=Q_{a,b}((-q)^m,(-q)^n)$ for all $a, b, m, n \in \bN$.
\end{proposition}

Note that $Q_{a,b}(t,u)$ is actually a polynomial in $u^2$, so the sign of the second argument does not really matter. Put $P_{a,b}(t)=Q_{a,b}(t,1)$. We compute these polynomials explicitly, by simply solving the above recurrences:

\begin{proposition}
We have
\begin{displaymath}
P_{a,b}(t) = q^{-(a+b)^2} t^a \prod_{i=0}^{a+2b-1}(t-(-q)^i).
\end{displaymath}
\end{proposition}

\subsection{Burnside rings}

Let $\BB=\BB(G, \bV)$ be the Burnside ring for $G$ relative to $\sE(\bV)$. Let $x_{a,b}$ be the class $\lbb X_{a,b} \rbb$; these elements form a $\bZ$-basis for $\BB$. We determine the ring structure:

\begin{proposition} \label{prop:U-burnside}
We have a ring isomorphism
\begin{displaymath}
\phi \colon \bZ_q[X,Y] \to \BB[1/q], \qquad
X \mapsto x_{1,0}, \quad Y \mapsto x_{0,1}
\end{displaymath}
\end{proposition}

\begin{proof}
We have a filtration on $\BB$ by taking $F^n \BB$ to be the span of the classes $x_{a,b}$ with $a+b \le n$. Defining $F^n \bZ[X,Y]$ to be the span of the monomial $X^i Y^j$ with $i+j \le n$, we see that $\phi$ is a map of filtered rings. The graded pieces of the two rings are free $\bZ$-modules of the same rank, and so it suffices to verify that $\phi$ is injective modulo $\ell$, for all $\ell \nmid q$. For this, we use the ring homomorphism.
\begin{displaymath}
\psi \colon \BB \to \bZ_q[t,u], \qquad x_{a,b} \mapsto Q_{a,b}.
\end{displaymath}
Since the polynomials
\begin{displaymath}
\psi(x_{1,0})=q^{-1}(t-1)tu^2, \qquad
\psi(x_{0,1})=-q^{-1}(t-1)(t+q)u^2+(u^2-1)
\end{displaymath}
are algebraically independent modulo $\ell$, the result follows.
\end{proof}

Let $\hat{\BB}$ be the infinitesimal Burnside ring of $G$ relative to $\sE(\bV)$. We now analyze it. Let $G(n) \subset G$ be the subgroup fixing each $e_i$ for $1 \le i \le n$, and let $\bV(n) \subset \bV$ be the span of the $e_i$ with $i>n$. Then $(G(n), \bV(n))$ is isomorphic to $(G, \bV)$. We define a ring homomorphism
\begin{displaymath}
\delta \colon \BB \to \BB
\end{displaymath}
by restricting to $G(1)$ and identifying with $G$. We compute this map.

\begin{proposition}
We have
\begin{align*}
\delta(x_{1,0}) &= (q+1)+(q^2-q-1)x_{1,0}+(q+1) x_{0,1} \\
\delta(x_{0,1}) &= (q^2-1) x_{1,0} + x_{0,1}.
\end{align*}
In particular, $\delta \colon \BB[1/q] \to \BB[1/q]$ is an isomorphism.
\end{proposition}

\begin{proof}
For $V \in \hat{\fC}$ and $a \in \bF$, let $Y_a(V)$ be the set of elements $x \in V$ such that $\langle x, x \rangle=a$. As in the proof of Proposition~\ref{prop:U-norm}, we have an isomorphism $Y_a(V) \cong Y_1(V)$ if $a \ne 0$. We have
\begin{displaymath}
\lbb Y_0(\bV) \rbb = 1+x_{0,1}, \qquad
\lbb Y_1(\bV) \rbb = x_{1,0}, \qquad
\bV \cong Y_0(\bV) \amalg Y_1(\bV)^{\amalg (q-1)}.
\end{displaymath}
Since $\bV = \bK e_1 \oplus \bV(1)$, we have $\delta(\lbb \bV \rbb)=q^2 \lbb \bV \rbb$, and so
\begin{displaymath}
\delta(1+x_{0,1}+(q-1)x_{1,0}) = q^2(1+x_{0,1}+(q-1)x_{1,0}).
\end{displaymath}
Now, we have
\begin{displaymath}
Y_0(\bV) \cong \coprod_{a \in \bF} \big( Y_a(\bK e_1) \times Y_{-a}(\bV(1))\big),
\end{displaymath}
which gives
\begin{displaymath}
Y_0(\bV) \cong Y_0(\bV(1)) \amalg Y_1(\bV(1))^{\amalg (q^2-1)}.
\end{displaymath}
This is again similar to the reasoning used in Proposition~\ref{prop:U-norm}. We thus find
\begin{displaymath}
\delta(1+x_{0,1}) = 1+x_{0,1}+(q^2-1)x_{1,0}.
\end{displaymath}
Combining this and the previous equation gives the stated formulas for $\delta$. It follows easily that $\delta$ is an isomorphism (see Remark~\ref{rmk:U-grading} for details).
\end{proof}

As usual, $\hat{\BB}$ is the direct limit of iterates of $\delta$, and so:

\begin{corollary}
The natural map $\BB[1/q] \to \hat{\BB}[1/q]$ is an isomorphism.
\end{corollary}

\begin{remark} \label{rmk:U-grading}
Putting
\begin{displaymath}
w_1=1+x_{0,1}-x_{1,0}, \qquad
w_2=\lbb \bV \rbb=1+x_{0,1}+(q-1)x_{1,0}
\end{displaymath}
we have
\begin{displaymath}
\delta(w_1)=-qw_1, \qquad \delta(w_2)=q^2 w_2.
\end{displaymath}
Since
\begin{displaymath}
qx_{1,0} = w_2-w_1, \qquad qx_{0,1}=-q+(q-1)w_1+w_2,
\end{displaymath}
we see that $w_1$ and $w_2$ generated $\BB[1/q]$. It follows that the action of $\delta$ on $\BB[1/q]$ is diagonalizable, and induces a grading; the degree $n$ piece is the $q^{2n}$ eigenspace of $\delta^2$.
\end{remark}

\subsection{Measures}

Let $N^{*,0}_{a,b}$ be the function on $\bN$ given by $m \mapsto N^{m,0}_{a,b}$. As in other cases, $\{V_{m,0}\}_{m \ge 0}$ is a smooth approximation of $G$. Thus, by \S \ref{ss:count}, we obtain a counting measure
\begin{displaymath}
\Theta(G, \bV) \to \Fun^{\circ}(\bN, \bZ), \qquad x_{a,b} \mapsto N^{*,0}_{a,b}.
\end{displaymath}
Since $N^{m,0}_{a,b}=P_{a,b}((-q)^m)$, it follows that we have a measure
\begin{displaymath}
\phi \colon \Theta(G, \bV) \to \bZ_q[t], \qquad \phi(x_{a,b})=P_{a,b}(t).
\end{displaymath}
This measure is universal:

\begin{theorem}
The map $\phi$ is an isomorphism after inverting $q$.
\end{theorem}

\begin{proof}
Let $a=[X_{1,0}]$ and $b=[X_{0,2}]$, and let $c=1+b-a$; these are each elements of $\Theta(G, \bV)$. We have $\phi(a)=P_{1,0}(t)$ and $\phi(b)=P_{0,1}(t)$, and so $\phi(c)=t$, which shows that $\phi$ is surjective. Since $x_{1,0}$ and $x_{0,1}$ generate $\hat{\BB}[1/q]$, it follows that $a$ and $b$ generate $\Theta(G, \bV)[1/q]$. We show that $c$ generates, which will complete the proof.

Identify $X_{2,0}$ with the set of orthonormal pairs $(v_1,v_2)$ in $\bV^2$, and identify $X_{0,1}$ with the set of non-zero isotropic vectors in $\bV$. Let $\epsilon \in \bK$ be an element of norm $-1$. We then have a $G$-map $X_{2,0} \to X_{0,1}$ via $(v_1,v_2) \mapsto v_1+\epsilon v_2$. Let $F$ be the fiber of this map over $e_1+\epsilon e_2$. We thus have $[X_{2,0}]=[F] [X_{0,1}]$ in $\Theta(G; \bV)$. We now analyze the structure of $F$.

Suppose $(v_1,v_2) \in F$. Then $v_1+\epsilon v_2=e_1+\epsilon e_2$, and so we can solve for $v_2$ in terms of $v_1$. Let $F'$ be the of vectors $v \in \bV$ such that
\begin{displaymath}
\langle v,v \rangle = 1, \qquad \langle v, e_1+\epsilon e_2 \rangle = 1.
\end{displaymath}
Then one finds that the map $F \to F'$ given by $(v_1,v_2) \mapsto v_1$ is an isomorphism. Now, given $v \in \bV$, write $v=\alpha e_1+\beta e_2+w$, where $\alpha,\beta \in \bK$ and $w \in \bV(2)$. Then one sees that $v$ belongs to $F'$ if and only if
\begin{displaymath}
\alpha \ol{\alpha} + \beta \ol{\beta} + \langle w, w \rangle = 1, \qquad
\alpha + \ol{\epsilon} \beta = 1,
\end{displaymath}
where the bar denotes the non-trivial Galois automorphism $\tau$. Thus, if $v \in F'$, then we can solve for $\beta$ in terms of $\alpha$. We thus see that $F'$ is isomorphic to the set $F''$ of pairs $(\alpha, w) \in \bK \times \bV(2)$ such that
\begin{displaymath}
\langle w, w \rangle = 2-(\alpha+\ol{\alpha}).
\end{displaymath}
Now, the function $\bK \to \bF$ given by $\alpha \mapsto 2-(\alpha+\ol{\alpha})$ is surjective and everywhere $q$-to-1. We thus see that $F''$ is isomorphic to $q$ copies of $\bV(2)$.

The above analysis shows that $[F]=q[\bV(2)]$. Since $[\bV]=q^4[\bV(2)]$, we have
\begin{displaymath}
[X_{2,0}]=q^{-3} [\bV] [X_{0,1}].
\end{displaymath}
On the other hand, we have
\begin{displaymath}
x_{2,0}=q^{-3}((q+1)(x_{0,1}+1)-x_{1,0}) x_{1,0}
\end{displaymath}
in $\BB[1/q]$ since the two sides agree after applying the injective homomorphism $\psi$ from the proof of Proposition~\ref{prop:U-burnside}. This gives an analogous identity in $\Theta$. Combining the above two equations and using the identity $[\bV]=1+b+(q-1)a$, we obtain the identity
\begin{displaymath}
(1+b+(q-1)a) b = ((q+1)(b+1)-a)a,
\end{displaymath}
which simplifies to
\begin{displaymath}
(a-b)^2=(q+1)a-b.
\end{displaymath}
A simple algebraic manipulation now shows
\begin{displaymath}
a=q^{-1}c(c-1), \qquad b = q^{-1}(c-1)(c+q),
\end{displaymath}
which completes the proof.
\end{proof}

As in other cases, the stabilizer class $\sE(\bV)$ is large, and so we find:

\begin{corollary}
The map $\phi$ induces an isomorphism $\Theta(G) \otimes \bQ \to \bQ[t]$.
\end{corollary}

Given a field $k$ of characteristic~0 and $t \in k$, we let $\mu_t$ be the measure for $G$ obtained from the homomorphism $\bQ[t] \to k$ mapping $t$ to $t$. This is the unique measure satisfying
\begin{displaymath}
\mu_t(X_{0,1})-\mu_t(X_{0,1})=t-1.
\end{displaymath}
The restriction of $\mu_t$ to $G(1)$ coincides with $\mu_{-t/q}$ when the latter is identified with $G$. The following two propositions are proved similarly to the other cases.

\begin{proposition}
The measure $\mu_t$ is regular if and only if $t \not\in \{0,(-q)^i\}_{i \ge 0}$, and quasi-regular if and only if $t \ne 0$.
\end{proposition}

\begin{proposition}
If $t \ne 0$ then $\mu_t$ is equivalent to some ultraproduct measure; in particular, it satisfies (Nil).
\end{proposition}

\subsection{Tensor categories}

Let $k$ be a field of characteristic~0 and let $t \in k$ be non-zero. We say $t$ is \defn{regular} if $\mu_t$ is. We write $\uPerm_k(\bU_t)$ for $\uPerm_k(G, \mu_t)$, and similarly $\uRep_k(\bU_t)$. We now analyze these categories.

\begin{theorem}
The category $\uRep(\bU_t)$ is pre-Tannakian, and the abelian envelope of $\uPerm(\bU_t)$. If $t$ is a regular parameter then $\uRep(\bU_t)$ is semi-simple.
\end{theorem}

Let $k_n$ be an algebraically closed field for each $n \in \bN$, fix a non-principal ultrafilter on $\bN$, and let $k$ be the ultraproduct, which we assume to be of characteristic~0. Let $t$ be the element $((-q)^n)_{n \ge 0}$. Let $H_n=\bU_n(\bF)$ be the finite unitary group, let $\cT_n=\Rep_{k_n}(H_n)$, let $\tilde{\cT}$ be the ultraproduct of the $\cT_n$'s, and let $\cT$ be the subcategory of $\tilde{\cT}$ defined in \S \ref{ss:ultra}. Let
\begin{displaymath}
\Phi \colon \uRep(\bU_t) \to \cT
\end{displaymath}
be the tensor functor defined in \S \ref{ss:ultra-ab}. Then, as in other cases, we obtain:

\begin{theorem}
The above functor $\Phi$ is an equivalence.
\end{theorem}

\addtocontents{toc}{\medskip}


\begin{thebibliography}{DMNO}
%
%
%
%
%

\bibitem[CEAH]{CEAH} Kevin Coulembier, Inna Entova-Aizenbud, Thorsten Heidersdorf. Monoidal abelian envelopes and a conjecture of Benson--Etingof. \arxiv{1911.04303}

%

\bibitem[CH]{CherlinHrushovski} G.\ Cherlin, E.\ Hrushovski. Finite structures with few types. Princeton, USA: Princeton University Press, 2003. Annals of Mathematics Studies vol.\ 152.

\bibitem[CO1]{ComesOstrik1} Jonathan Comes, Victor Ostrik. On blocks of Deligne's category $\uRep(S_t)$. \textit{Adv.\ Math.} \textbf{226} (2011), no.~2, pp.~1331--1377. \DOI{10.1016/j.aim.2010.08.010} \arxiv{0910.5695}

\bibitem[CO2]{ComesOstrik} Jonathan Comes, Victor Ostrik. On Deligne's category $\uRep^{ab}(S_d)$. \textit{Algebra Number Theory} \textbf{8} (2014), pp.~473--496. \DOI{10.2140/ant.2014.8.473} \arxiv{1304.3491}
%
%
%

\bibitem[Del1]{Deligne} P. Deligne. La cat\'egorie des repr\'esentations du groupe sym\'etrique $S_t$, lorsque $t$ n’est pas un entier naturel. In: Algebraic Groups and Homogeneous Spaces, in: Tata Inst. Fund. Res. Stud. Math., Tata Inst. Fund. Res., Mumbai, 2007, pp.~209--273. \\
Available at: {\tiny\url{https://www.math.ias.edu/files/deligne/Symetrique.pdf}}

\bibitem[Del2]{DeligneLetter} P. Deligne. Letter to A.~Snowden. October 24, 2023.

%
%
%
%
\bibitem[EAH]{EntovaAizenbudHeidersdorf} Inna Entova-Aizenbud, Thorsten Heidersdorf. Deligne categories for the finite general linear groups, part 1: universal property. \arxiv{2208.00241}
%
%
%
%

\bibitem[ES]{delalg} Pavel Etingof, Andrew Snowden. Classification of simple commutative algebras in the Delannoy category. In preparation.


\bibitem[FH]{FarahatHigman} H.~K.~Farahat, G.~Higman. The centres of symmetric group rings. \textit{Proc.\ R.\ Soc.\ Lond.\ Ser.~A Math.\ Phys.\ Eng.\ Sci.} \textbf{250} (1959), no.~1261, pp.~212--221. \DOI{10.1098/rspa.1959.0060}


\bibitem[GH]{GH} Shamgar Gurevich, Roger Howe. Harmonic analysis on $\GL(n)$ over finite fields. \textit{Pure Appl.\ Math.\ Q.} \textbf{17} (2021) no.~4, pp.~1387--1463. \DOI{10.4310/PAMQ.2021.v17.n4.a7} \arxiv{2105.12369}

\bibitem[GLT]{GLT} Robert M. Guralnick, Michael Larsen, Pham Huu Tiep. Character levels and character bounds. \textit{Forum Math.\ Pi} \textbf{8} (2020). \DOI{10.1017/fmp.2019.9} \arxiv{1708.03844}

\bibitem[Har1]{Harman1} Nate Harman. Stability and periodicity in the modular representation theory of symmetric groups. \arxiv{1509.06414v3}

\bibitem[Har2]{Harman2} Nate Harman. Deligne categories as limits in rank and characteristic. \arxiv{1601.03426}

\bibitem[HS1]{repst} Nate Harman, Andrew Snowden. Oligomorphic groups and tensor categories. \arxiv{2204.04526}

\bibitem[HS2]{discrete} Nate Harman, Andrew Snowden. Discrete pre-Tannakian categories. \arxiv{2304.05375}


\bibitem[HNS]{arboreal} Nate Harman, Ilia Nekrasov, Andrew Snowden, Noah Snyder. Arboreal tensor categories. \arxiv{2308.06660}

%

\bibitem[Kno1]{Knop} Friedrich Knop. A construction of semisimple tensor categories. \textit{C.~R.~Math.\ Acad.\ Sci.\ Paris C} \textbf{343} (2006), no.~1, pp.~15--18. \DOI{10.1016/j.crma.2006.05.009} \arxiv{math/0605126}

\bibitem[Kno2]{Knop2} Friedrich Knop. Tensor envelopes of regular categories. \textit{Adv.\ Math.} \textbf{214} (2007), pp.~571--617. \DOI{10.1016/j.aim.2007.03.001} \arxiv{math/0610552}

\bibitem[Kri1]{Kriz} Sophie Kriz. Oscillator representations and semisimple pre-Tannakian categories. Available at: {\tiny\url{https://krizsophie.github.io/}}

\bibitem[Kri2]{Kriz2} Sophie Kriz. Quantum Delannoy categories. Available at: {\tiny\url{https://krizsophie.github.io/}}

\bibitem[KR]{KannanRyba} Arun S. Kannan, Christopher Ryba. Stable Centres II: Finite Classical Groups. \arxiv{2112.01467}

\bibitem[Lax]{Laxton} R.~R.~Laxton. On a Problem of M.~Ward. \textit{Fibonacci Quart.} \textbf{12} (1974), no.~1, pp.~41--44. \DOI{10.1080/00150517.1974.12430767}



%
%

\bibitem[NS]{distal} Ilia Nekrasov, Andrew Snowden. Upper bounds for measures on distal classes. \arxiv{2407.19131}

\bibitem[Pol]{Polya} G. P\'olya. Arithmetische Eigenschaften der Reihenentwicklungen rationaler Funktionen. \textit{J.\ Reine Angew.\ Math.} \textbf{151} (1921), pp.~1--31. \DOI{10.1515/crll.1921.151.1}

\bibitem[Ple]{Pless} Vera Pless. The number of isotropic subspaces in a finite geometry. \textit{Atti.\ Accad.\ Naz.\ Lincei Rendic.} \textbf{39} (1965), no.~6, pp.~418--421. Available at: {\tiny\url{http://www.bdim.eu/item?id=RLINA_1965_8_39_6_418_0}}

\bibitem[Ryb]{Ryba} Christopher Ryba. Stable Centres I: Wreath Products. \arxiv{2107.03752}

\bibitem[Sol]{Solomon} Louis Solomon. The Burnside algebra of a finite group. \textit{J.\ Combin. Theory} \textbf{2} (1967), pp.~603--615. \DOI{10.1016/S0021-9800(67)80064-4}

\bibitem[Sn1]{cantor} Andrew Snowden. On the representation theory of the symmetry group of the Cantor set. \arxiv{2308.06648}

\bibitem[Sn2]{regcat} Andrew Snowden. Regular categories, oligomorphic monoids, and tensor categories. \arxiv{2403.16267}

\bibitem[Sn3]{azumaya} Andrew Snowden. Interpolation of the oscillator representation and Azumaya algebras in tensor categories. \arxiv{2408.00233}

\bibitem[Ste]{Steinberg} Benjamin Steinberg. Property of simplicity and semi-simplicity under base change of base field. {\tiny\url{https://mathoverflow.net/q/439913}} (Accessed June 15, 2025.)

\bibitem[Stacks]{stacks} Stacks Project. {\tiny\url{http://stacks.math.columbia.edu}} (accessed 2025).

\bibitem[WW]{WanWang} Jinkui Wan, Weiqiang Wang. Stability of the centers of group algebras of $\GL_n(q)$. \textit{Adv.\ Math} \textbf{349} (2019), pp.~749--780. \DOI{10.1016/j.aim.2019.04.026} \arxiv{1805.08796}


\bibitem[Yoo]{Yoo} Semin Yoo. Combinatorics of Euclidean spaces over finite fields. \textit{Ann.\ Comb.} \textbf{28} (2024), pp.~283--327. \DOI{10.1007/s00026-023-00661-3} \arxiv{1910.03482}

\end{thebibliography}
\end{document}